\documentclass[svgnames]{amsart}
\usepackage[english]{babel}
\usepackage[utf8]{inputenc}

\usepackage{svg}
\usepackage{hyperref}
\hypersetup{colorlinks,citecolor=Black,linkcolor=Black,urlcolor=Blue}

\usepackage{tabularx, ifthen}
\usepackage{amssymb, amsfonts, amsmath, amsthm, amscd, amsxtra, latexsym, mathabx}
\usepackage{epigraph}
\usepackage{enumitem}

\usepackage{tikz}\usetikzlibrary{cd}
\theoremstyle{plain}
\newtheorem{thm}{Theorem}[section] 
\newtheorem{lemma}[thm]{Lemma}
\newtheorem{cor}[thm]{Corollary}
\newtheorem{prop}[thm]{Proposition}
\newtheorem{thmintro}{Theorem}

\newtheorem{corintro}[thmintro]{Corollary}

\newtheorem{claim}[thm]{Claim}
\newtheorem{conjecture}[thmintro]{Conjecture}

\theoremstyle{definition}
\newtheorem{defn}[thm]{Definition}
\newtheorem{rem}[thm]{Remark}

\newtheorem{questintro}[thmintro]{Question}
\newtheorem{notation}[thm]{Notation}
\newtheorem{ax}{Axiom}

\newenvironment{claimproof}{\begin{proof}}{\end{proof}}

\newcommand{\Z}{\mathbb{Z}}
\newcommand{\R}{\mathbb{R}}
\newcommand{\Cay}[2]{\operatorname{Cay}\left(#1,#2\right)}

\newcommand{\frakS}{\mathfrak S}
\newcommand{\diam}{\mathrm{diam}}
\newcommand{\C}{\mathcal C}
\newcommand{\ov}[1]{\overline{#1}}
\newcommand{\U}{\mathcal{U}}

\newcommand{\dist}{\mathrm{d}}
\newcommand{\nest}{\sqsubseteq}
\newcommand{\propnest}{\sqsubsetneq}
\newcommand{\orth}{\bot}
\newcommand{\transverse}{\pitchfork}
\newcommand{\link}{\operatorname{Lk}}
\newcommand{\neb}{\mathcal N}
\newcommand{\cuco}[1]{{\mathcal #1}}
\newcommand{\Stab}[2]{\operatorname{Stab}_{#1}(#2)}

\newcommand{\duaug}[2]{{#1}^{+{#2}}}
\newcommand{\Sat}{\mathrm{Sat}}

\newcommand{\Proj}{\operatorname{Proj}}

\newcommand{\W}{\mathcal{W}}

\newcommand{\gate}{\mathfrak{g}}
\newcommand{\p}{\mathfrak{p}}

\newcommand{\h}{\widehat}

\newcommand{\MCG}{\mathcal{MCG}^\pm}
\newcommand{\Aut}[1]{\operatorname{Aut}\left(#1\right)}
\newcommand{\Id}{\text{Id}}
\newcommand{\Squid}{\operatorname{Cone}}

\newcommand{\tsh}[1]{\left\{\kern-.7ex\left\{#1\right\}\kern-.7ex\right\}}

\title[Short HHG I]{Short hierarchically hyperbolic groups I:\\
uncountably many coarse median structures}

\author[G. Mangioni]{Giorgio Mangioni}
    \address{(Giorgio Mangioni) Maxwell Institute and Department of Mathematics,\\ Heriot-Watt University,\\ Edinburgh, UK\\ Orcid: 0000-0003-2868-5032}
    \email{gm2070@hw.ac.uk}

\begin{document}

\begin{abstract}
We prove that the mapping class group of a sphere with five punctures admits uncountably many coarsely equivariant coarse median structures. The same is shown for right-angled Artin groups whose defining graphs are connected, triangle- and square-free, and have at least three vertices. Remarkably, in the latter case, the coarse median structures we produce are not induced by cocompact cubulations. 

To obtain the above results, we develop the theory of short hierarchically hyperbolic groups (HHG), which also include Artin groups of large and hyperbolic type, graph manifold groups, and extensions of Veech groups. We develop tools to modify their hierarchical structure, including using quasimorphisms to construct quasilines that serve as coordinate spaces, and this is where the abundance of coarse median structures comes from. These techniques are of independent interest, and are used in a follow-up paper with Alessandro Sisto to study quotients of short HHG. 

In the process, we also clarify a proof of Hagen, Martin, and Sisto on hierarchical hyperbolicity of Artin groups of large and hyperbolic type. 
\end{abstract}

\maketitle

\epigraph{And though she be but little,\\she is fierce.}{William Shakespeare, \textit{Helena}}

\setcounter{tocdepth}{1}
\tableofcontents
\section*{Introduction}
Hierarchically hyperbolic spaces (HHS for short) and groups (HHG), introduced by Behrstock, Hagen, and Sisto in \cite{HHS_I}, provide a common framework for the study of mapping class groups of finite-type surfaces, several Coxeter and Artin groups, most CAT(0) cubical groups, and many others. Essentially, a group $G$ is hierarchically hyperbolic if there exists a family $\{\C U\}$ of hyperbolic “coordinate spaces”, together with $G$-equivariant projections $\pi_U\colon G\to \C U$. The coordinate spaces are also arranged in a partial ordering with a unique maximal element (indeed, a hierarchy), and the number of layers of the hierarchy measures how far the group is from being hyperbolic (see Section \ref{sec:HHG/S} for background on HHS/G).

\subsection*{Short HHG}
This paper is the first in a series of two, in which we introduce the family of \emph{short HHG} (see Section~\ref{defn:short_HHG}). Roughly, a hierarchically hyperbolic group $G$ is short if:
\begin{itemize}
    \item $G$ acts on a triangle- and square-free simplicial graph $\ov X$, which we call the \emph{support graph}, with finitely many orbits of edges;
    \item each vertex stabiliser $\Stab{G}{v}$ for the action on $\ov X$ has a preferred cyclic normal subgroup $Z_v$, called the \emph{cyclic direction}, and the quotient $\Stab{G}{v}/Z_v$ is hyperbolic;
    \item The HHG structure “witnesses” the cyclic directions, in the sense that for every vertex $v\in\ov X$ there exists a coordinate space $\C\ell_v$ on which $Z_v$ acts geometrically. 
\end{itemize} The name \emph{short} comes from the fact that the hierarchy only has three layers, with the quasilines $\{\C\ell_v\}_{v\in\ov X}$ at the bottom (see Remark~\ref{rem:unbounded_dom_short_hhg}). Among others, this class of groups includes:
\begin{enumerate}[label=(\alph*)]
    \item \label{item:listfirst} The mapping class group of a five-punctured sphere (whose short structure is explained in Subsection~\ref{subsec:mcg});
    \item Right-angled Artin groups whose defining graph is connected, triangle- and square-free  (Proposition~\ref{prop:raag_short});
    \item Artin groups of large and hyperbolic type (Subsection~\ref{subsec:example_artin});
    \item Fundamental groups of non-geometric graph manifolds, and more generally of admissible graph of groups in the sense of \cite{croke-Kleiner} (Subsection~\ref{subsec:example_pi1});
    \item Certain extensions of Veech subgroups of mapping class groups, namely those considered in \cite{veech,bongiovanni2024extensions} (Subsection~\ref{subsec:example_veech});
    \item \label{item:listlast}  Relative hyperbolic groups whose peripherals are $\Z$-central extensions of hyperbolic groups (Proposition~\ref{prop:hyp_rel_Z-central_is_short}).
\end{enumerate}
Unlike most known examples of HHG, the hierarchical structure of short HHG can be easily modified, as we shall detail later, and this makes them the perfect playground where one can test to which extent a certain property depends on the choice of the hierarchical structure.

\subsection*{Application to coarse median structures}
To illustrate this principle, recall that, for example, by \cite[Theorem 7.3]{HHS_II}, a hierarchically hyperbolic space $Z$ admits a coarse median $\mu$, obtained as follows: for any three points $x,y,z\in Z$ and any coordinate space $\C U$, the projection of $\mu(x,y,z)$ to $\C U$ is the coarse centre of any triangle with vertices $\{\pi_U(x), \pi_U(y),\pi_U(z)\}$. Two coarse median $\mu_1$ and $\mu_2$ are equivalent if
$$\sup_{x,y,z\in G} \dist_Z(\mu_{\lambda_1}(x,y,z),\mu_{\lambda_2}(x,y,z))<+\infty,$$
and a coarse median structure is an equivalence class of coarse medians. 

It is natural to ask if a given group admits a unique coarse median structure. The question is further motivated by the fact that coarse median preserving automorphisms enjoy better properties than general automorphisms \cite{Elia:coarse_med_preserving}, and all automorphisms are clearly coarse median preserving in a group with a unique coarse median structure.

We prove that changing the HHG structure of a short HHG can result in different coarse median structures:

\begin{thmintro}[see Theorem~\ref{thm:coarse_median}]\label{thmintro_coarsemedian} Let $G$ be a short HHG. Suppose that, for some vertex $v\in\ov X^{(0)}$, the stabiliser $\Stab{G}{v}$ is a $\Z$-extension of a non-elementarily hyperbolic group. Then $G$ admits a continuum of coarsely $G$-equivariant coarse median structures.
\end{thmintro}
This is in stark contrast with the case of hyperbolic groups, which admit a unique coarse median structure by \cite{NWZ}; the latter result was recently generalised to products of ``bushy'' hyperbolic spaces \cite{FS_pantsgraph}.

By inspection of the short HHG structures of the examples \ref{item:listfirst}-\ref{item:listlast} above, Theorem~\ref{thmintro_coarsemedian} applies to the following collection of short HHG:

\begin{corintro}[see Corollary \ref{cor:coarsemedian_for_all}]\label{corintro_med}
    Let $G$ be either:
    \begin{enumerate}[label=(\alph*)]
        \item\label{item:mcg} the mapping class group of a sphere with five punctures;
        \item \label{item:raag} a RAAG on a connected, triangle- and square-free graph with at least three vertices;
        \item \label{item:eltag} an Artin group of large and hyperbolic type, whose defining graph is not discrete;
        \item the fundamental group of an admissible graph of groups;
        \item an extension of a Veech group as in \cite{veech,bongiovanni2024extensions};
        \item hyperbolic relative to $\Z$-central extensions of hyperbolic groups, of which at least one is non-elementary.
        \end{enumerate}
    Then $G$ admits a continuum of coarsely $G$-equivariant coarse median structures.
\end{corintro}

Let us spell out some consequences of Corollary~\ref{corintro_med}. Firstly, Item~\ref{item:mcg} gives:
\begin{thmintro}\label{thmintro_mcg}
    The mapping class group of the five-punctured sphere admits a continuum of coarsely equivariant coarse median structures.
\end{thmintro}
This disproves that mapping class groups have a unique coarse median structure coming from subsurface projections, a common belief supported by personal communications with Elia Fioravanti and Alessandro Sisto. Even more surprisingly, Fioravanti and Sisto recently proved that another HHS naturally associated to a finite-type surface, its \emph{pants graph}, has a unique coarse median structure if the complexity of the surface is sufficiently small, and conjectured the same holds in general \cite{FS_pantsgraph}. The difference between the HHS structure of the pants graph and that of the mapping class group is that the latter also includes certain quasilines, namely annular curve graphs, so we roughly exploit the fact that a product of quasilines admits uncountably many coarse median structures (see e.g. \cite[Remark 4.8]{FS_pantsgraph}).

\par\medskip
Moving to Item~\ref{item:raag}, recall that a RAAG acts geometrically on a CAT(0) cube complex, so it inherits a coarse median structure. However, the family of possible cubulations is countable, while we produce uncountably many coarse median structures. Hence we get:

\begin{thmintro}\label{thmintro_raag}
    Let $G_\Lambda$ be a RAAG whose defining graph is connected, triangle- and square-free, and has at least three vertices. Then $G_\Lambda$ admits a continuum of coarsely equivariant coarse median structures, which are not induced by cocompact cubulations.
\end{thmintro}

The Theorem applies in particular to RAAGs whose defining graphs are trees, of which there are infinitely many quasi-isometry classes by \cite{behrstock_neumann:graph_manifold_raag}. Each such RAAG admits countably many cubical coarse median structures by \cite[Theorem 3.19]{CoarseCubicalRig}, so the countable bound on cubulations can actually be achieved. On the other side of the spectrum, the same \cite[Theorem 3.19]{CoarseCubicalRig} states the uniqueness of a cubical coarse median structure for RAAGs on triangle- and square-free graphs without leaves, which are covered by our Theorem~\ref{thmintro_raag}. Surprisingly, this shows that groups with a unique cubical coarse median structure can still have plenty of unexpected coarse medians.

Furthermore, recall that, by e.g. \cite{raag_into_racg}, the RAAG with defining graph $\Lambda$ is commensurable to the right-angled Coxeter group with defining graph $\widetilde{\Lambda}$, obtained by doubling each vertex $a\in \Lambda^{(0)}$ to a pair $a_1,a_2$ and declaring $a_i$ to be adjacent to $b_j$ if and only if $a,b$ span an edge in $\Lambda$. Hence, non-equivalent coarse medians on a RAAG induce non-equivalent coarse medians on the associated RACG, though the latter might not be coarsely equivariant any more. As all RACGs are cocompactly cubulated \cite{Niblo_reeves}, we again get the following:

\begin{thmintro}\label{thmintro_racg}
    Let $\Lambda$ be a connected, triangle- and square-free graph on at least three vertices, and let $\widetilde{\Lambda}$ be its doubling. The right-angled Coxeter group $W_{\widetilde{\Lambda}}$ admits a continuum of coarse median structures, which do not come from cocompact cubulations.
\end{thmintro}

Finally, recall that, by \cite[Theorem C]{haettel_cocompactly_artin}, an extra-large type Artin group whose defining graph is an even star acts cocompactly on a CAT(0) cube complex (these are actually the only examples of cocompactly cubulated  two-dimensional Artin groups, together with $\Z$, dihedral, and free products of the previous, see \cite[Theorem 1.1]{C(0)CC_Artin}). Then again Corollary~\ref{corintro_med}.\ref{item:eltag} implies:
\begin{thmintro}\label{thmintro_artin}
    Let $\Gamma$ be a star with even labels, all greater than or equal to $4$. The Artin group $A_\Gamma$ admits a continuum of coarsely equivariant coarse median structures, which are not induced by cocompact cubulations.
\end{thmintro}

\subsection*{Tweaking the hierarchy}
The source of flexibility of the hierarchical structure of a short HHG, which gives rise to the abundance of coarse median structures, is twofold.
\begin{enumerate}
\item Given (almost) any $g\in G$, one can find a short HHG structure where some power of $g$ generates a cyclic direction, by introducing new vertices inside $\ov X$ which are stabilised by $g$. This can be done, for example, for any element that acts loxodromically on the top-level coordinate space (think of a pseudo-Anosov element in the mapping class group of the five-holed sphere); or for (almost) any element which fixes a vertex and acts loxodromically on its link (in the same analogy, think of a partial pseudo-Anosov element supported on a subsurface).
\item\label{item:qline_from_qmorph} Given any homogeneous quasimorphism $\phi\colon\Stab{G}{v}\to\R$ which is unbounded on $Z_v$, one can find a generating set $\tau$ for $\Stab{G}{v}$ such that $\phi\colon\Cay{\Stab{G}{v}}{\tau}\to\R$ is a quasi-isometry (this is done by mimicking some arguments from \cite{ABO}). Then, under suitable conditions (see Theorem \ref{thm:squidification}), one can construct a hierarchical structure in which the bottom level quasiline $\C\ell_v$ coincides with $\Cay{\Stab{G}{v}}{\tau}$. The idea of using quasimorphisms to generate coordinate spaces was already involved in how the hierarchical structure for large hyperbolic type Artin groups \cite{ELTAG_HHS} and for graph manifold groups \cite{HRSS_3manifold} were built, so our procedure can be seen as a generalisation of both. 
\end{enumerate}

The quasimorphisms we feed into the construction from Item~\eqref{item:qline_from_qmorph} come from Corollary~\ref{cor:finding_quasimorphisms}, which roughly states that, given a central extension $0\to \Z\to G\to H\to 1$, where $H$ is hyperbolic, and a finite collection of elements $g_1,\ldots, g_n\in G$ satisfying a certain assumption, there exists a homogeneous quasimorphism $\psi\colon G\to \Z$ which is the identity on $\Z$ and is trivial on $g_i$ for every ${i=1,\ldots, n}$. This result should be compared to \cite[Lemma 4.4]{ELTAG_HHS}, whose proof, however, contains a gap. In Remark~\ref{rem:error_in_ELTAG} we point out the issue and explain how to circumvent it using our Corollary~\ref{cor:finding_quasimorphisms}, to ensure that all results from \cite{ELTAG_HHS} still hold.

\subsection*{Taking quotients by central directions}
In joint work with Alessandro Sisto \cite{short_HHG:II}, the above tools are pushed further to develop a “Dehn Filling” procedure for short HHG, similar in spirit to the relatively Dehn filling theorem \cite{Osin_Dehn_Fill, GM}. This is then used to study residual properties of short HHG. For example, by constructing suitable relatively hyperbolic quotients, we are able to show that most Artin group of large and hyperbolic type are \emph{Hopfian} (every self-epimorphism is an isomorphism); the result could be upgraded to residual finiteness, provided that certain hyperbolic groups are residually finite. Secondly, we prove that most quotients of the five-holed sphere mapping class group are hierarchically hyperbolic, addressing \cite[Question 3]{rigidity_mcg_mod_dt}.

\subsection*{Future directions}
We expect that our Theorem~\ref{thmintro_mcg} can be generalised, namely:
\begin{conjecture}
    The mapping class group of a non-sporadic finite-type surface has at least two coarse median structures.
\end{conjecture}
In the spirit of our proof of Theorem~\ref{thmintro_coarsemedian}, a possible approach would be to find suitable homogeneous quasimorphisms on curve stabilisers which vanish on certain Dehn Twist flats, and then use them to produce the quasilines that should replace annular curve graphs. Most likely, this unified strategy would not require the whole machinery of this article, which we decided to include nonetheless in view of its applications to the second paper in this series. We also expect that the construction of the HHG structure with the given quasilines would be more fickle for surfaces of higher complexity.

\par\medskip
In Theorem~\ref{thmintro_racg}, the coarse median structures we produce are most likely not coarsely equivariant, so we ask:
\begin{questintro}
    Does there exist a RACG admitting a \emph{coarsely equivariant} coarse median structure which is not induced by a cocompact cubulation?
\end{questintro}
This is related to \cite[Question 4]{Elia:coarse_med_preserving}, which asked if RACGs admit finitely many equivariant coarse median structures. In this direction, future work of Behrstock, {\c{C}}i{\c{c}}eksiz, and Falgas-Ravry shows that generic RACGs admit a unique cubical coarse median structure \cite{BehrstockCiceksizFalgasRavry:connectivity}.

\subsection*{Overview of the paper}
Section~\ref{sec:HHG/S} contains the background on hierarchically hyperbolic spaces and groups. There we also provide a description of product regions in a combinatorial HHS (Lemma~\ref{lem:PR_in_CHHS}), a tool of independent interest which, to the best of the author's knowledge, was missing in the literature. 

In Section~\ref{sec:short} we define short HHG and present the main examples, along with some of their properties. Section~\ref{sec:squid} develops the machinery to build a short HHG structure for a group. The input of this procedure is what we call \emph{blowup materials}, and include a collection of quasimorphisms which are then used to construct the bottom level quasilines (see Definition~\ref{defn:squid_material}). 

In Section~\ref{sec:short_to_squid} we notice that short HHG admit blowup materials, which can then be modified to introduce new cyclic directions, as explained in Section~\ref{sec:short_to_squid}. In Section~\ref{sec:new_examples} we also produce new examples of short HHG by constructing suitable blowup materials.

In Section~\ref{sec:mcg_coarsemedian} we prove Theorem~\ref{thmintro_coarsemedian}, regarding the existence of uncountably many coarse median structures (see Theorem~\ref{thm:coarse_median}). As in the Introduction, Theorems~\ref{thmintro_mcg} to~\ref{thmintro_artin} are then deduced from Corollary~\ref{corintro_med}, which is Corollary~\ref{cor:coarsemedian_for_all} below.

Finally, in Appendix~\ref{sec:quasicocycles} we prove how to extract certain quasimorphisms from central extensions (see Corollary~\ref{cor:finding_quasimorphisms}). We also fill a gap in a proof from \cite{ELTAG_HHS} (see Remark~\ref{rem:error_in_ELTAG}).

\subsection*{Acknowledgements}
I am grateful to my supervisor Alessandro Sisto, for the constant and active support during the development of the project. A special thanks also to Elia Fioravanti and Jason Behrstock, who reviewed an earlier draft of this document and suggested the applications to Artin and Coxeter groups.

\section{Background}\label{sec:HHG/S}

\subsection{What is a HHS}\label{subsec:axioms}
We recall from~\cite{HHS_II} the definition of a hierarchically hyperbolic space.
\begin{defn}[HHS]\label{defn:HHS}
The quasigeodesic space  $( \cuco Z,\dist_{\cuco Z})$ is a \emph{hierarchically hyperbolic space} if there exists $E\geq1$, called the \emph{HHS constant}, an index set $\frakS$, whose elements will be referred to as \emph{domains}, and a set $\{\C  U\mid U\in\frakS\}$ of $E$--hyperbolic spaces $(\C  U,\dist_U)$, called \emph{coordinate spaces},  such that the following conditions are satisfied:
\begin{enumerate}
\item\textbf{(Projections.)}\label{item:dfs_curve_complexes}
There is a set $\{\pi_U:  \cuco Z\rightarrow 2^{\C  U}\mid U\in\frakS\}$ of \emph{projections} mapping points in $ \cuco Z$ to sets of diameter bounded by $E$ in the various $\C  U\in\frakS$. Moreover, for all $U\in\frakS$, the coarse map $\pi_U$ is $(E,E)$--coarsely Lipschitz and $\pi_U( \cuco Z)$ is $E$--quasiconvex in $\C  U$.

\item \textbf{(Nesting.)} \label{item:dfs_nesting}
$\frakS$ is equipped with a partial order $\nest$, and either $\frakS=\emptyset$ or $\frakS$ contains a unique $\nest$--maximal element, denoted by $S$. When $V\nest U$, we say $V$ is \emph{nested} in $U$. For each $U\in\frakS$, we denote by $\frakS_U$ the set of $V\in\frakS$ such that $V\nest U$. Moreover, for all $U,V\in\frakS$ with $V\propnest U$ there is a specified subset $\rho^V_U\subset\C  U$ with $\diam_{\C  U}(\rho^V_U)\leq E$. There is also a \emph{projection} $\rho^U_V: \C U\rightarrow 2^{\C V}$. (The similarity in notation is justified by viewing $\rho^V_U$ as a coarsely constant map $\C V\rightarrow 2^{\C  U}$.)
 
\item \textbf{(Orthogonality.)} \label{item:dfs_orthogonal}
$\frakS$ has a symmetric and anti-reflexive relation called \emph{orthogonality}: we write $U\orth V$ when $U,V$ are orthogonal. Also, whenever $V\nest U$ and $U\orth W$, we require that $V\orth W$. We require that for each $T\in\frakS$ and each $U\in\frakS_T$ such that $\{V\in\frakS_T\mid V\orth U\}\neq\emptyset$, there exists $W\in\frakS_T-\{T\}$, which we call a \emph{container} for $U$ inside $T$, so that whenever $V\orth U$ and $V\nest T$, we have $V\nest W$. Finally, if $U \orth V$, then $U,V$ are not $\nest$--comparable.

\item \textbf{(Transversality and consistency.)}\label{item:dfs_transversal}
If $U,V\in\frakS$ are not orthogonal and neither is nested in the other, then we say $U,V$ are \emph{transverse}, denoted $U\transverse V$. In this case there are sets $\rho^V_U\subseteq\C U$ and $\rho^U_V\subseteq\C  V$, each of diameter at most $E$ and satisfying:
$$\min\left\{\dist_{U}(\pi_U(z),\rho^V_U),\dist_{V}(\pi_V(z),\rho^U_V)\right\}\leq E$$
for all $z\in  \cuco Z$.

For $U,V\in\frakS$ satisfying $V\nest U$ and for all $z\in \cuco Z$, we have: 
$$\min\left\{\dist_{U}(\pi_U(z),\rho^V_U),\diam_{\C V}(\pi_V(z)\cup\rho^U_V(\pi_U(z)))\right\}\leq E.$$ 
 
The preceding two inequalities are the \emph{consistency inequalities} for points in $ \cuco Z$.
 
Finally, if $U\nest V$, then $\dist_W(\rho^U_W,\rho^V_W)\leq E$ whenever $W\in\frakS$ satisfies either $V\propnest W$ or $V\transverse W$ and $W\not\bot U$.
 
\item \textbf{(Finite complexity.)} \label{item:dfs_complexity}
There exists $n\geq0$, the \emph{complexity} of $ \cuco Z$ (with respect to $\frakS$), so that any set of pairwise--$\nest$--comparable elements has cardinality at most $n$.
  
\item \textbf{(Large links.)} \label{item:dfs_large_link_lemma}
Let $U\in\frakS$, let $z,z'\in \cuco Z$, and let $N=E\dist_{U}(\pi_U(z),\pi_U(z'))+E$. Then there exists $\{T_i\}_{i=1,\dots,\lfloor N\rfloor}\subseteq\frakS_U- \{U\}$ such that, for any domain $T\in\mathfrak S_U-\{U\}$, either $T\in\frakS_{T_i}$ for some $i$, or $\dist_{T}(\pi_T(z),\pi_T(z'))<E$.  Also, $\dist_{U}(\pi_U(z),\rho^{T_i}_U)\leq N$ for each $i$.

\item \textbf{(Bounded geodesic image.)}\label{item:dfs:bounded_geodesic_image}
For all $U\in\frakS$, all $V\in\frakS_U- \{U\}$, and all geodesics $\gamma$ of $\C  U$, either $\diam_{\C  V}(\rho^U_V(\gamma))\leq E$ or $\gamma\cap\neb_E(\rho^V_U)\neq\emptyset$.
 
\item \textbf{(Partial realisation.)} \label{item:dfs_partial_realisation}
Let $\{V_j\}$ be a family of pairwise orthogonal elements of $\frakS$, and let $p_j\in \pi_{V_j}( \cuco Z)\subseteq \C  V_j$. Then there exists $z\in  \cuco Z$, which we call a \emph{partial realisation point} for the family, so that:
\begin{itemize}
\item $\dist_{V_j}(z,p_j)\leq E$ for all $j$,
\item for each $j$ and 
each $V\in\frakS$ with $V_j\nest V$, we have 
$\dist_{V}(z,\rho^{V_j}_V)\leq E$, and
\item for each $j$ and 
each $V\in\frakS$ with $V_j\transverse V$, we have $\dist_V(z,\rho^{V_j}_V)\leq E$.
\end{itemize}

\item\textbf{(Uniqueness.)} For each $\kappa\geq 0$, there exists
$\theta_u=\theta_u(\kappa)$ such that if $x,y\in \cuco Z$ and
$\dist_{ \cuco Z}(x,y)\geq\theta_u$, then there exists $V\in\frakS$ such
that $\dist_V(x,y)\geq \kappa$.\label{item:dfs_uniqueness}
\end{enumerate}
We often refer to $\frakS$, together with the nesting and orthogonality relations, and the projections as a \emph{hierarchically hyperbolic structure} for the space $ \cuco Z$, denoted by $(\cuco Z,\frakS)$.
\end{defn}

\begin{rem}[Normalisation]\label{rem:normalise}
    As argued in \cite[Remark 1.3]{HHS_II}, it is always possible to assume that the HHS structure is \emph{normalised}, that is, for every $U\in \frakS$ the projection $\pi_U:\,\cuco Z\to \C  U$ is uniformly coarsely surjective.
\end{rem}

\begin{notation}\label{notation:suppress_pi}
Where it will not cause confusion, given $U\in\frakS$, we will often suppress the projection map $\pi_U$ when writing distances in $\C  U$, i.e., given $x,y\in \cuco Z$ and $p\in\C  U$  we write $\dist_U(x,y)$ for $\dist_U(\pi_U(x),\pi_U(y))$ and $\dist_U(x,p)$ for $\dist_U(\pi_U(x),p)$. Note that when we measure distance between a pair of sets (typically both of bounded diameter) we are taking the minimum distance between the two sets. Given $A\subseteq  \cuco Z$ and $U\in\frakS$ we set $$\pi_{U}(A)=\bigcup_{a\in A}\pi_{U}(a).$$
\end{notation}

\subsubsection{Factors, consistency and hierarchical quasiconvexity}
\begin{defn}[Consistent tuple]\label{defn:consistent_tuple}
Let $\kappa\geq1$ and let $(b_U)_{U\in\frakS}\in\prod_{U\in\frakS}2^{\C   U}$ be a tuple such that for each $U\in\frakS$, the $U$--coordinate  $b_U$ has diameter $\leq\kappa$.  Then $(b_U)_{U\in\frakS}$ is \emph{$\kappa$--consistent} if for all $V,W\in\frakS$, we have $$\min\{\dist_V(b_V,\rho^W_V),\dist_W(b_W,\rho^V_W)\}\leq\kappa$$
whenever $V\transverse W$ and 
$$\min\{\dist_W(b_W,\rho^V_W),\diam_V(b_V\cup\rho^W_V(b_W))\}\leq\kappa$$
whenever $V\propnest W$.
\end{defn}

The following is \cite[Theorem~3.1]{HHS_II}:

\begin{thm}[Realisation]\label{thm:realisation}
Let $(\cuco Z,\frakS)$ be a hierarchically hyperbolic space. Then for each $\kappa\geq1$, there exists $\theta=\theta(\kappa)$ so that,  for any $\kappa$--consistent tuple $(b_U)_{U\in\frakS}$, there exists $x\in\cuco Z$ such that $\dist_U(x,b_U)\leq\theta$ for all $U\in\frakS$.
\end{thm}

\noindent Observe that the uniqueness axiom (Definition~\eqref{item:dfs_uniqueness}) implies that the \emph{realisation point} $x$ for $(b_U)_{U\in\frakS}$, provided by Theorem~\ref{thm:realisation}, is coarsely unique.

\begin{defn}[Product regions and factors]\label{defn:factor}
Fix a constant $\kappa\ge E$. For any domain $U$, let $F_U$ be the set of $\kappa$-consistent tuples for $U$, that is, all tuples $(b_V)_{V\in\frakS_U}$ that satisfy the consistency inequalities. Similarly, one can define $E_U$ as the set of $\kappa$-consistent tuples of the form $(b_V)_{V\orth U}$.

Now let $P_U=F_U\times E_U$, which we call the \emph{product region} associated to $U$. By the realisation Theorem~\ref{thm:realisation} there is a coarsely well-defined map $\phi\colon P_U \to \cuco Z$, constructed as follows:
a pair $(y,z)=\left((y_V)_{V\nest U},(z_V)_{V\orth U} \right)\in F_U\times E_U$ is first extended to a full tuple $(x_V)_{V\in\frakS}$, defined as follows:
$$x_V=\begin{cases}
    y_V \mbox{ if }V\nest U;\\
    z_V \mbox{ if }V\orth U;\\
    \rho^U_V \mbox{ otherwise}.
\end{cases}$$
Such a tuple is $\kappa'$-consistent, for some $\kappa'$ depending on $\kappa$ and $E$, and therefore admits a realisation point. Then one sets $\phi(y,z)$ as this realisation point.

We can metrise $\phi(P_U)$, which we will still denote by $P_U$, by endowing it with the subspace metric. Similarly, if we fix $e\in E_U$, the image of the \emph{factor} $F_U\times\{e\}$, which we will still denote by $F_U$ when the dependence on $e$ is irrelevant, can be seen as a sub-HHS of $\cuco Z$ with domain set $\frakS_U=\{V\in\frakS | V\nest U\}$. Two parallel copies $F_U\times\{e\}$ and $F_U\times\{e'\}$ are quasi-isometric (see e.g. \cite[Section 2.2]{DHScorrection} for more details), thus the metric structure on $F_U$ is well-defined up to quasi-isometry. 
\end{defn}

\begin{defn}[Hierarchical quasiconvexity and gates]\label{defn:hqc}
A subset $\cuco Y\subseteq \cuco Z$ is \emph{hierarchically quasiconvex} if there exists a function $t\colon[0,+\infty)\to [0,+\infty)$ such that the
following hold:
\begin{itemize}
    \item For all $U\in \frakS$, the projection $\pi_U(\cuco Y)$ is a $t(0)$–quasiconvex subspace of $\C U$.
    \item \textbf{Realization property}:
    For all $\kappa\ge 0$ and every point $x\in \cuco Z$ for which $\dist_U(\pi_U(x), \pi_U(\cuco Y))\le \kappa$ for all $U\in\frakS$, we have that $\dist_{\cuco Z}(x, \cuco Y)\le t(\kappa)$.
\end{itemize}
Whenever $\cuco Y$ is hierarchically quasiconvex, there exists a coarsely Lipschitz, coarse retraction $\gate_{\cuco Y} \colon \cuco Z\to \cuco Y$, called the \emph{gate} on $\cuco Y$, such that for every $x\in \cuco Z$ and every $W\in \frakS$, we have that
$ \pi_W(\gate_{\cuco Y}(z))$ uniformly coarsely coincides with the projection of $\pi_W(x)$ onto the quasiconvex set $\pi_W(\cuco Y)$. See \cite[Lemma 5.5]{HHS_II} for further details. 
\end{defn}

\begin{rem}\label{rem:factors_are_hqc}
    Product regions and factors are hierarchically quasiconvex (see \cite[Construction 5.10]{HHS_II}). Furthermore, by inspection of how $F_U$ embeds inside $\cuco Z$, one gets the following explicit description of the gate $\gate_{F_U}$: for every $z\in \cuco{Z}$, $\gate_{F_U}(z)$ has coarsely the same projection as $z$ to every $\C W$ such that $W\nest U$, and coarsely coincides with $\rho^{U}_W$ otherwise. Similarly, the gate  $\gate_{P_U}(z)$ has coarsely the same projection as $z$ to every $\C W$ whenever $W$ is either nested in or orthogonal to $U$, and projects uniformly close to $\rho^{U}_W$ otherwise.
\end{rem}

\begin{rem}[Dependence of $P_U$ on $\kappa$]\label{rem:dependence_of_PU_on_k}
    Let $P_U^E$ (resp. $P_U^\kappa$) be the product region for $U\in\frakS$ constructed with $E$-consistent tuples (resp. with $\kappa$-consistent tuples, for some $\kappa\ge E$). Clearly $P_U^E\subseteq P_U^\kappa$. Conversely, given a point $p\in P_U^\kappa$, notice that $p$ and its gate $\gate_{P_U^E}(p)$ have coarsely the same projection to every coordinate space (this is because $\pi_W(P_U^E)$ is coarsely dense in $\C W$ whenever $W\nest U$ or $W\orth U$, and coarsely coincides with $\rho^U_W$ otherwise). Therefore, the uniqueness axiom~\eqref{item:dfs_uniqueness} gives that $p$ and its gate are uniformly close, which in turn means that $P_U^E$ and $ P_U^\kappa$ are within finite Hausdorff distance. Thus we can talk about \emph{the} product region associated to $U$ as any subspace within finite Hausdorff distance from $P_U^E$. 
\end{rem}

\subsection{Hierarchically hyperbolic groups}
\begin{defn}[Automorphism]\label{defn:auto}
Let $({\cuco Z},\frakS)$ be a HHS. An \emph{automorphism} consists of a map $g:\,{\cuco Z}\to {\cuco Z}$, a bijection $g^\sharp:\, \frakS\to \frakS$ preserving nesting and orthogonality, and, for each $U\in\frakS$, an isometry $g^\diamond(U):\,\C  U\to \C (g^\sharp(U))$ for which the following two diagrams commute for all $U,V\in\frakS$ such that $U\propnest V$ or $U\transverse V$:
$$\begin{tikzcd}
{\cuco Z}\ar{r}{g}\ar{d}{\pi_U}&{\cuco Z}\ar{d}{\pi_{g^\sharp (U)}}\\
\C  U\ar{r}{g^\diamond (U)}&\C  (g^\sharp (U))\\
\end{tikzcd}$$
and
$$\begin{tikzcd}
\C  U\ar{r}{g^\diamond (U)}\ar{d}{\rho^U_V}&\C  (g^\sharp (U))\ar{d}{\rho^{g^\sharp (U)}_{g^\sharp (V)}}\\
\C  V\ar{r}{g^\diamond (V)}&\C  (g^\sharp (V))\\
\end{tikzcd}$$
Whenever it will not cause ambiguity, we will abuse notation by dropping the superscripts and just calling all maps $g$.
\end{defn}
The above definition of an automorphism is equivalent to the original formulation, which is \cite[Definition 1.20]{HHS_II} (see the discussion in \cite[Section 2.1]{DHScorrection} for further explanations).

We say that two automorphisms $g,g'$ are \emph{equivalent} if $g^\sharp=(g')^\sharp$ and $g^\diamond(U)=(g')^\diamond(U)$ for each $U\in\frakS$.  Given an automorphism $g$, a quasi-inverse $\ov{g}$ for $g$ is an automorphism with $\ov{g}^\sharp=(g^\sharp)^{-1}$ and such that, for every $U\in \frakS$, $\ov{g}^\diamond(U)=g^\diamond(U)^{-1}$. Since the composition of two automorphisms is an automorphism, the set of equivalence classes of automorphisms forms a group, denoted $\text{Aut}(\frakS)$.

\begin{defn}\label{defn:action_on_hhs}
    A finitely generated group $G$ \emph{acts} on a HHS $({\cuco Z},\frakS)$ by automorphisms if there is a group homomorphism $G\to \text{Aut}(\frakS)$.
\end{defn}

\begin{defn}[HHG]\label{defn:HHG}
A finitely generated group $G$ is \emph{hierarchically hyperbolic} if there exists a hierarchically hyperbolic space $({\cuco Z},\frakS)$ and an action $G\rightarrow\text{Aut}(\frakS)$ so that the uniform quasi-action of $G$ on ${\cuco Z}$ is \emph{geometric} (by which we will always mean \emph{metrically proper} and \emph{cobounded}), and $\frakS$ contains finitely many $G$--orbits. Then we can equip $G$ with a HHS structure, whose domains and coordinate spaces are the same as the ones for ${\cuco Z}$, and whose projections are obtained by precomposing the projections for $({\cuco Z},\frakS)$ with a $G$-equivariant quasi-isometry $G\to {\cuco Z}$ given by the Milnor- \v{S}varc lemma. We call such a structure a \emph{HHG structure} for $G$.
\end{defn}

Now let $U\in \frakS$ be a domain of a HHG, and let $P_U$ be the associated product region, which we see as a collection of $E$-consistent tuples, as explained in Remark~\ref{rem:dependence_of_PU_on_k}. Then $\Stab{G}{U}$ acts on $P_U$, as it permutes such tuples. Furthermore, when we see $P_U$ as a subspace of $\cuco Z$ and we endow it with the subspace metric, the action is metrically proper, since $G$ acts metrically properly on the whole $\cuco Z$.

\begin{defn}[Cobounded product regions]\label{defn:cobounded_stab}
    A HHG $(G,\frakS)$ has \emph{cobounded product regions} if, for every $U\in \frakS$, its stabiliser $\Stab{G}{U}$ acts coboundedly on the corresponding product region $P_{U}$ (and therefore geometrically).
\end{defn}
The above property is satisfied by all naturally occurring examples of HHG, to the point that some people argue it should be part of the definition of a HHG.

\begin{rem}[Main coordinate space in a HHS/G]\label{rem:top_guy_HHG}
    In a normalised HHS $\cuco Z$, the main coordinate space $\C S$ is quasi-isometric to the \emph{factored space} $\h{Z}$, obtained from $\cuco Z$ by adding an edge between any two points belonging to the same proper product regions (this is \cite[Corollary 2.9]{hhs_asdim}). Now, suppose that $G$ is a normalised HHG with cobounded product regions, in the sense of Definition~\ref{defn:cobounded_stab}, and fix a finite generating set $T$ for $G$ and a collection $U_1,\ldots, U_r$ of orbit representative for the $G$-action on $\frakS-\{S\}$. Then the above observation implies that $\C S$ is quasi isometric to 
    $$\Cay{G}{T\cup \{\Stab{G}{U_i}\}_{i=1,\ldots, r}}.$$
\end{rem}

\subsection{Combinatorial HHS}\label{sec:what_is_CHHS} 
In this Section we recall the definition of a combinatorial HHS and its hierarchically hyperbolic structure, as first introduced in \cite{BHMS}. 

\begin{defn}[Induced subgraph]
    Let $X$ be a simplicial graph. Given a subset $S\subseteq X^{(0)}$ of the set of vertices of $X$, the subgraph \emph{spanned} by $S$ is the complete subgraph of $X$ with vertex set $S$.
\end{defn}

\begin{defn}[Join, link, star]\label{defn:join_link_star}
Given disjoint simplices $\Delta,\Delta'$ of $X$, we let $\Delta\star\Delta'$ denote the simplex spanned by $\Delta^{(0)}\cup\Delta'^{(0)}$, if it exists. 

For each simplex $\Delta$, the \emph{link} $\link(\Delta)$ is the union of 
all simplices $\Sigma$ of $X$ such that $\Sigma\cap\Delta=\emptyset$ and $\Sigma\star\Delta$ is a simplex of $X$.  Observe that $\link(\Delta)=\emptyset$ if and only if $\Delta$ is a maximal simplex. 

The \emph{star} of $\Delta$ is $\operatorname{Star}(\Delta)\coloneq \link(\Delta)\star\Delta$, i.e. the union of all simplices of $X$ that contain $\Delta$.
\end{defn}

\begin{defn}[$X$--graph, $\W$--augmented graph]\label{defn:X_graph}
An \emph{$X$--graph} is a graph $\W$ whose vertex set is the set of all maximal simplices of $X$.

For a simplicial graph $X$ and an $X$--graph $\W$, the \emph{$\W$--augmented graph} $\duaug{X}{\W}$ is the graph defined as follows:
\begin{itemize}
     \item the $0$--skeleton of $\duaug{X}{\W}$ is $X^{(0)}$;
     \item if $v,w\in X^{(0)}$ are adjacent in $X$, then they are adjacent in $\duaug{X}{\W}$; 	
     \item if two vertices in $\W$ are adjacent, then we consider $\sigma,\rho$, the associated maximal simplices of $X$, and in $\duaug{X}{\W}$ we connect each vertex of $\sigma$ to each vertex of $\rho$.
\end{itemize}
We equip $\W$ with the usual path-metric, in which each edge has unit length, and do the same for $\duaug{X}{\W}$.
\end{defn}

\begin{defn}[Equivalence between simplices, saturation]\label{defn:simplex_equivalence}
For $\Delta,\Delta'$ simplices of $X$, we write $\Delta\sim\Delta'$ to mean $\link(\Delta)=\link(\Delta')$. We denote the $\sim$--equivalence class of $\Delta$ by $[\Delta]$.
Let $\Sat(\Delta)$ denote the set of vertices $v\in X$ for which there exists a simplex $\Delta'$ of $X$ such that $v\in\Delta'$ and $\Delta'\sim\Delta$, i.e.
$$\Sat(\Delta)=\left(\bigcup_{\Delta'\in[\Delta]}\Delta'\right)^{(0)}.$$
We denote by $\frakS$ the set of $\sim$--classes of non-maximal simplices in $X$.
\end{defn}

\begin{defn}[Complement, link subgraph]\label{defn:complement}
Let $\W$ be an $X$--graph.  For each simplex $\Delta$ of $X$, let $Y_\Delta$ be the subgraph of $\duaug{X}{\W}$ induced by the set $(\duaug{X}{\W})^{(0)}-\Sat(\Delta)$ of vertices. Let $\C  (\Delta)$ be the induced subgraph of $Y_\Delta$ spanned by $\link(\Delta)^{(0)}$, which we call the \emph{augmented link} of $\Delta$.  Note that $\C  (\Delta)=\C  (\Delta')$ whenever $\Delta\sim\Delta'$. (We emphasise that we are taking links in $X$, not in $\duaug{X}{\W}$, and then considering the subgraphs of $Y_\Delta$ induced by those links.)
\end{defn}

\begin{defn}[Combinatorial HHS]\label{defn:combinatorial_HHS}
A \emph{combinatorial HHS} $(X,\W)$ consists of a simplicial graph $X$ and an $X$--graph $\W$ satisfying the following conditions:
\begin{enumerate}
    \item \label{item:chhs_flag} There exists $n\in\mathbb N$, called the \emph{complexity} of $X$, such that any chain $\link(\Delta_1)\subsetneq\dots\subsetneq\link(\Delta_i)$, where each $\Delta_j$ is a simplex of $X$, has length at most $n$;
    \item \label{item:chhs_delta} There is a constant $\delta$ so that for each non-maximal simplex $\Delta$, the subgraph $\C  (\Delta)$ is $\delta$--hyperbolic and $(\delta,\delta)$--quasi-isometrically embedded in $Y_\Delta$, where $Y_\Delta$ is as in Definition~\ref{defn:complement};
    \item \label{item:chhs_join} Whenever $\Delta$ and $\Sigma$ are non-maximal simplices for which there exists a non-maximal simplex $\Gamma$ such that $\link(\Gamma)\subseteq\link(\Delta)\cap \link(\Sigma)$, and $\diam(\C   (\Gamma))\geq \delta$, then there exists a simplex $\Pi$ which extends $\Sigma$ such that $\link(\Pi)\subseteq \link(\Delta)$, and all $\Gamma$ as above satisfy $\link(\Gamma)\subseteq\link(\Pi)$;
    \item \label{item:C_0=C} If $v,w$ are distinct non-adjacent vertices of $\link(\Delta)$, for some simplex $\Delta$ of $X$, contained in $\W$-adjacent maximal simplices, then they are contained in $\W$-adjacent simplices of the form $\Delta\star\Sigma$.
\end{enumerate}
\end{defn}

\begin{defn}[Nesting, orthogonality, transversality, complexity]\label{defn:nest_orth}
Let $X$ be a simplicial graph.  Let $\Delta,\Delta'$ be non-maximal simplices of $X$.  Then:
\begin{itemize}
     \item $[\Delta]\nest[\Delta']$ if $\link(\Delta)\subseteq\link(\Delta')$;
     \item $[\Delta]\orth[\Delta']$ if $\link(\Delta')\subseteq \link(\link(\Delta))$.
\end{itemize}
If $[\Delta]$ and $[\Delta']$ are neither $\orth$--related nor $\nest$--related, we write 
$[\Delta]\transverse[\Delta']$.
\end{defn}

\begin{defn}[Projections]\label{defn:projections}
Let $(X,\W,\delta,n)$ be a combinatorial HHS.  

Fix $[\Delta]\in\frakS$ and define a map $\pi_{[\Delta]}:\W\to 2^{\C  ([\Delta])}$ as follows. Let
$$p:Y_\Delta\to2^{\C  ([\Delta])}$$
be the coarse closest point projection, i.e. 
$$p(x)=\{y\in\C  ([\Delta]):\dist_{Y_\Delta}(x,y)\le\dist_{Y_\Delta}(x,\C  ([\Delta]))+1\}.$$

Suppose that $w$ is a vertex of $\W$, so $w$ corresponds to a unique simplex $\Sigma_w$ of $X$. Now, \cite[Lemma 1.15]{BHMS} states that the intersection $\Sigma_w\cap Y_\Delta$ is non-empty and has diameter at most $1$. Define $$\pi_{[\Delta]}(w)=p(\Sigma_w\cap Y_\Delta).$$

We have thus defined  $\pi_{[\Delta]}:\W^{(0)}\to 2^{\C  ([\Delta])}$. If $v,w\in \W$ are joined by an edge $e$ of $\W$, then $\Sigma_v,\Sigma_w$ are joined by edges in $\duaug{X}{\W}$, and we let
$$\pi_{[\Delta]}(e)=\pi_{[\Delta]}(v)\cup\pi_{[\Delta]}(w).$$

Now let $[\Delta],[\Delta']\in\frakS$ satisfy $[\Delta]\transverse[\Delta']$ or $[\Delta']\propnest [\Delta]$, and set $$\rho^{[\Delta']}_{[\Delta]}=p(\Sat(\Delta')\cap Y_\Delta).$$

Finally, if $[\Delta]\propnest [\Delta']$, let $\rho^{[\Delta']}_{[\Delta]}:\C  ([\Delta'])\to \C  ([\Delta])$ be defined as follows.  On $\C  ([\Delta'])\cap Y_\Delta$, it is the restriction of $p$ to $\C  ([\Delta'])\cap Y_\Delta$. Otherwise, it takes the value $\emptyset$.
\end{defn}

We are finally ready to state the main theorem of \cite{BHMS}:
\begin{thm}[HHS structures for $X$--graphs]\label{thm:hhs_links}
Let $(X,\W)$ be a combinatorial HHS. Then $\W$ is a hierarchically hyperbolic space with the structure defined above.

Moreover, let $G$ be a group acting on $X$ with finitely many orbits of subcomplexes of the form $\link(\Delta)$, where $\Delta$ is a simplex of $X$. Suppose moreover that the action on maximal simplices of $X$ extends to an action on $\W$, which is metrically proper and cobounded. Then $G$ is a HHG.
\end{thm}

\begin{defn}
We will say that a group $G$ satisfying the assumptions of Theorem~\ref{thm:hhs_links} is a \emph{combinatorial} HHG.
\end{defn}

\begin{rem}
    Notice that, by how the projections $\pi_{[\Delta]}$ are defined, a combinatorial HHS is always normalised, as in Remark~\ref{rem:normalise}.
\end{rem}

We also recall that combinatorial HHS satisfy the following strengthening of the bounded geodesic image axiom, which forces a geodesic with large projection to a subdomain to pass through the $\rho$-point (and not just close to it):
\begin{lemma}[Strong BGI]\label{lem:strong_bgi}
    Let $(X,\W)$ be a combinatorial HHS. Then $\W$ satisfies the \emph{strong Bounded Geodesic Image} axiom: there exists $M>0$ such that, for every $[\Sigma]\subseteq[\Delta]$ and every $x,y\in \C(\Delta)$, if $\dist_{\C(\Sigma)}\left(\rho^{[\Delta]}_{[\Sigma]}(x),\rho^{[\Delta]}_{[\Sigma]}(y)\right)\ge M$, then every geodesic $[x,y]\subseteq\C(\Delta)$ has a vertex in $\Sat(\Sigma)$.
\end{lemma}

\begin{proof}
This follows by combining \cite[Theorem 4.9 and Lemma 5.2]{BHMS}. 
\end{proof}

\subsubsection{Product regions in a combinatorial HHS}
Here we describe the structure of product regions in a CHHS. We shall assume the following mild property, satisfied by most known examples, which should be thought of as a weak analogue of the \emph{wedge} property from \cite{BerlaiRobbio}.
\begin{defn}\label{defn:simp_cont} A combinatorial HHS $(X,\mathcal{W})$ has \emph{weak simplicial containers} if for every simplex $\Delta\subseteq X$ there exist two (possibly empty) simplices $\Theta, \Psi\subseteq X$ such that $$\link(\link(\Delta))=\link(\Theta)\star\Psi.$$
\end{defn}

\begin{lemma}\label{lem:PR_in_CHHS}
    Let $(X,\W)$ be a combinatorial HHS with weak simplicial containers, and let $\Delta$ be a non-maximal simplex of $X$. Then the product region $P_{[\Delta]}$, associated to the domain $[\Delta]$, uniformly coarsely coincides with the set of maximal simplices of $X$ of the form $\Sigma=\Pi_1\star\Pi_2\star\Pi_3$, where
    \begin{itemize}
        \item $\Pi_1$ is any simplex which is maximal inside $\link(\Delta)$;
        \item $\Pi_2$ is any simplex which is maximal inside $\link(\link(\Delta))$;
        \item $\Pi_1$ is any simplex which is maximal inside $\link(\Pi_1\star\Pi_2)$.
    \end{itemize}
\end{lemma}

\begin{proof}
Let $\Sigma=\Pi_1\star\Pi_2\star\Pi_3$ be any maximal simplex as in the statement. Arguing as in \cite[Section 5.2, \textbf{partial realization}]{BHMS}, one can prove that the projection of $\Sigma$ to any domain $[\Delta']$ which is neither nested into nor orthogonal to $[\Delta]$ is uniformly close to the projection $\rho^{[\Delta]}_{[\Delta']}$. In other words, $\Sigma$ is uniformly close to some point in $P_{[\Delta]}$.
\par\medskip
Conversely, pick a pair $$((b_{[\Delta']})_{[\Delta']\nest [\Delta]}, (c_{[\Delta'']})_{[\Delta'']\orth [\Delta]})\in F_{[\Delta]} \times E_{[\Delta]}=P_{[\Delta]},$$ and we want to show that it is realised by some $\Sigma$ as in the statement. Notice that $F_{[\Delta]}$ is the set of $E$-consistent tuples of the combinatorial HHS $(\link(\Delta), \W^\Delta)$, where two maximal simplices $\Phi,\Phi'\subset\link(\Delta)$ are $\W^\Delta$-adjacent if and only if $\Phi\star \Delta$ and $\Phi'\star \Delta$ are $\W$-adjacent (see \cite[Theorem 4.9]{BHMS}). Therefore, by the realisation Theorem~\ref{thm:realisation}, applied to $(\link(\Delta), \W^\Delta)$, $(b_{[\Delta']})_{[\Delta']\nest [\Delta]}$ must be realised by some simplex $\Pi_1$ which is maximal in $\link(\Delta)$. 

Similarly, $(c_{[\Delta'']})_{[\Delta'']\orth [\Delta]}$ is a $E$-consistent tuple of $E_{[\Delta]}$. Now, $[\Delta'']\orth [\Delta]$ means that $\link(\Delta'')\subseteq \link(\link(\Delta))=\link(\Theta)\star\Psi$. If $\link(\Delta'')\cap\Psi\neq\emptyset$ then $\link(\Delta'')$ is a join, and therefore $\C(\Delta'')$ has diameter at most $2$. Hence  $(c_{[\Delta'']})_{[\Delta'']\orth [\Delta]}$ and its restriction to $F_{[\Theta]}$ have the same realisation points, and as above the restriction to $F_{[\Theta]}$ is realised by some simplex $\Pi'_2$ which is maximal in $\link(\Theta)$. Then let $\Pi_2=\Pi_2'\star\Psi$, and complete $\Pi_1\star\Pi_2$ to a maximal simplex $\Sigma=\Pi_1\star\Pi_2\star\Pi_3$. By construction $\Sigma$ realises $(b_{[\Delta']}, c_{[\Delta'']})$, and we are done.
\end{proof}

\section{Short hierarchically hyperbolic groups}\label{sec:short}
\subsection{Blowup graphs}\label{subsec:blowup}
\begin{defn}[Blowup graph]\label{defn:blowup} Let $\ov X$ be a simplicial graph, whose vertices are labelled by graphs $\{L_v\}_{v\in\ov{X}^{(0)}}$.
The \emph{blowup} of $\ov{X}$, with respect to the collection $\{L_v\}$, is the graph $X$ obtained from $\ov{X}$ by replacing every vertex $v$ with the \emph{cone} $\Squid(v)=v * (L_v)^{(0)}$. Two cones $\Squid(v)$ and $\Squid(w)$ span a join in $X$ if and only if $v,w$ are adjacent in $\ov{X}$, and are disjoint otherwise. 
\end{defn}

We can (and will) identify $\ov X$ with its image under the embedding $\ov X\hookrightarrow X$, mapping each vertex $v\in \ov x^{(0)}$ to the apex of $\Squid(v)$. This map has a Lipschitz retraction $p\colon X\to \ov{X}$, collapsing every $\Squid(v)$ to $v$.

\begin{figure}[htp]
    \centering
    \includegraphics[width=\textwidth, alt={The cones of two adjacent vertices span a join}]{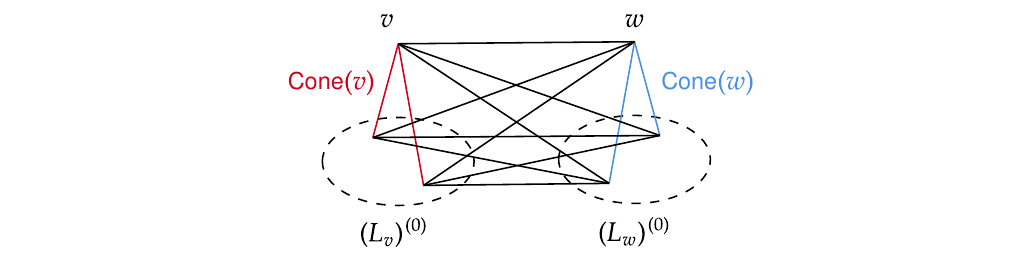}
    \caption{The blowup of two adjacent vertices of $\ov{X}$.}
    \label{fig:blowup}
\end{figure}

For every simplex $\Delta\subseteq X$, let $\ov\Delta=p(\Delta)$, which we call the \emph{support} of $\Delta$, and for every $v\in\ov\Delta$ let $\Delta_v=\Delta\cap \Squid(v)$. When describing a simplex $\Delta$, we shall put vertices belonging to the same $\Delta_v$ in parentheses: for example, if the vertices of $\Delta$ are $\{v,w,x\}$, where $v,w\in\ov X^{(0)}$ and $x\in (L_v)^{(0)}$, then we denote $\Delta$ by $\{(v,x),(w)\}$. 

A careful inspection of the construction yields the following:
\begin{lemma}[Decomposition of links]\label{lem:decomposition_of_links}
Let $\Delta$ be a simplex of $X$. Then
$$\link_X(\Delta)=p^{-1}\left(\link_{\ov X}(\ov \Delta)\right)\star (\bigstar_{v\in \ov \Delta^{(0)}}\link_{\Squid(v)}(\Delta_v)).$$
Moreover, $\link_{\Squid(v)}(\Delta_v)$ is either:
\begin{itemize}
    \item $\emptyset$, if $\Delta_v=\{(v,x)\}$ is an edge;
    \item $\{v\}$, if $\Delta_v=\{(x)\}$ where $x\in L_v$;
    \item $(L_v)^{(0)}$, if $\Delta_v=\{(v)\}$.
\end{itemize}
\end{lemma}

\begin{cor}\label{cor:bounded_links}
Suppose that $\ov{X}$ is triangle-free, and that no connected component of $\ov{X}$ is a single point. Then, given a simplex $\Delta$ of $X$, one of the following holds:
\begin{enumerate}
\item \label{cor:bounded_links_emptyset} $\Delta=\emptyset$, and $\link_X(\Delta)=X$;
 \item \label{cor:bounded_links_edge} (Edge-type simplex) $\Delta=\{(v,x)\}$, where $v\in \ov{X}^{(0)}$ and $x\in L_v$, and $\link_X(\Delta)=p^{-1}\link_{\ov X}(v)$;
 \item \label{cor:bounded_links_almostmax} (Triangle-type simplex)  $\Delta=\{(v,x),(w)\}$, where $v,w\in \ov{X}^{(0)}$ are adjacent and $x\in L_v$, and $\link_X(\Delta)=(L_w)^{(0)}$;
 \item \label{cor:bounded_links_maximal} $\Delta=\left\{(v,x), (w,y)\right\}$ is a maximal simplex, where $v,w\in \ov{X}^{(0)}$ are adjacent, $x\in L_v$ and $y\in L_w$, and $\link_X(\Delta)=\emptyset$. 
  \item \label{cor:bounded_links_bounded} $\link_X(\Delta)$ is either a single vertex or a non-trivial join.
\end{enumerate}
\end{cor}

\begin{figure}[htp]
    \centering
    \includegraphics[width=\textwidth, alt={A simplex of edge type connects a vertex with a point in its cone. A simplex of triangle type is a triangle whose link is the base of a cone.}]{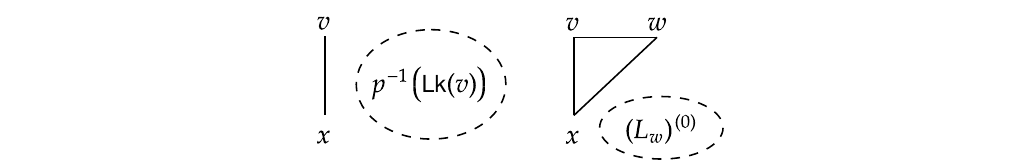}
    \caption{A simplex of edge-type (on the left) and a simplex of triangle-type (on the right). Their links are represented by the dashed areas.}
    \label{fig:links}
\end{figure}

\subsection{Definition}\label{defn:short_HHG}
Let $G$ be a combinatorial HHG, whose structure comes from the action on the combinatorial HHS $(X,\W)$. We say that $G$ is \emph{short} if it satisfies Axioms \eqref{short_axiom:graph}-\eqref{short_axiom:extension}-\eqref{short_axiom:cobounded} below.
\begin{ax}[Underlying graph]\label{short_axiom:graph} $X$ is obtained as a blowup of some graph $\ov{X}$, which is triangle- and square-free and such that no connected component of $\ov X$ is a point. Moreover, $\ov X$ is a $G$-invariant subgraph of $X$. 
\end{ax}

The above Axiom implies, in particular, that the $G$-action on $X$ restricts to a $G$-action on $\ov X$ with finitely many $G$-orbits of vertices and edges.

\begin{ax}[Vertex stabilisers are cyclic-by-hyperbolic]\label{short_axiom:extension}
    For every $v\in\ov{X}^{(0)}$ there is an extension
        $$\begin{tikzcd}
        0\ar{r}&Z_v\ar{r}&\Stab{G}{v}\ar{r}{\p_v}&H_v\ar{r}&0
        \end{tikzcd}$$
        where $H_v$ is a finitely generated hyperbolic group and $Z_v$ is a cyclic, normal subgroup of $\Stab{G}{v}$ which acts trivially on $\link_{\ov{X}}(v)$. We call $Z_v$ the \emph{cyclic direction} associated to $v$.
        
        Moreover, one requires that the family of such extensions is equivariant with respect to the $G$-action by conjugation; in particular, $Z_{gv}=gZ_vg^{-1}$ for every $v\in \ov{X}^{(0)}$ and $g\in G$.
\end{ax}

\begin{notation}\label{notation:lv_Uv}
    For every $v\in \ov{X}^{(0)}$, let $\ell_v=[\Delta]$ be the domain associated to any triangle-type simplex whose link is $(L_v)^{(0)}$, and let $\U_v=[\Sigma]$ for any edge-type simplex supported on $v$.
\end{notation}

\begin{ax}[(Co)bounded actions]\label{short_axiom:cobounded}
For every $v\in \ov{X}^{(0)}$, the cyclic direction $Z_v$ acts geometrically on $\C\ell_v$ and with uniformly bounded orbits on $\C\ell_w$ for every $w\in\link_{\ov X}(v)$. In particular, $\C\ell_v$ is a quasiline if $Z_v$ is infinite cyclic, and uniformly bounded otherwise.
\end{ax}

We will denote a short HHG, together with its short structure, by $(G, \ov X, \W)$.

\subsection{Examples}\label{sec:example}
Here we list some known combinatorial HHG which happen to be short. All examples (possibly excluding the last one, see the footnote below) also enjoy the following property, which will be relevant when taking quotients in the second paper:
\begin{defn}\label{defn:colourable}
    A short HHG $(G, \ov X, \W)$ is \emph{colourable} if the graph $\ov X$ is $G$-colourable, meaning that one can partition of the vertices of $\ov{X}$ into finitely many colours, such that no two adjacent vertices share the same colour, and the $G$-action on $\ov X$ descends to an action on the set of colours.
\end{defn}

\subsubsection{Mapping class group of a 5-holed sphere}\label{subsec:mcg}
Let $S$ be a two-dimensional sphere with five points removed, and let $\MCG(S)$ be its (extended) mapping class group, that is, the group of self-homeomorphism of $S$ up to isotopy. In \cite[Theorem 9.8]{converse}, together with Mark Hagen, we constructed a combinatorial HHG structure for $\MCG(S)$, where:
\begin{itemize}
    \item The underlying graph is the blow-up of the \emph{curve graph} $\C S$ of $S$, first introduced by Harvey \cite{Harvey}. By e.g. \cite[Lemma 2.1]{MM1}, $\C S$ is connected. Furthermore, $\C S$ is triangle- and square-free (we highlight that this is false for surfaces of higher complexity, where there is “more space” to allow these configurations of curves).
    \item In \cite{BBF}, the authors produce a finite colouring of $\C S$ which is preserved by $\MCG(S)$, and such that two adjacent curves have different colours.
    \item The stabiliser of a curve $v\in (\C S)^{(0)}$ is isomorphic to the mapping class group of the surface $Y_v$, which is the sphere with three punctures and a boundary component cut out by $v$. The centre of $\MCG(Y_v)$ contains the subgroup generated by the \emph{Dehn Twist} $T_v$ around $v$; furthermore, by e.g. \cite[Proposition 3.19]{FarbMargalit}) the quotient $\MCG(Y_v)/\langle T_v\rangle$ is a finite-index subgroup of the mapping class group of a four-punctured sphere, which is non-elementarily hyperbolic. 
    \item Given a vertex $v$ of $\C S$, the associated coordinate space $\C\ell_v$ is the \emph{annular curve graph} of $v$, which is a quasiline on which $\langle T_v\rangle $ acts coboundedly (see e.g. \cite[Section 2.4]{MM2}). Notice that, if $v$ and $w$ are two adjacent curves, then $T_v$ fixes $w$, and the $\langle T_v\rangle$-orbits inside $\C\ell_v$ are uniformly bounded (see e.g \cite[Lemma 3.1]{Mousley}).
\end{itemize}

\subsubsection{Artin groups of large and hyperbolic type}\label{subsec:example_artin}
Let $A_\Gamma$ be an Artin group of large and hyperbolic type, defined by the labelled graph $\Gamma$. In \cite{ELTAG_HHS}, the authors produce a combinatorial HHG structure $(X,\W)$ for $A_\Gamma$, with the following properties.
\begin{itemize}
    \item Let $\mathcal H$ be the collection of all cyclic subgroups generated by either a standard generator $a\in \Gamma^{(0)}$ (up to a certain equivalence relation), or by the centre $z_{ab}$ of a standard dihedral subgroup $A_{ab}\coloneq\langle a, b\rangle$. For every $H\in \mathcal H$, let $N(H)$ be its normaliser. Then $X$ is a blowup of the \emph{commutation graph} $Y$, whose vertices are the cosets of the $N(H)$, and two cosets $gN(H)$ and $hN(H')$ are adjacent if and only if $gHg^{-1}$ commutes with $hH'h^{-1}$. The authors prove that $Y$ contains no triangles or squares \cite[Lemma 3.9]{ELTAG_HHS}, and that if $\Gamma$ is connected then so is $Y$ \cite[Lemma 3.10]{ELTAG_HHS}. 
    \item By construction, $Y$ is bipartite, as two different conjugates of the standard generators never commute, nor do two different conjugates of centres of Dihedral subgroups. This gives an $A_\Gamma$-invariant colouring of $Y$.
    \item In \cite[Lemmas 2.27 and 2.28]{ELTAG_HHS}, the authors describe $N(H)$ as follows: if $H=\langle a\rangle$ then $N(H)$ is the direct product of $\langle a\rangle$ and a free group; if $H=\langle z_{ab}\rangle$ then $N(H)=A_{ab}$ is the corresponding Dihedral subgroup, which is a central extension of a virtually free group with kernel $\langle z_{ab}\rangle$. In both cases, $N(H)$ coincides with the centraliser of $H$. Notice that, as long as at least one connected component of $\Gamma$ contains an edge, there exists $H$ such that $N(H)/H$ is non-elementarily hyperbolic.
    \item If $gN(H)$ and $h N(H')$ are adjacent, then $gHg^{-1}$ fixes $h N(H')$ by \cite[Lemma 3.21]{ELTAG_HHS}.
    \item The quasiline associated to each coset is constructed in \cite[Section 4.1]{ELTAG_HHS} using \cite[Lemma 4.15]{ABO}, which takes as input a certain quasimorphism associated to the central extension and produces a quasiline. In turn, such quasimorphisms are constructed using \cite[Lemma 4.4]{ELTAG_HHS}, which is false but can be replaced by our Corollary~\ref{cor:finding_quasimorphisms}. We postpone all further explanations to the Appendix, and in particular to Remark~\ref{rem:error_in_ELTAG}.
\end{itemize}
All this shall be explored more thoroughly in the second paper.

\subsubsection{Graph manifold groups and beyond}\label{subsec:example_pi1}
Let $\mathcal G$ be a graph of groups which is \emph{admissible}, in the sense of \cite{croke-Kleiner}. This class includes all $3$-dimensional non-geometric graph manifolds. In \cite{HRSS_3manifold}, it is proved that $\pi_1(\mathcal G)$ admits a combinatorial HHG structure $(X,\W)$ with the following properties.

\begin{itemize}
    \item $X$ is obtained as a blowup of the Bass-Serre tree $T$ of $\mathcal G$, which is clearly triangle- and square-free; furthermore, $T$ is not a point (unless the graph of groups decomposition is trivial)
    \item Note that, as $T$ is a tree, it is bipartite (just fix some $v_0\in T^{(0)}$ and colour each vertex $v$ according to the parity of $\dist_T(v,v_0)$). This results in $(\pi_1(\mathcal G),T)$ being colourable, as for every $g\in \pi_1(\mathcal G)$ one has that $$\dist_T(gv, v_0)\equiv \dist_T(gv, gv_0)+\dist_T(gv_0, v_0)\equiv \dist_T(v,v_0)+\dist_T(gv_0, v_0)\mod 2.$$
    \item By how an admissible graph of groups is defined (see \cite[Definition 2.13]{HRSS_3manifold}), each vertex group $G_v$, which is the stabiliser of the corresponding vertex of $T$, has infinite cyclic centre $Z_v$, and the quotient is non-elementarily hyperbolic. Furthermore, whenever $\{v,w\}$ is an edge of $\Gamma$, $Z_v$ is also a subgroup of $G_w$, and therefore fixes the corresponding vertex of $T$.
    \item Again, quasilines are built using quasimorphisms and invoking \cite[Lemma 4.15]{ABO}.
\end{itemize}

\subsubsection{Extensions of Veech groups}\label{subsec:example_veech}
Let $S$ be a closed, connected, oriented surface of genus at least $2$. To every subgroup $G$ of $\MCG(S)$, which is isomorphic to $\text{Out}(\pi_1(S))$ by the Dehn-Nielsen-Baer theorem (see e.g. \cite[Theorem 8.1]{FarbMargalit}), one can associate a group extension 
$$\begin{tikzcd}
    0\ar{r}&\pi_1(S)\ar{r}&\Gamma\ar{r}&G\ar{r}&0
\end{tikzcd}$$
where $\Gamma$ is the fundamental group of the $S$-bundle with monodromy $G$. We now focus on the case when $G$ varies among all \emph{lattice Veech groups}, which are punctured-surface subgroups of $\MCG(S)$ first introduced in \cite{Mosher_ext} as an analogue of geometrically finite subgroups of Kleinean groups. In \cite{veech}, the authors prove that the extension $\Gamma$ admits a combinatorial HHG structure $(X,\W)$ with the following properties\footnote{To the present, the author does not know how to prove that the support graph is $\Gamma$-colourable, so we would be grateful to the authors of \cite{veech} if they could suggest an argument.}:

\begin{itemize}
    \item The graph $X$ is a blowup of a disjoint union of trees. Each tree is the Bass-Serre tree of a subgroup of $\Gamma$ isomorphic to the fundamental group of a non-geometric graph manifold. 
    \item The vertex groups admit the same description as in the graph manifold case. 
    \item Again, quasilines are built using quasimorphisms.
\end{itemize}

All this has been recently extended to all finitely generated Veech groups in \cite{bongiovanni2024extensions}.

\subsection{Properties of short HHG}
\subsubsection{Unbounded domains}
\begin{rem}\label{rem:unbounded_dom_short_hhg}
    Axiom~\eqref{short_axiom:graph} puts us in the framework of Subsection~\ref{subsec:blowup}. Hence, as a consequence of Corollary~\ref{cor:bounded_links}, if a simplex $\Delta\subseteq X$ is such that the associated coordinate space $\C (\Delta)$ is unbounded, then $$[\Delta]\in\{S\}\cup\{\ell_v\}_{v\in\ov{X}^{(0)}}\cup \{\U_v\}_{v\in\ov{X}^{(0)},\,|\link_{\ov X}(v)|=\infty}$$
    where $S=[\emptyset]$ is the $\nest$-maximal domain, and $\ell_v$, $\U_v$ are defined as in Notation~\ref{notation:lv_Uv}. 
    
    Furthermore, by how nesting and orthogonality are defined in a combinatorial HHS (Definition~\ref{defn:nest_orth}), we see that:
    \begin{itemize}
        \item $\ell_v\orth \ell_w$ whenever $v\neq w$ are adjacent in $\ov X$, and are transverse otherwise; 
        \item $\ell_v\orth \U_v$; 
        \item $\ell_v\nest \U_w$ whenever $v\neq w$ are adjacent in $\ov X$;
        \item If $v$ has valence greater than one in $\ov X$, and $\dist_{\ov X}(v,w)\ge 2$, then $\U_v\transverse \ell_w$.
        \item If both $v$ and $w$ have valence greater than one in $\ov X$, and $\dist_{\ov X}(v,w)\ge 2$, then $\U_v\transverse \U_w$.
    \end{itemize}
    We avoided describing the slightly more complicated relations involving $\U_v$ when $v$ has valence one in $\ov X$, as we shall not need them.
\end{rem}

\subsubsection{Strong Bounded Geodesic Image}
\begin{rem}\label{rem:coord_spaces_wrt_lines}
 The main coordinate space $\C S=X^{+\W}$ $G$-equivariantly retracts onto the \emph{augmented support graph} $\ov X^{+\W}$, obtained from $\ov{X}$ by adding an edge between $v$ and $w$ if they belong to $\W$-adjacent maximal simplices of $X$. In other words, $\C S$ is $G$-equivariantly quasi-isometric to a graph with vertex set $\ov X^{(0)}$, which contains $\ov X$ as a (non-full) $G$-invariant subgraph.
 
Similarly, for every $v\in \ov{X}^{(0)}$, $\C \U_v$ is $\Stab{G}{v}$-equivariantly quasi-isometric to the \emph{augmented link} $\link_{\ov{X}}(v)^{+\W}$, on which $Z_v$ acts trivially.
  \end{rem}

\begin{rem}[Saturations of triangle-type] \label{rem:saturation_in_Short}
Let $\Delta=\left\{(v,x), (w)\right\}$. Then $$\Sat(\Delta)=\{w\}\cup_{u\in\link_{\ov{X}}(w)}\Squid(u).$$ 
Furthermore, if $\Sigma=\left\{(v,x)\right\}$ is of edge-type then $\Sat(\Delta)\cap \link(\Sigma)=\{w\}$, as $\ov{X}$ is triangle-free.
\end{rem}

\begin{notation}\label{notation:rho}
    Set $\rho^w_{w'}\coloneq \rho^{[\Delta]}_{[\Delta']}$ for any two simplices of triangle-type $\Delta=\left\{(v,x), (w)\right\}$ and $\Delta'=\left\{(v',x'), (w')\right\}$. 
\end{notation}
Then we can invoke Lemma~\ref{lem:strong_bgi}, together with the description of saturations from Remark~\ref{rem:saturation_in_Short}, to get the following version of the strong bounded geodesic image axiom (to avoid having to deal with many constants, we first enlarge the HHS constant $E$ for $(X,\W)$ to be bigger than the constant $M$ from Lemma~\ref{lem:strong_bgi}):

\begin{lemma}[Strong BGI for projections on quasilines]\label{lem:strong_bgi_for_short}
    Whenever $u,v,w\in \ov X^{(0)}$, if both $\rho^u_w$ and $\rho^v_w$ are defined and at least $2E$-apart in $\C \ell_w$, then every geodesic $[u,v]\subseteq \ov X^{+\W}$ must pass through $\operatorname{Star}_{\ov{X}}(w)$.

    Similarly, whenever $u,v,w\in\link_{\ov{X}}(z)$, if both $\rho^u_w$ and $\rho^v_w$ are defined and at least $2E$-apart in $\C\ell_w$, then every geodesic $[u,v]\subseteq \link_{\ov{X}}(z)^{+\W}$ must pass through $w$.
\end{lemma}

We can establish a similar result for the projection to $\C\U_w$. Let $\Delta=\{(w,y)\}$ be a simplex of edge type supported on $w$. If $w$ has valence at lest two in $\ov X$ then $\Sat(\Delta)=\Squid(w),$ as $\ov X$ is square-free. Moreover, the retraction $p$ maps $\C(\Delta)$ onto $\link_{\ov X}(w)^{+\W}$, so for every $u,v\in \ov X^{(0)}-\{w\}$ we can define $$\dist_{\link_{\ov X}(w)^{+\W}}(u, v)=\dist_{\link_{\ov X}(w)^{+\W}}\left( p(\rho^{\ell_u}_{\U_w}), p(\rho^{\ell_v}_{\U_w})\right).$$
If instead $w$ has valence one then $\link_{\ov X}(w)$ is a point, and we set $\dist_{\link_{\ov X}(w)^{+\W}}(u, v)=0$.

\begin{lemma}[Strong BGI for projections on links]\label{lem:strong_bgi_for_short_edge_type}
    Let $w\in\ov X^{(0)}$. For every $u,v\in \ov X^{(0)}-\{w\}$, if $\dist_{\link_{\ov X}(w)^{+\W}}(u, v)\ge 2E$, then every geodesic $[u,v]\subseteq \ov X^{+\W}$ must pass through $w$.
\end{lemma}

\begin{proof}
     If $w$ has valence one then $\link_{\ov X}(w)^{+\W}=\{w\}$ is a point, and the lemma holds vacuously. Otherwise the conclusion follows from Lemma~\ref{lem:strong_bgi}. 
\end{proof}

\subsubsection{Cobounded product regions}
In this paragraph, we prove that short HHG have cobounded product regions, in the sense of Definition~\ref{defn:cobounded_stab}. We first need the following general lemma:
\begin{lemma}\label{lem:simplicial_cont_for_squid}
    Let $(X,\W)$ be a combinatorial HHS, where $X$ is a blowup of some triangle- and square-free graph $\ov X$ such that no connected component of $\ov X$ is a point. Then $(X,\W)$ has weak simplicial containers, in the sense of Definition~\ref{defn:simp_cont}. More precisely, for every simplex $\Delta$ of $X$, there exists $\Theta$ and $\Psi$ such that 
    $$\link(\link(\Delta))=\link(\Theta)\star\Psi,$$
    and one can always choose $\Psi=\emptyset$, unless $\Delta=\{(v,x)\}$ is of edge-type and $v$ has valence greater than one in $\ov X$.
\end{lemma}
\begin{proof}
Recall that, for every simplex $\Delta$ of $X$ and every $v\in \ov\Delta$,  $\Delta_v$ is defined as $\Delta\cap\Squid(v)$. Taking the link in the expression from Lemma~\ref{lem:decomposition_of_links}, we get that 
$$\link(\link(\Delta))=p^{-1}(\link_{\ov X}(\link_{\ov X}(\ov\Delta)))\cap\bigcap_{v\in\ov\Delta}\link_X(\link_{\Squid(v)}(\Delta_v)).$$
We explicitly go through all possible shapes of $\Delta$ and produce two simplices $\Theta, \Psi$  such that $\link(\link(\Delta))=\link(\Theta)\star\Psi$.

If $\Delta=\emptyset$ then clearly $\link(\link(\Delta))=\emptyset$ is the link of any maximal simplex.

Now assume that $\ov\Delta=\{v\}$ is a single vertex. Let $w\in\link_{\ov X}(v)$, which exists as no component of $\ov X$ is a single vertex, and let $x\in (L_v)^{(0)}$ and $y\in (L_w)^{(0)}$. Figure~\eqref{table:lklk1} shows how to choose $\Theta$, according to the shape of $\Delta$. The only case which requires attention is when $\Delta=\{(v,x)\}$ is of edge-type: if $\link_{\ov X}(v)=\{w\}$ is a single vertex, then we set $\Theta=\{(w,y)\}$ and $\Psi=\emptyset$; otherwise, as $\ov X$ is triangle- and square-free, we have that $\link_{\ov X}(\link_{\ov X}(v))=\{v\}$, so we set $\Theta=\{(v), (w,y)\}$ and $\Psi=\{(v)\}$.

Finally, suppose that $\ov\Delta=\{v,w\}$ is an edge. Let $x\in (L_v)^{(0)}$ and $y\in (L_w)^{(0)}$. Then we choose $\Theta$ and $\Psi$ as in Figure~\eqref{table:lklk2}.
\end{proof}

\begin{figure}[htp]
    \centering
    \begin{tabular}{c|c|c|c|c}
         $\Delta$& $\link(\Delta)$ & $\link(\link(\Delta))$ & $\Theta$ & $\Psi$ \\
         \hline
         $\{(v)\}$& $p^{-1}(\link_{\ov X}(v)) \star (L_v)^{(0)}$ & $\{v\}$ & $\{(x), (w, y)\}$& $\emptyset$\\
         $\{(x)\}$& $p^{-1}(\link_{\ov X}(v)) \star \{v\}$ & $(L_v)^{(0)}$ & $\{(v), (w, y)\}$& $\emptyset$\\
        $\{(v,x)\}$ (valence $1$)& $\Squid(w)$ & $p^{-1}(\link_{\ov X}(w))$ & $\{(w,y)\}$& $\emptyset$\\
        $\{(v,x)\}$ (valence $\ge 2$)& $p^{-1}(\link_{\ov X}(v))$ & $\Squid(v)$ & $\{(v), (w,y)\}$& $\{(v)\}$
    \end{tabular}
    \caption{How to choose $\Theta$ and $\Psi$ when $\ov\Delta$ is a single vertex. }
    \label{table:lklk1}
\end{figure}

\begin{figure}[htp]
    \centering
    \begin{tabular}{c|c|c|c|c}
         $\Delta$& $\link(\Delta)$ & $\link(\link(\Delta))$ & $\Theta$ & $\Psi$ \\
         \hline
         $\{(v),(w)\}$&  $(L_v)^{(0)}\star (L_w)^{(0)}$ & $\{v,w\}$ & $\{(x), (y)\}$& $\emptyset$\\
         $\{(v),(y)\}$&  $(L_v)^{(0)}\star \{w\}$ & $\{v\}\star (L_w)^{(0)}$ & $\{(x), (w)\}$& $\emptyset$\\
         $\{(x),(y)\}$&  $\{v,w\}$ & $(L_v)^{(0)}\star (L_w)^{(0)}$ & $\{(v), (w)\}$& $\emptyset$\\
        $\{(v,x),(w)\}$&  $(L_w)^{(0)}$ & $p^{-1}(\link_{\ov X}(w))\star \{w\}$ & $\{(y)\}$& $\emptyset$\\
        $\{(v,x),(y)\}$&  $\{w\}$ & $p^{-1}(\link_{\ov X}(w))\star (L_w)^{(0)}$ & $\{(w)\}$& $\emptyset$
    \end{tabular}
    \caption{How to choose $\Theta$ and $\Psi$ when $\ov\Delta$ is an edge. In this case one can always choose $\Psi=\emptyset$.}
    \label{table:lklk2}
\end{figure}

\begin{lemma}\label{lem:cobounded_stab_for_short}
    A short HHG has cobounded product regions, in the sense of Definition~\ref{defn:cobounded_stab}.
\end{lemma}

\begin{proof}
Let $\Delta$ be a simplex of $\ov X$. We want to prove that $\Stab{G}{[\Delta]}$ (which is the setwise stabiliser of $\link(\Delta)$) acts coboundedly on the associated product region $P_{[\Delta]}$. In Figure~\eqref{fig:stab_and_PR} we go through all possible shapes of $\Delta$ and describe both $\Stab{G}{[\Delta]}$ and $P_{[\Delta]}$. Recall that, since $(X,\W)$ has weak simplicial containers, we can use Lemma~\ref{lem:PR_in_CHHS} to get a description of $P_{[\Delta]}$ as the maximal simplices of $X$ of the form $\Sigma=\Pi_1\star\Pi_2\star \Pi_3$, where $\Pi_1$ is maximal in $\link(\Delta)$ and $\Pi_2$ is maximal in $\link(\link(\Delta))$.

\begin{figure}[htp]
    \centering
    \makebox[\textwidth][c]{
    \begin{tabular}{c|c|c|c|c}
         $\Delta$& $\link(\Delta)$ &$\link(\link(\Delta))$ & $\Stab{G}{\link(\Delta)}$ & $\Sigma \in P_{[\Delta]}$  \\
         \hline
         $\{(v)\}$& $p^{-1}(\link_{\ov X}(v)) \star (L_v)^{(0)}$ &$\{v\}$ & $\Stab{G}{v}$ & $\{(v)\}\subseteq \Sigma$\\
         
         $\{(x)\}$& $p^{-1}(\link_{\ov X}(v)) \star \{v\}$ & $(L_v)^{(0)}$ & $\Stab{G}{v}$ & $\{(v)\}\subseteq \Sigma$\\
         
        $\{(v,x)\}$ (valence $1$)& $\Squid(w)$ & $p^{-1}(\link_{\ov X}(w))$ & $\Stab{G}{w}$& $\{(w)\}\subseteq \Sigma$ \\
        
        $\{(v,x)\}$, (valence $\ge 2$)& $p^{-1}(\link_{\ov X}(v))$ & $\Squid(v)$ & $\Stab{G}{v}$& $\{(v)\}\subseteq \Sigma $ \\
        
         $\{(v),(w)\}$&  $(L_v)^{(0)}\star (L_w)^{(0)}$ & $\{v,w\}$ & $\Stab{G}{e}$& $\{(v),(w)\}\subseteq \Sigma $ \\
         
         $\{(v),(y)\}$&  $(L_v)^{(0)}\star \{w\}$ & $\{v\}\star (L_w)^{(0)}$ & $\text{P}\Stab{G}{e}$& $\{(v),(w)\}\subseteq \Sigma $ \\

         $\{(x),(y)\}$&  $\{v,w\}$ & $(L_v)^{(0)}\star (L_w)^{(0)}$ & $\Stab{G}{e}$& $\{(v),(w)\}\subseteq \Sigma $ \\
         
        $\{(v,x),(w)\}$&  $(L_w)^{(0)}$ & $p^{-1}(\link_{\ov X}(w))\star \{w\}$ & $\Stab{G}{w}$ & $\{(w)\}\subseteq \Sigma $ \\
        $\{(v,x),(y)\}$&  $\{w\}$ & $p^{-1}(\link_{\ov X}(w))\star (L_w)^{(0)}$ &  $\Stab{G}{w}$ & $\{(w)\}\subseteq \Sigma $ 
    \end{tabular}}
    \caption{Description of the stabilisers of links, and of the associated product regions. Here $v,w\in\ov{X}^{(0)}$ are adjacent, $e=\{v,w\}$ is the edge with endpoints $v$ and $w$, and $x\in (L_v)^{(0)}$ and $y\in (L_w)^{(0)}$. Moreover $\text{P}\Stab{G}{e}\coloneq \Stab{G}{v}\cap\Stab{G}{w}$ is the point-wise stabiliser of $e$, while $\Stab{G}{e}$ is the set-wise stabiliser. Clearly $\text{P}\Stab{G}{e}$ has index at most two in $\Stab{G}{e}$.}
    \label{fig:stab_and_PR}
\end{figure}

By inspection of the Table, we only need to show that $\text{P}\Stab{G}{e}$ acts coboundedly on $P_e\coloneq\{\Sigma\in \W\,|\, \{(v), (w)\}\subseteq \Sigma\}$ (edge case below), and that 
the same holds for the $\Stab{G}{v}$-action on $P_v\coloneq\{\Sigma\in \W\,|\, \{(v)\}\subseteq \Sigma\}$ (vertex case below).

\par\medskip

\textbf{Edge case}:  Notice that $\text{P}\Stab{G}{e}$ contains $\langle Z_v, Z_w\rangle$, which acts coboundedly on $P_e$ as every cyclic direction acts coboundedly on the corresponding $\C\ell$ and with uniformly bounded orbits on the other. Thus $\text{P}\Stab{G}{e}$ acts coboundedly on $P_e$. For later purposes, let $B\ge 0$ be such that, for every edge $e$, every $\text{P}\Stab{G}{e}$-orbit inside $P_e$ is $B$-dense in $P_e$ (such a uniform constant exists, as there are finitely many $G$-orbits of $\ov X$-edges).

\par\medskip

\textbf{Vertex case}: Let $\Sigma, \Sigma'\in P_v$, and let $w$ (resp. $w'$) be the vertex of $\ov \Sigma$ (resp. $\ov \Sigma'$) other than $v$. Notice that $\Stab{G}{v}$ acts cofinitely on $\link_{\ov X}(v)$, because $G$ acts on $\ov X$ with finitely many orbits of edges. Thus, there exist a constant $K\in\mathbb{N}_{>0}$, some $g\in \Stab{G}{v}$, and a collection $w_0=w, w_1, \ldots w_{k-1}, w_k=g w'\in \link_{\ov X}(v)$, where $k\le K$, such that every $w_i$ is $\W$-adjacent to $w_{i+1}$. Furthermore, we can choose $K$ to be uniform in $v$, as there are finitely many $G$-orbits of vertices of $\ov X$.

Now, by fullness of links, Definition~\ref{defn:combinatorial_HHS}.\eqref{item:C_0=C}, for every $i=1,\ldots, k$ we can find two $\W$-adjacent maximal simplices $\Phi_i, \Psi_i$, such that $\ov\Phi_i=\{v, w_{i-1}\}$ while $\ov\Psi_i=\{v, w_{i}\}$. Now set $\Psi_0=\Sigma$ and $\Phi_{i+1}=g \Sigma'$. As a consequence of the edge case, for every $i=0,\ldots, k$ we can find $g_i\in \Stab{G}{\{v, w_i\}}\le \Stab{G}{v}$ such that $$\dist_{P_v}(\Psi_i, \Phi_{i+1})\le \dist_{P_{\{v,w_i\}}}(\Psi_i, \Phi_{i+1})\le B.$$
Putting everything together, we get that the distance between $\Sigma$ and $g\Sigma'$ in $P_v$ is at most $(k+1)B+k\le KB+K+B$, and this concludes the proof.
\end{proof}

\subsubsection{(Relative) hyperbolicity}
\begin{lemma}\label{lem:hyperbolic_short_HHG}
    A short HHG whose cyclic directions are all bounded is hyperbolic.
\end{lemma}

\begin{proof}
As each cyclic direction $Z_v$ acts coboundedly on the associated $\C\ell_v$, we have that the latter is always bounded. Thus, by
Remark~\ref{rem:unbounded_dom_short_hhg} the collection of domains with unbounded coordinate spaces is contained in $$\{S\}\cup\{\U_v\}_{\{v\in \ov X^{(0)},\,|\link_{\ov X}(v)|=\infty\}}.$$ 
Moreover, as described in the same Remark, no two domains in this collection are orthogonal, and therefore $G$ is hyperbolic by e.g. \cite[Corollary 2.16]{quasiflats}.
\end{proof}

We now show that, for every vertex $v\in \ov X^{(0)}$, $H_v$ has a simple relatively hyperbolic structure, whose peripherals are (finite index overgroups of) the adjacent central directions. We first need two lemmas:

\begin{lemma}\label{lem:finite_int_w-v-w'}
    Whenever $w,w'\in\link_{\ov X}(v)$ are distinct, the intersection $Z_w\cap Z_{w'}$ is finite. 
\end{lemma}

\begin{proof}
    The corollary trivially holds if either $Z_w$ or $ Z_{w'}$ is finite. Otherwise, let $g\in Z_w\cap Z_{w'}$. Since $g\in \Stab{G}{w}\cap \Stab{G}{v'}$, it must fix the projection $\rho^{w'}_w\in\C\ell_w$. Therefore $g=1$, as every non-trivial element of $Z_w$ acts loxodromically on $\C \ell_w$.
\end{proof}

\begin{lemma}\label{lem:malnormality}
    Let $v\in \ov{X}^{(0)}$, and let $W$ be a collection of $\Stab{G}{v}$-orbit representatives of vertices in $\link_{\ov X}(v)$ with unbounded cyclic directions. The collection $\{\p_v(Z_w)\}_{w\in W}$ is \emph{independent} in $H_v$, that is, for any two different $w,w'\in W$ and any $h\in H$, $\p_v(Z_w)^h\cap \p_v(Z_{w'})=\{1\}$.
\end{lemma}

\begin{proof}
    Given $h,w,w'$ as in the statement, let $Z_v=\langle z_v\rangle$, and similarly define $z_w$ and $z_{w'}$. Let $g\in \Stab{G}{v}$ be a preimage of $h$. If $\p_v(Z_w)^h\cap \p_v(Z_{w'})$ was not trivial, one could find integers $a,b\in \Z-\{0\}$, $c\in \Z$ such that $$z_{gw}^a=gz_w^ag^{-1}=z_{w'}^bz_v^c.$$
    Notice that $gw\neq w'$, as $w$ and $w'$ are in different $\Stab{G}{v}$-orbits. But this yields a contradiction, as the element on the right-hand side fixes $w'$ while $z_{gw}$ does not (again, because it acts loxodromically on $\C \ell_{gw}$ and therefore does not fix the projection $\rho^{w'}_{gw}$).
\end{proof}

Now recall that, if $H$ is a hyperbolic group and $h\in H$ has infinite order, then $\langle h\rangle$ is undistorted in $H$ by e.g. \cite[Corollary III.3.10]{bridsonhaefliger}. In particular, it acts loxodromically on the Cayley graph for $H$ with respect to any finite set of generators $S$. Furthermore, since $H$ acts properly discontinuously on $\Cay{H}{S}$, $h$ is \emph{weakly properly discontinuous}, or \emph{WPD}, in the sense of e.g. \cite{bestvina_fujiwara}.
Then \cite[Lemma 6.5]{DGO} grants the existence of a maximal virtually cyclic subgroup $K_H(h)\le H$ containing $h$. Notice that $K_H(h)=K_H(h^n)$ for every $n\in \Z-\{0\}$ (as a consequence of e.g. \cite[Corollary 6.6]{DGO}); hence, if $Q$ is a cyclic subgroup acting loxodromically on $\Cay{H}{S}$, with a little abuse of notation we can define $K_H(Q)\coloneq K_H(h)$ for some (equivalently, any) non-trivial $h\in Q$.

\begin{lemma}\label{lem:Hv_hyp_1}
    Let $(G, \ov X, \W)$ be a short HHG, $v\in \ov{X}^{(0)}$, and let $W$ be a collection of $\Stab{G}{v}$-orbit representatives of vertices in $\link_{\ov X}(v)$ with unbounded cyclic directions. Then $H_v$ is hyperbolic relative to $\{K_{H_v}(\p_v(Z_w))\}_{w\in W}$. In particular, $H_v$ is hyperbolic.
\end{lemma}

\begin{proof}
Let $F_{\U_v}$ be the factor associated to $\U_v$, which is a HHS whose domain set corresponds to the domains nested in $\U_v$. Hence, every domain for $F_{\U_v}$ with unbounded coordinate space is in the collection 
$$\{\U_v\}\cup\bigcup_{w\in\link_{\ov X}(v)} \ell_w.$$
Since no two domains in such collection are orthogonal, \cite[Theorem 3.2]{russell} yields that $F_{\U_v}$ is hyperbolic.

Now, $\Stab{G}{v}$ acts on $F_{\U_v}$. As $Z_v$ acts trivially on $\C\U_v=\link_{\ov X}(v)^{+\W}$ and with uniformly bounded orbits on every $\C\ell_w$, by the uniqueness axiom~\ref{defn:HHS}.\eqref{item:dfs_uniqueness}, there exists a uniform constant $B\ge 0$ such that $Z_v$-orbits in $F_{\U_v}$ have diameter at most $B$. In turn, this means that the action $\Stab{G}{v}\circlearrowleft F_{\U_v}$ induces a \emph{quasi-action} $H_v\circlearrowleft F_{\U_v}$, meaning that, for every $h,h'\in H_v$ and every $x\in F_{\U_v}$, we have that
$$\dist_{F_v}\left( (h h')(x), h(h'(x))\right)\le B.$$
\begin{claim}\label{claim:quasiaction_geometric_H_V}
    The quasi-action is metrically proper and cobounded.
\end{claim}

\begin{claimproof}[Proof of Claim~\ref{claim:quasiaction_geometric_H_V}]
First, as short HHG have cobounded product regions (Lemma \ref{lem:cobounded_stab_for_short}), $\Stab{G}{v}$ acts coboundedly on $P_{\ell_v}=F_{\ell_v}\times F_{\U_v}$. The action is simply the direct product of the actions on the factors, so $\Stab{G}{v}$ acts coboundedly on $F_{\U_v}$ as well. Furthermore, as $Z_v$ acts with uniformly bounded orbits, the quasi-action of $H_v$ on $F_{\U_v}$ is still cobounded. 
\par\medskip
Moving to metric properness, fix any point $x\in F_{\U_v}$ and any radius $R\ge0$. Suppose that $h\in H_v$ is such that $\dist_{F_{\U_v}}(x,hx)\le R$, and we want to show that $h$ belongs to a finite collection of elements. For every $g\in \Stab{G}{v}$ such that $\p_v(g)=h$, we have that
$$\dist_{F_{\U_v}}(x,gx)\le R+B.$$
Now fix $p\in F_{\ell_v}$. As $Z_v$ acts coboundedly on $F_{\ell_v}$, we can choose $g$ in such a way that $\dist_{F_{\ell_v}}(p,gp)$ is bounded above by some constant $B'$. But then, if we see the pair $(p, x)$ as a point in $P_{\ell_v}$, we have that $\dist_{P_{\ell_v}}\left((p,x),g(p,x)\right)$ is bounded above in terms of $R$, $B$, and $B'$. Since $\Stab{G}{v}$ acts metrically properly on $P_{\ell_v}$, there are finitely many such $g$, and therefore finitely many such $h$.
\end{claimproof}
Combining the Claim with e.g. \cite[Proposition 4.4]{Manning_pseudochar} we have that $H_v$ acts metrically properly and coboundedly on a geodesic metric space $F'$ quasi-isometric to $F_{\U_v}$, which is therefore hyperbolic; thus $H_v$ is a hyperbolic group.

Moving to the relative hyperbolic structure, $\{\p_v(Z_W)\}_{w\in W}$ is an independent collection of infinite cyclic subgroups, which act loxodromically (hence WPD) on $\Cay{H_v}{T}$ for some finite generating set $T$. Then by \cite[Theorem 6.8]{DGO} the collection $\{K_{H_v}(\p_v(Z_w))\}_{w\in W}$ is \emph{hyperbolically embedded} in $(H_v,T)$, in the sense of \cite[Definition 2.1]{DGO}. Then, since $T$ is finite, \cite[Proposition 4.28]{DGO} yields that $H_v$ is hyperbolic relative to the collection, as required.
\end{proof}

\subsubsection{Absence of hidden symmetries}

\begin{defn}[No hidden symmetries]\label{defn:hiddensymm_general}
    Let $0\to \Z\to G\xrightarrow[]{\pi} H\to 1$ be a group extension, with $H$ hyperbolic, and let $C\le G$ be a cyclic subgroup. We say that $C$ has \emph{no hidden symmetries} if $\pi(C)$ is infinite and $C$ contains a finite-index subgroup which is normal in $\pi^{-1}(K_H(\pi(C)))$. We say that an element $g\in G$ has no hidden symmetries if $\langle g\rangle$ has no hidden symmetries.
\end{defn}

\begin{lemma}\label{lem:nohidden}
    Let $(G,\ov X, \W)$ be a short HHG, and let $v,w\in \ov X^{(0)}$ be adjacent vertices with infinite cyclic directions. Then $Z_w$ has no hidden symmetries in $\Stab{G}{v}$.
\end{lemma}

\begin{proof} 
    Let $n\in \Z-\{0\}$ be such that $\langle \p_v(z_w^n)\rangle$ is normal in $K_{H_v}(\p_v(z_w))$. Then for every $p\in \p_v^{-1}(K_{H_v}(\p_v(z_w)))$ there exist $k\in \Z$ and $\varepsilon\in \{\pm 1\}$ such that $$z_w^{\varepsilon n}=pz_w^np^{-1}z_v^k=z_{pw}^nz_v^k.$$
    The vertices $w$ and $pw$ are both in $\link_{\ov X}(v)$, so they either coincide or are not adjacent. In the latter case we would get a contradiction, because the right hand side would fix $pw$ while the left hand side would not (again, since $Z_w$ acts loxodromically on $\C\ell_w$ and therefore cannot fix $\rho^{pw}_w$). Then $w=pw$, meaning that $z_w^{(\varepsilon -1)n}=z_v^k$. In turn, this is possible if and only if $\varepsilon=1$ and $k=0$, because $Z_v$ acts with uniformly bounded orbits on $\C\ell_w$. This proves that every $p\in \p_v^{-1}(K_{H_v}(\p_v(z_w)))$ normalises $\langle z_w^n\rangle$, as required.
\end{proof}

\section{Short structures from blowup materials}\label{sec:squid}
In this Section we describe a procedure to manufacture a short HHG structure on a group. The inputs of our machinery are what we call \emph{blowup materials}, which roughly consist of:
\begin{itemize}
    \item a graph $\ov X$ to blow-up, on which $G$ acts with cyclic-by-hyperbolic vertex stabilisers;
    \item a choice of a \emph{quasimorphism} $\Stab{G}{v}\to \R$ every vertex $v\in \ov X^{(0)}$, which we use to build the coordinate space associated to $\ell_v$ (see Subsection~\ref{subsec:quasilines_from_quasimorph} below);
    \item a family of coarse retractions $G\to \Stab{G}{v}$, which we use to define projections.
\end{itemize}  
As we shall see in Proposition~\ref{prop:short_is_squid}, admitting blowup materials is equivalent to being a short HHG. However, the advantage of this point of view is that one is often able to replace some blowup materials to get a plethora of different short HHG structures. This will be the source of the non-equivalent coarse medians (see Section~\ref{sec:mcg_coarsemedian}).

\subsection{Quasilines from quasimorphisms}\label{subsec:quasilines_from_quasimorph}
We first describe how, given a quasimorphism on a group, one can construct a quasiline on which the group acts. We start by recalling some definitions.
\begin{defn}[(Homogeneous) quasimorphism]
    Let $H$ be a group. A map $m\colon H\to \mathbb{R}$ is a \emph{quasimorphism} if there exists a constant $D\ge 0$, called the \emph{defect} of $m$, such that for every $h,k\in H$ we have that
    $$|m(hk)-m(h)-m(k)|\le D.$$
    A quasimorphism is \emph{homogeneous} if it restricts to a group homomorphism on every cyclic subgroup of $H$, i.e. $m(h^n)=nm(h)$ for every $h\in H$ and $n\in\mathbb{Z}$. 
\end{defn}

\begin{rem}\label{rem:plasmon}
To every quasimorphism $m$ one can associate a homogeneous quasimorphism $m'$, defined as
$$m'(h)=\lim_{n\to+\infty} \frac{m(h^n)}{n}.$$
Notice that, if $m$ is already homogeneous, then $m'$ coincides with $m$. This in particular implies that a homogeneous quasimorphism is invariant under conjugation. Indeed,
$$m(khk^{-1})=\lim_{n\to+\infty} \frac{m(kh^nk^{-1})}{n}.$$
Then by the properties of the homogeneous quasimorphism we get that
$$m(kh^nk^{-1})\in \left[n m(h)+m(k)+m(k^{-1})-2D, nm(h)+m(k)+m(k^{-1})+2D\right],$$ 
and therefore the limit coincides with $m(h)$.
\end{rem}

We now recall the main tool we shall use to construct quasilines that witness the infinite cyclic directions of our extensions.

\begin{lemma}[{\cite[Lemma 4.15]{ABO}}]\label{lem:abo}
    Let $H$ be a finitely generated group, and let $m\colon H\to \mathbb{R}$ be a non-zero homogeneous quasimorphism. Let
$$\tau_{\{m,C\}} =  \left\{g \in G \mbox{ s.t. } |m(g)| < C\right\},$$
where $C$ is any constant such that the defect of $m$ is at most $C/2$, and that there exists a value of $m$ in the interval $(0, C/2)$. Then $\tau_{\{m,C\}}$ is a generating set of $H$, and the map $m\colon \Cay{H}{\tau_{\{m,C\}}} \to \mathbb{R}$ is a quasi-isometry. 
\end{lemma}

Notice that, in the above Lemma, given any finite generating set $S$ for $H$, one can always choose $C$ large enough that $\tau_{\{m,C\}}$ contains $S$.

In \cite{HRSS_3manifold, ELTAG_HHS}, the previous lemma is used as follows: given a $\Z$-central extension $E$ of a hyperbolic group, one can construct a quasimorphism $E\to \mathbb{R}$ which is unbounded on $\Z$ (see e.g. \cite[Lemma 4.3]{ELTAG_HHS}), and therefore Lemma~\ref{lem:abo} grants the existence of a Cayley graph $L$ for $E$ which is a quasiline and where $\Z$ is unbounded. Then $L$ can be used as one of the domains of the HHG structure, which “witnesses” the $\Z$ direction. However, we are interested in short HHG, where a cyclic direction is not necessarily central in the corresponding stabiliser. What is true, however, is that a $\Z$-by-hyperbolic group has an index-two subgroup where $\Z$ is central. This justifies the following lemma:

\begin{lemma}[Extending quasilines to finite-index overgroups]\label{lem:quasiline_extension}
    Let $G$ be a finitely generated group and $E$ a finite-index, normal subgroup. Let $g_1=1,g_2,\ldots, g_k$ be coset representatives for $G/E$. Let $\langle z\rangle\le E$ be an infinite cyclic subgroup which is central in $E$ and normal in $G$. 
    
    Given a homogeneous quasimorphism $m\colon E\to \mathbb{R}$ which is unbounded on $\langle z\rangle$, define $m^G\colon E\to \mathbb{R}$ as
    $$m^G(h)=\frac{1}{[E:G]}\sum_{i=1}^k\varepsilon(g_i) m(g_ihg_i^{-1}),$$
    where $\varepsilon(g)\in\{\pm1\}$ is such that $gz g^{-1}=z^{\varepsilon(g)}$. Then $m^G$ is a homogeneous quasimorphism which is unbounded on $\langle z\rangle$.
    
    Furthermore, let $C>2D$, such that there exists a value of $m^G$ in the interval $(0, C/2)$. Then the inclusion $E\hookrightarrow G$ induces a quasi-isometry
    $$i_m\colon \Cay{E}{\tau_{\{m^G,C\}}}\hookrightarrow \Cay{G}{\tau_{\{m^G,C\}}\cup\{g_1,\ldots, g_k\}},$$
    where $\tau_{\{m^G,C\}}$ is the generating set for $E$ from  Lemma~\ref{lem:abo}. In particular, $\langle z\rangle$ acts loxodromically on $\Cay{G}{\tau_{\{m^G,C\}}\cup\{g_1,\ldots, g_k\}}$.
\end{lemma}

\begin{rem}\label{rem:mG_trivial_on_adjacent_Zw}
    Later we shall also need that, if $e\in E$ is such that $m(geg^{-1})=0$ for every $g\in G$, then clearly $m^G(e)=0$. 
\end{rem}

\begin{proof}
Firstly, each $m(g_i \cdot g_i^{-1})$ is a homogeneous quasimorphism, being the composition of a group automorphism and $m$. Then $m^G$ is a homogeneous quasimorphism, since it is defined as a linear combination of homogeneous quasimorphisms. Furthermore, by construction $m^G(z)=m(z)$, so $m^G$ is itself unbounded on $\langle z\rangle$.

Now we claim that the absolute value of $m^G$ is $G$-invariant. Indeed, for every $g\in G$ and every $h\in E$, we have that
$$m^G(g h g^{-1})=\sum_{i=1}^k\varepsilon(g_i) m(g_ighg^{-1}g_i^{-1}).$$
Let $g_ig=g_{j(i)} h_{j(i)}$, for some $g_{j(i)}$ in the set of representatives and some $h_{j(i)}\in E$. Notice that $i\to j(i)$ is a permutation of the set $\{1, \ldots, k\}$, as $g_ig$ and $g_jg$ cannot lie in the same coset of $E$. Furthermore $\varepsilon(g_i)=\varepsilon(g)\varepsilon(g_{j(i)})$, as $h_{j(i)}$ commutes with $z$. Hence
$$m^G(g h g^{-1})=\varepsilon(g)\sum_{i=1}^k\varepsilon(g_{j(i)}) m(g_{j(i)}h_{j(i)}hh_{j(i)}^{-1}g^{-1}g_{j(i)})=$$
$$=\varepsilon(g)\sum_{i=1}^k\varepsilon(g_{j(i)}) m\left((g_{j(i)}h_{j(i)}g_{j(i)}^{-1})(g_{j(i)}hg_{j(i)}^{-1})(g_{j(i)}h_{j(i)}^{-1}g_{j(i)}^{-1})\right).$$
Now $g_{j(i)}h_{j(i)}^{-1}g_{j(i)}^{-1}\in E$, as $E$ is normal in $G$, and using that $m$ is invariant under conjugation by elements of $E$ we get that
$$m^G(g h g^{-1})=\varepsilon(g)\sum_{i=1}^k\varepsilon(g_{j(i)}) m\left(g_{j(i)}hg_{j(i)}^{-1}\right)=\varepsilon(g)m^G(h).$$

Now choose $C$ as above. Clearly $i_m$ is $1$-Lipschitz and coarsely surjective, so in order to show that it is a quasi-isometry we shall prove that there exists a coarsely Lipschitz retraction 
$$r_m\colon \Cay{G}{\tau_{\{m^G,C\}}\cup\{g_1,\ldots, g_k\}}\to \Cay{E}{\tau_{\{m^G,C\}}}.$$
For every $g\in G$, we can write it uniquely as $g=hg_i$ for some $h\in E$ and $g_i$ in the chosen collection of coset representatives. Then we set $r_m(g)=h$, which is the identity on $E$. 

Now suppose $g, g'\in G$ are adjacent in $\Cay{G}{\tau_{\{m^G,C\}}\cup\{g_1,\ldots, g_k\}}$. We can write $g=hg_i$ as above. If $g'=gg_j$ for some coset representative $g_j$ then, setting $g_ig_j=h_{ij}g_{l}$ for some $h_{ij}\in E$ and $g_j$ in the set of representatives, we have that $r_m(g)=h$ and $r_m(g')=hh_{ij}$ differ by $h_{ij}$, chosen in a finite subset of $E$, and in particular their distance is uniformly bounded.

If instead $g'=gh'$ for some $h'\in \tau_{\{m^G,C\}}$, then $g'=h(g_i h' g_i^{-1}) g_i$, so $r_m(g)$ and $r_m(g')$ differ by $g_i h' g_i^{-1}$, which still belongs to $\tau_{\{m^G,C\}}$ as the absolute value of $m^G$ is $G$-invariant.
\end{proof}

\subsection{Blowup materials}

\begin{defn}\label{defn:weak_hyp}
    Given a group $G$, a generating set $S$ for $G$, and a collection of subgroups $\{\Lambda_1,\ldots, \Lambda_k\}$, the \emph{coned-off Cayley graph} $\h G$ is the graph obtained from the Cayley graph $\Cay{G}{S}$ by adding a vertex for every coset of $\Lambda_i$, for $i=1,\ldots, k$, and declaring that the link of such vertex is the corresponding coset. 
    
    A finitely generated group $G$ is \emph{weakly hyperbolic} relative to a collection  of conjugacy classes of subgroups if, given any generating set $S$ for $G$, and any choice of a subgroup $H_i$ in each conjugacy class, the coned-off Cayley graph of $G$ with respect to the collection $\{H_i\}$ is hyperbolic.
\end{defn}

\begin{defn}
    Let $G$ be a finitely generated group, and let $\mathcal{H}$ be a collection of finitely generated subgroups. A \emph{compatible} generating set for $(G,\mathcal{H})$ is a finite generating set $\tau$ for $G$ such that, for every $H\in \mathcal{H}$, $\tau\cap H$ generates $H$. 
\end{defn}

\begin{notation}\label{notation:pre_def_squid}
In what follows, $G$ is a group acting on a simplicial graph $\ov X$ with finitely many orbits of edges. Fix $V=\{v_1,\ldots, v_k\}$ a set of representatives of the $G$-orbits of vertices in $\ov{X}$. For every $v_i\in V$ and every $g\in G$, set $P_{gv_i}=g\Stab{G}{v_i}$, which we call the \emph{product region} associated to $gv_i$ (with respect to this choice of orbit representatives). We always see $P_v$ as a metric subspace of $G$, where the latter is equipped with the word metric $\dist_G$ with respect to any generating set which is compatible with the $\Stab{G}{v_i}$. 

For every $v_i\in V$ fix a collection $\{h_i^1 v_{i(1)}, \ldots h_i^l v_{i(l)}\}$ of representatives of the $G$-orbits of vertices of $\link_{\ov X}(v_i)$, where every $h_i^j$ belongs to $G$. Whenever the dependence of some $h_i^j v_{i(j)}$ on $i$ and $j$ is irrelevant, we denote $h_i^j$ by $h$ and $v_{i(j)}$ by $v'$. This way, every $w\in\link_{\ov X}(v_i)$ can be expressed as $w=ghv'$, for some $g\in \Stab{G}{v_i}$. Later we shall also need the following constant: $$r\coloneq\max_{i,j}\left|h_i^j\right|=\max_{i,j}\left|(h_i^j)^{-1}\right|,$$ where $|\cdot|$ denotes the norm in the word metric $\dist_G$.
\end{notation}

\begin{defn}[Blowup materials]\label{defn:squid_material}
The following data define \emph{blowup materials} for a finitely generated group $G$:
    \begin{enumerate}
        \item\label{squid_material:graph} $G$ acts cocompactly on a simplicial graph $\ov{X}$, called the \emph{support graph}, which is triangle- and square-free, and such that no connected component of $\ov{X}$ is a single point.
        \item\label{squid_material:extensions}
        For every $v\in \ov X^{(0)}$, its stabiliser is an extension
        $$\begin{tikzcd}
        0\ar{r}&Z_v\ar{r}&\Stab{G}{v}\ar{r}{\p_v}&H_v\ar{r}&0
        \end{tikzcd}$$
        where $H_v$ is a finitely generated hyperbolic group and $Z_v$ is a cyclic, normal subgroup of $\Stab{G}{v}$ which acts trivially on $\link_{\ov X}(v)$. The family of such extensions is equivariant with respect to the $G$-action by conjugation.
        \item\label{squid_material:edge_stab} Whenever $e=\{v,w\}$ is an edge of $\ov{X}$, $\text{P}\Stab{G}{e}\coloneq\Stab{G}{v}\cap\Stab{G}{w}$ contains $\langle Z_v, Z_w\rangle $ as a finite index subgroup. Moreover, $\p_v(Z_w)$ is quasiconvex in $H_v$.
        \item\label{squid_material:big_papa} $G$ is weakly hyperbolic relative to $\{\Stab{G}{v_i}\}_{v_i\in V}$.
        \item\label{squid_material:quasimorphisms}
        For all $v_i\in V$ for which $Z_{v_i}$ is infinite, there exists a finite-index, normal subgroup $E_{v_i}$ of $\Stab{G}{v_i}$, containing $Z_{v_i}\cap E_{v_i}$ in its centre. Furthermore, there is a homogeneous quasimorphism
        $$\phi_{v_i}\colon E_{v_i}\to \mathbb{R},$$
        which is unbounded on $Z_{v_i}\cap E_{v_i}$ and trivial on $Z_w\cap E_{v_i}$ for every vertex $w\in\link_{\ov X}(v_i)$. If, instead, $Z_{v_i}$ is finite, we set $E_{v_i}=\Stab{G}{v_i}$ and $\phi_{v_i}\equiv 0$.
        \item\label{squid_material:gates}
        There exist a constant $B\ge 0$ and, for every $v_i\in V$, a coarsely Lipschitz, coarse retraction $$\gate_{v_i}\colon G\to 2^{E_{v_i}},$$
        which we call \emph{gate}. We require that, whenever $w\in\link_{\ov X}(v_i)$, $$\gate_{v_i}(P_w)\subseteq N_B(P_w).$$ Furthermore, whenever $\dist_{\ov X}(v_i, u)\ge 2$, there exist $g\in\Stab{G}{v_i}$, $h\in G$, and $v'\in V$, as in Notation~\ref{notation:pre_def_squid}, such that 
        $$\gate_{v_i}(P_u)\subseteq N_B(gZ_{hv'}).$$
        In particular, if one sets $w(u)=ghv'$, then
        $$\gate_{v_i}(P_u)\subseteq N_B(gh\Stab{G}{v'}h^{-1})\subseteq N_{B+r}(gh\Stab{G}{v'})=N_{B+r}(P_{w(u)}).$$
    \end{enumerate}
\end{defn}

\subsection{From blowup materials to a short HHG}\label{subsec:squid_to_short}
\begin{thm}\label{thm:squidification}
    Let $G$ be a finitely generate group admitting blowup materials, with support graph $\ov X$. Then $(G, \ov X, \W)$ is a short HHG, where the cyclic direction associated with each $v\in\ov{X}^{(0)}$ is (a finite-index subgroup of) $Z_v$.
\end{thm}

\begin{proof}[Outline of the proof] In Section~\ref{subsec:construction_XW}, see in particular Definitions~\ref{defn:blowup_for_squidification} and~\ref{defn:W-edges}, we construct a pair $(X,\W)$, which we then prove to be a combinatorial HHG. Here is where each axiom from Definition~\ref{defn:combinatorial_HHS} is verified;
    \begin{itemize}
        \item Definition~\ref{defn:combinatorial_HHS}.\eqref{item:chhs_flag} is Corollary~\ref{cor:finite_complexity_HHS_blow_up};
        \item Definition~\ref{defn:combinatorial_HHS}.\eqref{item:chhs_delta} is split between Subsections~\ref{subsec_link_hyp} and~\ref{subsec:link_qi_emb};
        \item Definition~\ref{defn:combinatorial_HHS}.\eqref{item:chhs_join} is Lemma~\ref{lem:simplicial_wedge_property};
        \item Definition~\ref{defn:combinatorial_HHS}.\eqref{item:C_0=C} is Lemma~\ref{lem:edges_in_link}.
    \end{itemize}
    In Remarks~\ref{rem:G_action_on_X} and~\ref{rem:G_action_on_W} we check that $G$ acts on $X$ with finitely many orbits of links, and the action extends to $\W$. Moreover, in Definition~\ref{defn:realisation} we construct a $G$-equivariant map $f\colon \W\to G$, which we prove to be a quasi-isometry in Lemma~\ref{lem:f_qi}. This in particular implies that $G$ acts geometrically on $\W$, and therefore it is a combinatorial HHG. Finally, in Subsection~\ref{subsec:check_short_for_squid} we check that the $G$-action on $(X,\W)$ satisfies the axioms of a short HHG.
\end{proof}

\subsubsection{The candidate combinatorial structure}\label{subsec:construction_XW}

\begin{rem}
It is easy to check that, if we replace each $Z_v$ with a subgroup $Z_v'$ of finite index (in a $G$-equivariant way), then the new extensions
$$\begin{tikzcd}
0\ar{r}&Z_v\ar{r}&\Stab{G}{v}\ar{r}&\Stab{G}{v}/Z_v'\ar{r}&0
\end{tikzcd}$$
still satisfy all properties of Definition~\ref{defn:squid_material}. Thus, as we stated Theorem~\ref{thm:squidification} to allow this kind of replacements, we can (and will) assume that: 
\begin{itemize}
\item every finite $Z_v$ is trivial;
\item $Z_v\le E_u$ for every $u\in\operatorname{Star}(v)$; in particular, $Z_v$ and $Z_w$ commute whenever $v,w$ are adjacent in $\ov X$.
\end{itemize}
\end{rem}

\begin{defn}[Quasilines]\label{defn:quasilines_from_quasimorph_for_squid} 
For every $v_i\in V$, if $Z_{v_i}$ is trivial let $$L_{v_i}=\Cay{\Stab{G}{v_i}}{\Stab{G}{v_i}},$$ that is, the complete graph on $\Stab{G}{v_i}$. If instead $Z_{v_i}$ is infinite, we can apply Lemma~\ref{lem:quasiline_extension} to $\langle z\rangle=Z_{v_i}$, $E=E_{v_i}$, $G=\Stab{G}{v_i}$, $m=\phi_{v_i}$, and a suitable choice of the constant $C$.  This yields a generating set $\lambda_i$ for $\Stab{G}{v_i}$ such that $$L_{v_i}\coloneq \Cay{\Stab{G}{v_i}}{\lambda_i}$$ is a quasiline, on which $Z_{v_i}$ acts geometrically.

For every $g\in G$ and every $v=gv_i$, let $L_v=gL_{v_i}$ be a parallel copy of $L_{v_i}$, on which $\Stab{G}{v}=g\Stab{G}{v_i}g^{-1}$ acts. 
\end{defn}

\begin{rem}[Bounded orbits on $L_v$]\label{rem:bounded_orbits_Lv}
    Let $v_i\in V$ have infinite cyclic direction. Notice that, for every $g\in \Stab{G}{v_i}$ and every $w\in\link_{\ov X}(v_i)$, $gZ_wg^{-1}=Z_{gw}$, on which $\phi_{v_i}$ is trivial. Then by Remark~\ref{rem:mG_trivial_on_adjacent_Zw} each $Z_w$ lies in the generating set $\lambda_i$. In turn, this means that every $Z_w$-orbit has diameter $1$ in $L_{v_i}$ (to see this, notice that, if $z\in Z_w$ and $x\in \Stab{G}{v_i}$, then $|\phi^G_{v_i}(x^{-1}zx)|=|\phi^G(z)|=0$, so in particular $\dist_{L_{v_i}}(x,zx)=\dist_{L_{v_i}}(1,x^{-1}zx)=1$). 

Furthermore, recall that, whenever $\dist_{\ov{X}}(v_i,u)\ge 2$, the projection $\gate_{v_i}(P_u)$ is contained in the $B$-neighbourhood of some $Z_{w(u)}$-orbit, where $w(u)$ is some vertex in the link of $v_i$ and $B$ is the constant from Definition~\ref{defn:squid_material}.\eqref{squid_material:gates}. Hence $\gate_{v_i}(P_u)$ has diameter at most $2B+1$ in $L_v$.
\end{rem}

\begin{rem}\label{rem:Pv_to_Lv}
For every $v=gv_i\in \ov{X}^{(0)}$, the word metric $\dist_{g\Stab{G}{v_i}}$ on the coset $P_v=g\Stab{G}{v_i}$, coming from the chosen generating set for $\Stab{G}{v_i}$, is quasi-isometric to the inherited distance $\dist_G$, as $G$ has a coarsely Lipschitz, coarse retraction on $P_v$. 

Furthermore, consider the $\Stab{G}{v}$-equivariant projection $$\Proj_v\colon P_v\to L_v,$$ induced by the identity on the elements of $g\Stab{G}{v_i}$. This map is coarsely Lipschitz with respect to $\dist_{g\Stab{G}{v_i}}$, as $L_{v_i}=\Cay{\Stab{G}{v_i}}{\lambda_i}$ and $\lambda_i$ contains a finite generating set (up to choosing a larger $C$ in the definition of $L_{v_i}$); therefore $\Proj_v$ is also coarsely Lipschitz when $P_v$ is equipped with $\dist_G$. 
\end{rem}

\begin{defn}\label{defn:blowup_for_squidification} Let $X$ be the blowup of $\ov{X}$ with respect to the family $\{L_v\}_{v\in\ov{X}^{(0)}}$, as in Definition~\ref{defn:blowup}.
\end{defn}

Let $p\colon X\to \ov X$ be the retraction mapping every cone to its tip. Recall that Lemmas~\ref{lem:decomposition_of_links} and~\ref{cor:bounded_links} give descriptions of the possible shapes of a simplex $\Delta$ and its link, in terms of its support $\ov \Delta=p(\Delta)$ and the intersections with the cones $\Delta_v=\Delta\cap\Squid(v)$.

\begin{rem}[$G$-action on $X$]\label{rem:G_action_on_X}
    Notice that the $G$-action on $\ov X$ extends to an action on $X$ by simplicial automorphisms. Indeed, if $v=gv_i$ then $L_v$ is a Cayley graph of $g\Stab{G}{v_i}$, and one can set $h\cdot x=hx\in L_{hv}$ for every $h\in G$ and $x\in L_v$. 

    Furthermore, by Lemma~\ref{lem:decomposition_of_links} the link of a simplex $\Delta$ of $X$ is uniquely determined by the support $\ov \Delta$ and by the subcomplexes $\link_{\Squid(v)}(\Delta_v)$ for every $v\in \ov \Delta$. Since there are finitely many $G$-orbits of supports (as $G$ acts cocompactly on $\ov{X}$), and since the “moreover” part of Lemma~\ref{lem:decomposition_of_links} tells us that there are only three possibilities for $\link_{\Squid(v)}(\Delta_v)$, we get that the $G$-action on $X$ has finitely many orbits of subcomplexes of the form $\link(\Delta)$.
\end{rem}

\begin{defn}[Coarse level sets]
    Fix a constant $R\ge 0$. For every $x\in L_v$, let 
$$N(x)=N_{R}\left(\left\{g\in P_v\,|\,\dist_{L_v}(x, \Proj_{v}(g))\le R\right\}\right).$$
In other words, $N(x)$ is a thickening (in $G$) of a “level set” of the projection $\Proj_{v}$.
\end{defn}

\begin{defn}[Realisation]\label{defn:realisation}
Given a maximal simplex $\Delta=\Delta(x,y)$, we define its \emph{realisation} as
$$f(\Delta)=N(x)\cap N(y).$$
Notice that, by construction, for every $g\in G$ and every $\Delta(x,y)$, we have that 
$$gf(\Delta(x,y))=f(\Delta(gx,gy)),$$
i.e. realisation is defined in a $G$-equivariant way. \end{defn}

\begin{lemma}\label{lem:realisation_exists} There exists $R_0\in\mathbb{R}$, depending only on the blowup materials, such that the following holds if $R> R_0$. 

For every maximal simplex $\Delta$, supported on the edge $e=\{v,w\}$ of $\ov{X}$, its realisation $f(\Delta)$ is non-empty, and bounded in terms of $R$. Moreover, the coarse map $$f\colon L_v\times L_v\to G,$$ sending $(x,y)$ to $f(\Delta(x,y))$, is a $\text{P}\Stab{G}{e}$-equivariant quasi-isometric embedding whose constants depend on $R$, and the Hausdorff distance between its image and the \emph{edge product region} $P_e:=N_R(P_v)\cap N_R(P_w)$ is bounded in terms of $R$.
\end{lemma}

\begin{proof}
All the coarse level sets and maps involved in the statement are defined in a $G$-equivariant way. Thus, without loss of generality, we can assume that $v=v_i\in V$ and $w=h v'$ as in Notation~\ref{notation:pre_def_squid}. First, we show that the edge product region coarsely coincides with $\langle Z_v, Z_w\rangle$:

\begin{claim}\label{claim:Pe_coarsely_square} Let $R$ be any constant greater than $r$. There exists $C_0\ge 0$, depending only on $R$ and $G$, such that $P_e$ is within Hausdorff distance at most $C_0$ from $\langle Z_v, Z_w\rangle$.
\end{claim}

\begin{claimproof}[Proof of Claim~\ref{claim:Pe_coarsely_square}]
Since $r\ge |h^{-1}|$, we have that 
$$P_e=N_R(P_v)\cap N_R(P_w)=N_R\left(\Stab{G}{v}\right)\cap N_R\left(h\Stab{G}{v'}\right) \supseteq $$
$$\supseteq\Stab{G}{v}\cap \left(h\Stab{G}{v'} h^{-1}\right)=\Stab{G}{v}\cap \Stab{G}{w},$$
where we used that $R\ge r\ge |h^{-1}|$. Conversely, by \cite[Lemma 4.5]{Hruska_Wise} there exists a constant $R'$, depending only on the generating set for $G$, $R$, and the finitely many choices for $v=v_i$ and $w=h_i^jv_{i(j)}$, such that 
$$P_e\subseteq N_{R'}\left(\Stab{G}{v}\cap h\Stab{G}{v'}h^{-1}\right)=N_{R'}\left(\Stab{G}{v}\cap \Stab{G}{w}\right).$$
This proves that $P_e$ and  $\Stab{G}{v}\cap \Stab{G}{w}$ are within Hausdorff distance at most $R'$, depending on $R$ and $G$. Moreover, $\Stab{G}{v}\cap \Stab{G}{w}$ contains $\langle Z_v,Z_w\rangle$ as a subgroup of finite index, and thus the Hausdorff distance between $P_e$ and $\langle Z_v,Z_w\rangle$ is bounded by some constant $C_0$, depending on $G$ and $R$.
\end{claimproof}

Now we turn to the proof of Lemma~\ref{lem:realisation_exists}. Let $\Delta=\Delta(x,y)$, and let $1\in G$ be the identity element. As $1\in P_v=\Stab{G}{v}$ and $h\in P_w= h\Stab{G}{v'}$, we can define
$$\Phi\colon \langle Z_v, Z_w\rangle \to L_v\times L_w$$
by mapping $g\in \langle Z_v, Z_w\rangle$ to 
$$( g\Proj_v(1), g\Proj_w(h))= (\Proj_{gv}(g), \Proj_{gw}(gh))=(\Proj_v(g), \Proj_w(gh)),$$
where we used that $\langle Z_v, Z_w\rangle\le \Stab{G}{v}\cap \Stab{G}{w}$.
This map is a quasi-isometry, as $\langle Z_v, Z_w\rangle $ acts geometrically on $ L_v\times L_w$ (in turn, this is because each cyclic direction acts geometrically on the associated quasiline and with uniformly bounded orbits on the other). Furthermore, the constants of the quasi-isometry only depend on the blowup materials, and more precisely on the constants of the homogeneous quasimorphisms we used to define the projections. We can then find a constant $R_0\ge r$, again independent on $R$, such that $\Phi$ has $R_0$-dense image. Thus, if $R> R_0$ and if $I_R(x,y)$ denotes the product of the ball of radius $R$ around $x$ in $L_v$ and the ball of radius $R$ around $y$ in $L_w$, we have that $\Phi^{-1}(I_R(x,y)\cap\text{im}\Phi)$ is a non-empty subset of  $\langle Z_v, Z_w\rangle$. Moreover, by construction, every $g\in \Phi^{-1}(I_R(x,y)\cap\text{im}\Phi)$ belongs to $f(\Delta(x,y))$. Indeed, $g\in N(x)$, as $g\in P_v$ projects $R$-close to $x$ in $L_v$; furthermore, $g\in N(y)$ since it is at distance $r\le R$ from $gh\in P_w$ which projects $R$-close to $y$. This proves that $f(\Delta(x,y))$ is non-empty.

\par\medskip
    Next, we argue that $f(\Delta(x,y))$ is bounded in terms of $R$. Pick any $k\in f(\Delta(x,y))$. Since $f(\Delta)\subseteq P_e$, there is some $g\in \langle Z_v, Z_w\rangle\le P_v$ such that $\dist_G(k,g)\le C_0$. Moreover, let $a\in P_v$ such that $\dist_G(k,a)\le R$ and $\dist_{L_v}(\Proj_v(a), x)\le R$. Now, both $g$ and $a$ belong to $P_v$, and we have that $\dist_G(g,a)\le C_0+R$. By Remark~\ref{rem:Pv_to_Lv}, their projections to $L_v$ are uniformly close in terms of $R$, and therefore the distance between $\Proj_v(g)$ and $x$ is bounded in terms of $R$. The same is true for $\Proj_w(gh)$ and $y$: there is some $b\in P_w$ such that $\dist_G(k,b)\le R$ and $\dist_{L_v}(\Proj_w(b), y)\le R$, and then one uses that $gh\in P_w$ and $\dist_{G}(gh,b)\le r+C_0+R$. Combining the two facts, we get that $g$ belongs to the bounded set $\Phi^{-1}(I_{C_1}(x,y))$, for some radius $C_1$ depending on $R$. Then in turn $f(\Delta(x,y))$ lies in the $C_0$-neighbourhood of $\Phi^{-1}(I_{C_1}(x,y)I_{C_1}(x,y))$, and is therefore uniformly bounded.

\par\medskip
    Regarding the “moreover” part of the Lemma, we first construct a map $$f'\colon L_v\times L_w\to \langle Z_v,Z_w\rangle$$ by setting $f'(x,y)= f(\Delta(x,y))\cap  \langle Z_v,Z_w\rangle$. As a consequence of the arguments above, $f'$ is a well-defined coarse map, and it coincides with $f$ up to a uniform error. Then the conclusion follows if we show that $f'$ is a quasi-inverse for $\Phi$. Indeed, we noticed above that, for every $g\in \langle Z_v, Z_w\rangle$, $$f'(\Phi(g))=f(\Delta(\Proj_v(g), \Proj_w(gh)))\cap \langle Z_v,Z_w\rangle$$ is a uniformly bounded set containing $g$. Conversely, $\Phi(f'(x,y))$ uniformly coarsely coincides with $(\Proj_v(g),\Proj_w(gh))$ for any $g\in f(\Delta(x,y))\cap \langle Z_v,Z_w\rangle$. But then the distance between $\Proj_v(g)$ and $x$ (resp. $\Proj_w(gh)$ and $y$) is bounded in terms of $R$.
\end{proof}

Before proceeding, let us point out some easy consequences of the arguments in the above proof:
\begin{cor}\label{cor:g_in_edge_prod_region} The following holds if $R\ge R_0$. For every $v\in \ov{X}^{(0)}$ and every $g\in P_v$ there exists $w\in \link_{\ov X}(v)$ such that $g\in N_R(P_v)\cap N_R(P_w)$. 
\end{cor}

\begin{proof}
    Again, up to the $G$-action it is enough to prove the Corollary for $v\in V$, so that $P_v=\Stab{G}{v}$. Choose any $hv'$ in the set of representatives of all $G$-orbits of vertices of $\link_{\ov X}(v)$, as in Notation~\ref{notation:pre_def_squid}. Then
    $$g\in g\left(\Stab{G}{v}\cap h\Stab{G}{v'}h^{-1}\right)\subseteq  \Stab{G}{v}\cap N_R\left(gh\Stab{G}{v'}\right)=P_v\cap N_R(P_{ghv'}),$$
    where we used that we chose $R\ge r\ge |h^{-1}|$.
    Hence, it suffices to set $w=gh v'$.
\end{proof}

\begin{cor}\label{cor:close_regions} The following holds if $R\ge R_0$. If $v,w\in\ov{X}^{(0)}$ are $\ov{X}$-adjacent then $\dist_G(P_v, P_w)\le 2R$. 
\end{cor}

\begin{proof}
Just notice that $N_R(P_v)\cap N_R(P_w)$ is always non-empty, as it contains $f(\Delta(x,y))$ for any choice of $x\in L_v$ and $y\in L_w$.
\end{proof}

\begin{defn}[$\W$-edges]\label{defn:W-edges}
Fix a constant $T\ge 0$. Let $\W$ be the graph whose vertices are maximal simplices of $X$, and where two simplices $\Delta=\Delta(x,y)$  and $\Delta'=\Delta(x',y')$ are $\W$-adjacent if and only if one of the following holds:
    \begin{itemize}
        \item \textbf{Type 1 (close realisations)}: $\dist_{G}(f(\Delta), f(\Delta'))\le 1$.
        \item \textbf{Type 2 (staple edges)}: $x=x'$ and $\dist_{G}(N(y), N(y'))\le T+1$.
    \end{itemize}
\end{defn}

\begin{rem}[$G$-action on $\W$]\label{rem:G_action_on_W}
    The $G$-action on $X$ induces an action on $\W$. Indeed, $G$ maps maximal simplices of $X$ to maximal simplices; moreover, $\W$-edges are defined in a $G$-equivariant way, since they depend on the distance in $G$ between (intersections of) coarse level sets of the projections, which are $G$-equivariant.
\end{rem}

\begin{lemma}\label{lem:W-edge_close} The following holds if $R> \max\{R_0, B\}$. There exists a constant $\widetilde{K}\ge0$, depending on $R$ and $T$, such that, if $\Delta,\Delta'$ are two $\W$-adjacent maximal simplices, then $\dist_{G}(f(\Delta), f(\Delta'))\le \widetilde{K}$.
\end{lemma}
\begin{proof} Throughout the proof, we shall say that a bound is \emph{uniform} if it only depends on $R$ and $T$. 

Let $\Delta,\Delta'$ be two $\W$-adjacent simplices. If they have close realisations, then by definition $\dist_{G}(f(\Delta), f(\Delta'))\le 1$. Thus suppose $\Delta=\Delta(x,y)$ and $\Delta'=\Delta(x,y')$ are joined by a staple edge. Let $v=p(x)$, $w=p(y)$, and $w'=p(y')$. There exist $g\in P_w$, $g'\in P_{w'}$ which are $(T+2R+1)$-close, and such that $\dist_{L_w}(\Proj_w(g), y)\le R$ and similarly $\dist_{L_{w'}}(\Proj_{w'}(g'), y')\le R$. Let $h=\gate_v(g)$ and $h'=\gate_v(g')$, which are uniformly close as $\gate_v$ is coarsely Lipschitz. Moreover, $\gate_v(P_w)\subseteq N_B(P_w)$ by Definition~\ref{defn:squid_material}.\eqref{squid_material:gates}; hence there exists $k\in P_w$ such that $\dist_G(h,k)\le B\le  R$. This means that $h\in f(\Delta(\Proj_v(h), z))$ where $z=\Proj_w(k)$. Similarly, one can find an element $k'\in P_{w'}$ such that $\dist_G(h',k')\le R$, and so $h'\in f(\Delta(\Proj_v(h'), z'))$ where $z'=\Proj_w(k')$. The situation is depicted in Figure~\ref{fig:countless_projections}. Notice that, as $\dist_G(h,h')$ is uniformly bounded, then so is $\dist_G(k,k')$. 

\begin{figure}[htp]
    \centering
    \includegraphics[width=\linewidth, alt={The various points and projections involved in the construction of the Lemma}]{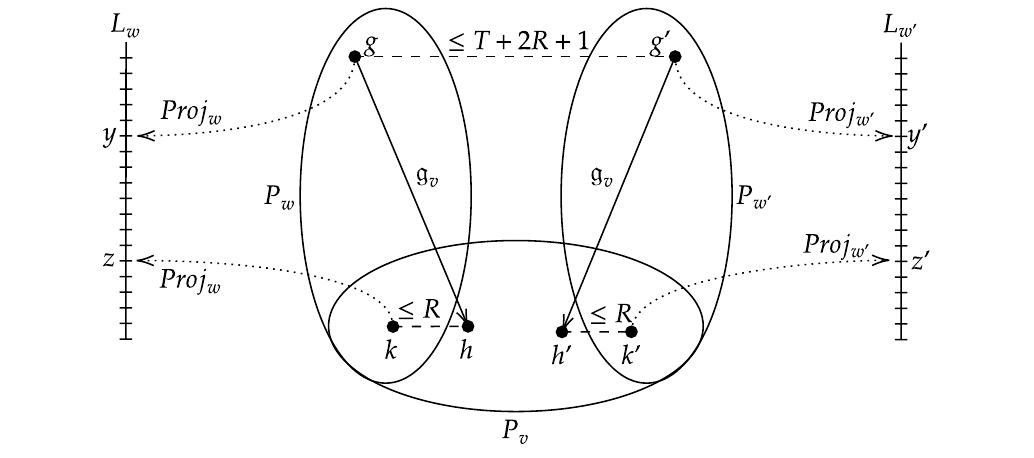}
    \caption{The various points and projections involved in the construction from Lemma~\ref{lem:W-edge_close}. A dashed line between two point means that their distance is bounded in terms of $R$ and $T$. The goal is to prove that $y$ and $z$ (resp. $y'$ and $z'$) are uniformly close in the quasiline $L_w$ (resp. $L_{w'}$).}
    \label{fig:countless_projections}
\end{figure}

    Now we claim the following:
    \begin{claim}\label{claim:W-edge_close_gate_does_not_change}
        Both $\dist_{L_{w}}(y, z)$ and $\dist_{L_{w'}}(y', z')$ are uniformly bounded.
    \end{claim}
    Let us assume Claim~\ref{claim:W-edge_close_gate_does_not_change} for a moment, and show how it implies Lemma~\ref{lem:W-edge_close}. We know that $h\in f(\Delta(\Proj_v(h), z))$ and $h'\in f(\Delta(\Proj_v(h'), z'))$ are uniformly close. Furthermore, combining the Claim with the “moreover” part of Lemma~\ref{lem:realisation_exists}, we get that the distance between $f(\Delta(\Proj_v(h), y))$ and $f(\Delta(\Proj_v(h), y'))$ is also uniformly bounded. Now, one can find an element $t\in Z_v$ such that $t\Proj_v(h)$ is uniformly close to $x$ (as $Z_v$ acts geometrically on $L_v$) while $ty$ and $ty'$ are still uniformly close to $y$ and $y'$, respectively (as $Z_v$ acts with uniformly bounded orbits on both $L_w$ and $L_{w'}$). Then the distance between $f(\Delta(x, y))$ and $f(\Delta(x, y'))$ will also be bounded by some uniform constant $\widetilde{K}$, as required. 
\end{proof}

    \begin{claimproof}[Proof of Claim~\ref{claim:W-edge_close_gate_does_not_change}]
        We prove that $\dist_{L_{w}}(y, z)$ is uniformly bounded, as the symmetrical statement follows analogously. 
        
        We will repeatedly use that, if the distance between two subsets of $P_w$ is uniformly bounded, then so is the distance between their projections to $L_w$, as the map $\Proj_w\colon P_w\to L_w$ is coarsely Lipschitz by Remark~\ref{rem:Pv_to_Lv}. We have the following chain, where $A\sim B$ denotes that the subsets $A$ and $B$ of $L_w$ are uniformly close:
        $$z=\Proj_w(k)\sim \Proj_w(\gate_w(k)).$$
        In the second passage we used that $\gate_w$ is a coarse retraction, that is, it coarsely coincides with the identity on $P_w$. Furthermore, as $\gate_w$ is coarsely Lipschitz and $\dist_G(k,k')$ is uniformly bounded, we have that
        $$\Proj_w(\gate_w(k))\sim \Proj_w(\gate_w(k')).$$
        But now $k'\in P_{w'}$, and $\Proj_w(P_{w'})$ has diameter less than $2B+1$ in $L_w$ by Remark~\ref{rem:bounded_orbits_Lv}. In particular, we can replace $k'$ with $g'$:
        $$\Proj_w(\gate_w(k'))\sim \Proj_w(\gate_w(g'))\sim \Proj_w(\gate_w(g))\sim \Proj_w(g)\sim y,$$
        where again we used that $\dist_G(g, g')$ is uniformly bounded, and that $g\in P_w$ is uniformly close to its gate $\gate_w(g)$. 
    \end{claimproof}

We shall prove that $(X,\W)$ is a combinatorial HHS, under the following choice of constants for the construction of $\W$:
\begin{notation}\label{notation:constants_squidification}
Let $R> \max\{R_0,1, B+r\}$, where:
\begin{itemize}
    \item $R_0$ is the constant from Lemma~\ref{lem:realisation_exists},
    \item $B$ is the constant from Definition~\ref{defn:squid_material}.\eqref{squid_material:gates}, and
    \item $r$ is the constant from Notation~\ref{notation:pre_def_squid}.
\end{itemize}
Choose $T$ such that, for every $v\in \ov{X}^{(0)}$, if two points $g,g'\in G$ are $2R+1$-close, then their gates $\gate_v(g), \gate_v(g')$ are $T$-close. Such $T$ exists as gates are uniformly coarsely Lipschitz. 
\end{notation}

\subsubsection{Finite complexity and intersection of links}\label{subsubsec:finite_complexity_and_int_link_squid}
\begin{lemma}[Intersection of links]\label{lem:intersection_links}
    Let $\Sigma,\Delta$ be non-maximal simplices of $X$. Then there exist two (possibly empty or maximal) simplices $\Pi,\Psi\subseteq X$ such that $\Sigma\subseteq \Pi$ and $$\link(\Sigma)\cap\link(\Delta)=\link(\Pi)\star\Psi.$$
\end{lemma}

\begin{proof}
There are three cases to consider, depending on the supports $\ov\Sigma$ and $\ov \Delta$.

If $\Sigma$ is empty, then the results clearly holds with $\Pi=\Delta$ and $\Psi=\emptyset$.
\par\medskip
If $\ov\Sigma=\{v,w\}$ is an edge, then $\link(\Sigma)=\link_{\Squid(v)}(\Sigma_v)\star \link_{\Squid(w)}(\Sigma_w) $, by Corollary~\ref{cor:bounded_links}. Moreover, $\link_{\Squid(v)}(\Sigma_v)\cap \link(\Delta)$ either coincides with $\link_{\Squid(v)}(\Sigma_v)$ or is trivial, and similarly for $w$. Then let $\Pi$ be a simplex obtained from $\Sigma$ by completing $\Sigma_v$ to an edge if $\link_{\Squid(v)}(\Sigma_v)\cap \link(\Delta)=\emptyset$, and similarly completing $\Sigma_w$ to an edge if necessary. By construction, we have that
$\link(\Sigma)\cap\link(\Delta)=\link(\Pi)$. 

\par\medskip
Finally, suppose that $\ov\Sigma=\{v\}$ is a single vertex. We must then look at how $\link_{\ov X}(\ov\Sigma)$ interacts with $\ov\Delta$. If $\ov\Delta\cap\link_{\ov X}(\ov \Sigma)$ is non-trivial, we set $\ov\Phi=\ov\Delta\cap\link_{\ov X}(\ov \Sigma)$. If $\link_{\ov X}(\ov \Sigma)\cap \link_{\ov X}(\ov \Delta)$ is a single vertex (which happens if both $\ov\Sigma$ and $\ov\Delta$ are single vertices at distance $2$), then set $\ov\Psi=\link_{\ov X}(\ov \Sigma)\cap \link_{\ov X}(\ov \Delta)$. If none of the previous is true, then $\link_{\ov X}(\ov \Sigma)\cap \link_{\ov X}(\ov \Delta)$ is trivial, and we set $\ov \Theta=\{w\}$ where $w$ is any vertex inside $\link_{\ov X}(\ov\Sigma)$. Notice that, by construction, exactly one between $\ov\Phi$, $\ov\Psi$ and $\ov\Theta$ is non-empty.

Now let $\Pi$ be the simplex defined as follows:
\begin{itemize}
    \item $\ov\Pi=\ov\Sigma\star\ov\Phi\star\ov\Psi\star\ov\Theta$.
    \item If $v\in \ov\Sigma$ does not belong to $( \ov\Delta\cup\link_{\ov X}(\ov\Delta))$ then $\Pi_v$ is an edge containing $\Sigma_v$, so that $\link_{\Squid(v)}\left(\Pi_v\right)=\emptyset$;
    \item If $v\in \ov\Sigma\cap \link_{\ov X}(\ov\Delta)$ then $\Pi_v= \Sigma_v$;
    \item If $v\in \ov\Sigma\cap \ov\Delta$ then $\Pi_v$ is an edge containing $\Sigma_v$ if $\link_{\Squid(v)}(\Sigma_v)\cap \link(\Delta)=\emptyset$; otherwise $\Pi_v=\Sigma_v$. In other words, we choose $\Pi_v$ so that 
    $$\link_{\Squid(v)}\left(\Pi_v\right)=\link_{\Squid(v)}(\Sigma_v)\cap \link_{\Squid(v)}(\Delta_v);$$
    \item If $v\in \ov\Phi$ then $\Pi_v= \Delta_v$;
    \item If $v\in \ov\Psi$ then $\Pi_v$ is the cone point $v$.
    \item If $v\in \ov\Theta$ then $\Pi_v$ is an edge.
\end{itemize}

\begin{figure}[htp]
\centering
\renewcommand{\arraystretch}{1.5}
\begin{tabular}{c|c|c}
$\Pi_v$&$\ov\Sigma$&$\link_{\ov X}(\ov\Sigma)$\\
\hline
$\ov\Delta$& extend $\Sigma_v$ if needed&$\Delta_v$, if $\ov\Phi=\{v\}$\\
$\link_{\ov X}(\ov\Delta)$&$\Sigma_v$& $v$, if $\ov\Psi=\{v\}$\\
$\ov X- (\operatorname{Star}_{\ov X}(\ov\Delta))$&complete $\Sigma_v$ to an edge& any edge in $\Squid(v)$, if $\ov\Theta=\{v\}$
\end{tabular}
    \caption{Schematic representation of the simplex $\Pi$. Each cell describes how $\Pi_v$ is defined whenever the vertex $v$ belongs to the area given by the intersection between the row label and the column label (for example, if $v\in\ov\Sigma\cap \link_{\ov X}(\ov\Delta)$ we have that $\Pi_v=\Sigma_v$).}
    \label{tab:Sigma_star_Pi}
\end{figure}

Moreover, if $\ov\Psi=\{u\}$ is non-empty, we set $\Psi=\{u\}$. Now one can check that $\link(\Sigma)\cap\link(\Delta)=\link(\Pi)\star\Psi$ (one can argue exactly as in \cite[Lemma 5.7, \textbf{Finding the extension of $\Sigma$}]{converse}).
\end{proof}

\begin{cor}[Verification of Definition~\ref{defn:combinatorial_HHS}.\eqref{item:chhs_flag}]\label{cor:finite_complexity_HHS_blow_up}
$X$ has complexity at most $25$.
\end{cor}

\begin{proof}
One can argue exactly as in the proof of \cite[Claim 6.9]{BHMS}, which only uses \cite[Condition 6.4.B]{BHMS} (our Lemma \ref{lem:intersection_links} here) and that $X$ has finite dimension. By inspection of the same proof, one also gets that the complexity is at most $(d+1)^2$, where $d$ is the dimension of $X$ (which is $4$ in our case).
\end{proof}

\begin{cor}[Verification of Definition~\ref{defn:combinatorial_HHS}.\eqref{item:chhs_join}]\label{lem:simplicial_wedge_property}
Let $\Sigma,\Delta$ be non-maximal simplices of $X$, and suppose that there exists a non-maximal simplex $\Gamma$ such that $[\Gamma]\nest[\Sigma]$, $[\Gamma]\nest[\Delta]$ and $\diam(\C (\Gamma))\ge 3$. Then there exists a non-maximal simplex $\Pi$ which extends $\Sigma$ such that $[\Pi]\nest[\Delta]$ and all $\Gamma$ as above satisfy $[\Gamma]\nest[\Pi]$.
\end{cor}

\begin{proof} 
One can argue as in the proof of \cite[Theorem 6.4]{BHMS} (more precisely, at the beginning of the paragraph named “\textbf{$(X,W)$ is a combinatorial HHS}”) to deduce the Corollary from Lemma~\ref{lem:intersection_links}.
\end{proof}

\subsubsection{Fullness of links}
\begin{lemma}[Verification of Definition~\ref{defn:combinatorial_HHS}.\eqref{item:C_0=C}]\label{lem:edges_in_link}
Let $\Delta\neq \emptyset$ be a simplex of $X$. Suppose that $a,b\in\link(\Delta)$ are distinct, non-adjacent vertices which are contained in $\W$--adjacent maximal simplices $\Sigma^a,\Sigma^b$. Then there exist $\W$--adjacent maximal simplices $\Pi^a,\Pi^b$ of $X$ such that $\Delta\star a\subseteq\Pi^a$ and $\Delta\star b\subseteq\Pi^b$.
\end{lemma}

\begin{proof}
    Suppose first that $p(a)=p(b)=w$ for some $w\in\ov{X}^{(0)}$. Then $a$ and $b$ must belong to the base $(L_w)^{(0)}$ of the cone under $w$, as they are non-adjacent. Moreover, since $\Sigma^a$ and $\Sigma^b$ are $\W$-adjacent, we must have that $\dist_{G}\left(N(a),N(b)\right)\le T+1$, regardless of the type of edge connecting $\Sigma^a$ and $\Sigma^b$. Now fix $v\in \link_{\ov X}(w)$ and $x\in (L_v)^{(0)}$, in such a way that $\Delta$ is contained in $\{(v,x),(w)\}$. Then the maximal simplices $\Pi^a=\Delta(x,a)$ and $\Pi^b=\Delta(x,b)$ contain $\Delta$ and are joined by a staple edge.
\par\medskip
    Now suppose that $p(a)=w\neq p(b)=w'$. In particular, $w$ and $w'$ are not $\ov{X}$-adjacent, or $a$ and $b$ would be joined by an edge of $X$. This forces $\ov\Delta=\{v\}$ to be a single vertex, such that $w,w'\in\link_{\ov X}(v)$. Let $y=\Sigma^a\cap (L_w)^{(0)}$, so that $a$ is either $y$ or $w$, and $y'=\Sigma^b\cap (L_{w'})^{(0)}$. Again, since $\Sigma^a$ and $\Sigma^b$ are $\W$-adjacent, we must have that $\dist_{G}\left(N(y),N(y')\right)\le T+1$, and we can complete $\Delta$ to two simplices $\Pi^a=\Delta(x,y)$ and $\Pi^b=\Delta(x,y)$, for some $x\in (L_v)^{(0)}$, which are joined by a staple edge.
\end{proof}

\subsubsection{Hyperbolicity of augmented links}\label{subsec_link_hyp}
Our next goal is to show that, for every non-maximal simplex $\Delta\subseteq X$, the augmented link $\C(\Delta)=\link(\Delta)^{+\W}$ is uniformly hyperbolic. If $\link(\Delta)$ has already diameter $2$ in $X$ then it is clearly $2$-hyperbolic. Thus we only have to focus on the cases when $\link(\Delta)$ is unbounded, which were described in Corollary~\ref{cor:bounded_links}.
\par\medskip
Let $\h G$ be the cone-off graph of $G$ with respect to the collection $\{\Stab{G}{v_i}\}$, as in Definition~\ref{defn:weak_hyp}, which is hyperbolic by Definition~\ref{defn:squid_material}.\eqref{squid_material:big_papa}.
\begin{lemma}[$\Delta=\emptyset$]\label{lem:CS_qi_to_squidoff}
    $X^{+\W}$ is $G$-equivariantly quasi-isometric to $\h G$.
\end{lemma}

\begin{proof}
    Define a coarse map $P_{p(\cdot)}\colon X^{+\W}\to \h G$, which, at the level of vertices, maps each $\Squid(v)$ to the product region $P_v$, which is bounded in $\h G$. This map is coarsely surjective, as $G$ is covered by the union of the $P_v$s.

    Next, we show that $P_{p(\cdot)}$ is Lipschitz, by showing that whenever $a,b\in (X^{+\W})^{(0)}$ are joined by an edge then $P_{p(a)}$ and $P_{p(b)}$ are uniformly close. If $a,b$ are joined by an edge of $X$ then either they belong to the same cone, and thus they both map to $P_{p(a)}$, or $p(a)$ and $p(b)$ are $\ov{X}$-adjacent, and therefore $\dist_G(P_{p(a)}, P_{p(b)})\le 2R$ by Corollary~\ref{cor:close_regions}. If instead $a,b$ belong to $\W$-adjacent maximal simplices $\Delta$ and $\Sigma$, then
    $$\dist_G(P_{p(a)}, P_{p(b)})\le 2R+\dist_G(N_R(P_{p(a)}), N_R(P_{p(b)}))\le 2R+\dist_G(f(\Delta), f(\Sigma)).$$
    But then $\dist_G(f(\Delta), f(\Sigma))$ is bounded above by the constant $\widetilde{K}$ from Lemma~\ref{lem:W-edge_close}.
\par\medskip
    Finally, in order to prove that $P$ is a quasi-isometry we are left to show that  $v$ and $v'$ are joined by a $\W$-edge whenever $\dist_G(P_v,P_{v'})\le 1$,. Let $g\in P_v$ and $g'\in P_{v'}$ be such that $\dist_G(g,g')\le 1$. By Corollary~\ref{cor:g_in_edge_prod_region}, there exist $w\in\link_{\ov X}(v)$ and $w'\in\link_{\ov X}(v')$ such that $g\in N_R(P_w)$ and $g'\in N_R(P_{w'})$, thus let $k\in P_w$ and $k'\in P_{w'}$ be $R$-close to $g$ and $g'$, respectively. Now, by construction $g$ belongs to $f(\Delta(\Proj_v(g),\Proj_w(k)))$, and similarly $g'\in f(\Delta(\Proj_{v'}(g'),\Proj_{w'}(k')))$. But then, since $\dist_G(g,g')\le 1$, we see that the simplices $\Delta(\Proj_v(g),\Proj_w(k))$ and $\Delta(\Proj_{v'}(g'),\Proj_{w'}(k'))$ have close realisations, and this implies that $v$ and $v'$ are $\W$-adjacent.
\end{proof}

Now, for every $v\in \ov X^{(0)}$, let $\h H_v$ be the cone-off graph of $H_v$ with respect to the finite collection $\{\p_v(Z_{hv'})\}$, as in Definition~\ref{defn:weak_hyp}. Such graph is hyperbolic by e.g \cite[Theorem 7.11]{Bowditch_relhyp}, which applies as all $\p_v(Z_{hv'})$ are quasiconvex subgroups of the hyperbolic group $H_v$.
\begin{lemma}[$\Delta$ of edge-type]\label{lem:hyp_link_edge_type}
    Let $\Delta=\{(v,x)\}$ be of edge-type. Then $\C(\Delta)$ is $\Stab{G}{v}$-equivariantly, uniformly quasi-isometric to $\h H_v$.
\end{lemma}

\begin{proof}
Since $G$ acts on both $X$ and $\W$ by isometries, we can assume without loss of generality that $v\in V$, so that $P_v=\Stab{G}{v}$. Furthermore, every $w\in \link_{\ov X}(v)$ is of the form $w=ghv'$, where $g\in \Stab{G}{v}$, $h\in G$, and $v'\in V$ are as in Notation~\ref{notation:pre_def_squid}.

Then we can define a coarse map $\theta_v\colon\C(\Delta)\to \h H_v$ which, at the level of vertices, maps the whole $\Squid(w)$ under $w=ghv'$ to the set 
$$\theta_v(w)=\bigcup_{g'} \p_v(g' Z_{hv'}),$$
where $g'$ varies among all elements in $\Stab{G}{v}$ such that $g'hv'=ghv'=w$. Notice that, by construction, $\theta_v$ is $\Stab{G}{v}$-equivariant.

We first point out that $\theta_v(w)$ is always uniformly bounded in $\h{H}_v$, and therefore $\theta_v$ is a well-defined coarse map. Indeed, if $g'\in \Stab{G}{v}$ is such that $w=ghv'=g'hv'$, we have that $g^{-1}g'\in\Stab{G}{v}\cap\Stab{G}{hv'}$, which by Definition~\ref{defn:squid_material}.\eqref{squid_material:edge_stab} virtually coincides with $\langle Z_v, Z_{hv'}\rangle$. Then $\p_v(g^{-1}g')$ lies in a finite index overgroup of $\p_v(Z_{hv'})$, which in turn means that the distance $$\dist_{\h H_v}(\p_v(g Z_{hv'}), \p_v(g' Z_{hv'}))\le \dist_{\h H_v}(\p_v(g), \p_v(g'))$$ is uniformly bounded. Furthermore, $\theta_v$ is coarsely surjective, as the cosets of the various $\p_v(Z_{hv'})$ cover $H_v$.

We now prove that $\theta_v$ is coarsely Lipschitz. Pick $a,b\in\link(\Delta)$ which are $\W$-adjacent, and let $w_a=p(a)$ and $w_b=p(b)$. We want to show that, if $w_a=g_a h_a v_a$ and $w_b=g_b h_b v_b$, then $\p_v(g_a Z_{h_av_a})$ is uniformly close to $\p_v(g_b Z_{h_bv_b})$ in $H_v$. By fullness of links, Lemma~\ref{lem:edges_in_link}, we can find two $\W$-adjacent maximal simplices $a\in \Sigma_a$ and $b\in\Sigma_b$ which extend $\Delta$. Then we have that
$$\dist_G(N_R(P_v)\cap N_R(P_{w_a}), N_R(P_v)\cap N_R(P_{w_b}))\le \dist_G(f(\Sigma_a),f(\Sigma_b))\le \widetilde{K},$$
where the second inequality is Lemma~\ref{lem:W-edge_close}. Now, by adapting the argument of Claim~\ref{claim:Pe_coarsely_square}, one sees that the Hausdorff distance between $N_R(P_v)\cap N_R(P_{w_a})$ and $g_a\langle Z_v, Z_{h_a v_a}\rangle$ is bounded by the constant $C_0$, and similarly if we replace $a$ by $b$. Thus
$$\dist_G(g_a\langle Z_v, Z_{h_a v_a}\rangle,g_b\langle Z_v, Z_{h_b v_b}\rangle)\le 2C_0+ \widetilde{K}.$$
Now, by Remark~\ref{rem:Pv_to_Lv}, there exists a constant $N\ge 0$, depending on $C_0$ and $\widetilde{K}$, such that 
$$\dist_{\Stab{G}{v}}(g_a\langle Z_v, Z_{h_a v_a}\rangle,g_b\langle Z_v, Z_{h_b v_b}\rangle)\le N,$$
where $\dist_{\Stab{G}{v}}$ is the word metric we previously fixed on $\Stab{G}{v}$. Then we can take the quotient projection to $H_v$, which is $1$-Lipschitz, and see that 
$$\dist_{H_v}(\p_v(g_a Z_{h_av_a}),\p_v(g_b Z_{h_bv_b}))\le N.$$

To conclude that $\theta_v$ is a quasi-isometry, we can almost read the above argument backwards. Indeed, suppose that $w_a=g_a h_a v_a$ and $w_b=g_b h_b v_b$ are such that $$\dist_{H_v}(\p_v(g_a Z_{h_av_a}),\p_v(g_b Z_{h_bv_b}))\le 1,$$
and we want to show that $w_a$ is $\W$-adjacent to $w_b$. Taking the preimages with respect to $\p_v$, we see that 
$$\dist_{\Stab{G}{v}}(g_a \langle Z_v, Z_{h_av_a}\rangle,g_b \langle Z_vZ_{h_bv_b}\rangle)\le 1.$$
Now, $\dist_G$ is bounded above by the intrinsic distance $\dist_{\Stab{G}{v}}$, as we chose the generating set for $G$ to contain a generating set for $\Stab{G}{v}$. Thus we get that 
$$\dist_G(g_a \langle Z_v, Z_{h_av_a}\rangle,g_b \langle Z_vZ_{h_bv_b}\rangle)\le 1,$$
and in turn, as $g_a \langle Z_v, Z_{h_av_a}\rangle\subseteq N_R(P_{w_a})$ and symmetrically for $b$, we get that $\dist_G(N_R(P_{w_a}), N_R(P_{w_b}))\le 1$. Then we see that there exists a staple edge between some simplex supported on $\{v,w_a\}$ and some simplex supported on $\{v,w_b\}$, so in particular $w_a$ and $w_b$ are $\W$-adjacent.
\end{proof}

\begin{lemma}[$\Delta$ of triangle-type]\label{lem:hyp_link_triangle_type} Let $\Delta=\{(v,x),(w)\}$ be a simplex of triangle-type. Then the identity map on $(L_w)^{(0)}$ is a $\Stab{G}{w}$-equivariant, uniform quasi-isometry $\lambda_w\colon L_w \to \C(\Delta)$.
\end{lemma}

\begin{proof}
First, we show that $\lambda_w$ is $1$-Lipschitz, that is, if $y,z\in (L_w)^{(0)}$ are adjacent in $L_w$ then they are connected by a $\W$-edge. Indeed, choose $g\in P_w$ such that $\Proj_w(g)=y$, and notice that $g$ belongs to both $N(y)$ and $N(z)$ as $\dist_{L_w}(\Proj_w(g), z)=\dist_{L_w}(y,z)\le 1\le R$. In other words, $N(y)$ intersects $N(z)$, which implies that $\Delta(x,y)$ and $\Delta(x,z)$ are connected by a staple edge.

Conversely, suppose that $y,z\in (L_w)^{(0)}$ are $\W$-adjacent, and we want to show that $y,z$ are uniformly close in $L_w$. By fullness of links, Lemma~\ref{lem:edges_in_link}, we can assume that the simplices $\Delta(x,y)$ and $\Delta(x,z)$ are $\W$-adjacent, and in particular $\dist_G(N(y),N(z))\le T+1$, regardless of the type of edge. Thus we can find elements  $g, h \in P_w$ such that $\Proj_w(g)$ and $\Proj_w(h)$ are $R$-close to $y$ and $z$, respectively, and $\dist_G(g,h)\le 2R+T+1$. Then Remark~\ref{rem:Pv_to_Lv} grants the existence of some constant $N\ge 0$, depending on $R$ and $T$, such that 
$$\dist_{L_w}(y,z)\le 2R+ \dist_{L_w}(\Proj_w(g),\Proj_w(h))\le 2R+N,$$
as required.
\end{proof}

\subsubsection{Quasi-isometric embeddings}\label{subsec:link_qi_emb}
The final axiom to check, in order to prove that $(X,\W)$ defines a combinatorial HHS, is that every augmented link $\C(\Delta)$ is quasi-isometrically embedded in $Y_\Delta$. Again, we look at all possible shapes of $\link(\Delta)$, according to Corollary~\ref{cor:bounded_links}. If $\link(\Delta)$ has diameter at most $2$, or if $\Delta=\emptyset$ so that $\C(\emptyset)=Y_\emptyset=X^{+\W}$, then the conclusion is trivial. Then there are two cases left to consider.

\begin{lemma}[$\Delta$ of edge-type]\label{lem:qi_embedding_edge_case}
    Let $\Delta=\{(v,x)\}$ be of edge-type. Then there is a coarsely Lipschitz, coarse retraction
    $$\varrho_{\link_{\ov X}(v)}\colon Y_\Delta\to \C(\Delta),$$
    whose constants are independent of $v$. In particular, $\C(\Delta)$ is quasi-isometrically embedded in $Y_\Delta$.
\end{lemma}

\begin{proof}
If $v$ has valence one in $\ov X$, that is, if $\link_{\ov X}(v)$ is a single vertex $w$, then $\link(\Delta)=\Squid(w)$ is uniformly bounded, and we have nothing to prove. Otherwise, as $\ov X$ is square-free, there is no $v'\neq v$ such that $\link_{\ov X}(v)=\link_{\ov X}(v')$. This implies that 
$$Y_{\Delta}=p^{-1}(\ov X-\{v\}).$$
Up to the action of $G$, assume that $v\in V$. Define $\varrho_{\link_{\ov X}(v)}$ by mapping the whole $\Squid(u)$ under $u\in\ov{X}^{(0)}-\{v\}$ to a vertex $w(u)\in\link_{\ov X}(v)$, chosen as follows:
\begin{itemize}
    \item If $\dist_{\ov X}(u,v)\ge 2$, choose any $w(u)$ such that $\gate_v(P_u)\subseteq N_R(P_{w(u)})$ (such a vertex exists by Definition~\ref{defn:squid_material}.\eqref{squid_material:gates}, combined with our choice of $R\ge B+r$);
    \item If instead $u\in\link_{\ov X}(v)$ set $w(u)=u$. Notice that, in this case as well, Definition~\ref{defn:squid_material}.\eqref{squid_material:gates} and our choice of $R$ give that $\gate_v(P_u)\subseteq N_R(P_{w(u)})$.
\end{itemize}

By construction, $\varrho_{\link_{\ov X}(v)}$ is a coarse retraction onto $\C(\Delta)$, so it is enough to prove that it is Lipschitz. Let $u,u'\in \ov{X}^{(0)}-\{v\}$ be adjacent in $Y_\Delta$, and we claim that $\dist_{\C(\Delta)}(w(u), w(u'))\le 2$. There are several cases to consider.

\begin{itemize}
    \item If $\dist_{\ov{X}}(u,u')\le 1$, then by Corollary~\ref{cor:close_regions} we have that $\dist_G(P_u, P_{u'})\le 2R$. But then $$\dist_G(P_v\cap N_R(P_{w(u)}), P_v\cap N_R(P_{w(u')}))\le \dist_G(\gate_v(P_u), \gate_v(P_{u'}))\le T,$$
where we used that we chose $T$ in such a way that, if two points in $G$ are $2R+1$-close, then their gates are $T$-close. In particular, there exists a staple edge between $w(u)$ and $w(u')$.
\item Suppose that $\dist_{\ov{X}}(u,u')\ge 2$ and they are joined by a $\W$-edge. If there exists $\Delta$ and $\Delta'$ with close realisations, whose supports contain $u$ and $u'$, respectively, then in particular $\dist_G(P_u, P_{u'})\le 2R+1$, and again this implies the existence of a staple edge between $w(u)$ and $w(u')$.
\item If instead $u$ and $u'$ are joined by a staple edge, then there exists  $z\in\ov{X}^{(0)}$ which is $\ov X$-adjacent to both $u$ and $u'$. If $z=v$ then $u,u'\in\link_{\ov X}(v)$, and therefore $w(u)=u$ and $w(u')=u'$ are already $\W$-adjacent. Otherwise, $w(z)$ is also well-defined, and by the first bullet we get that 
$$\dist_{\C(\Delta)}(w(u), w(u'))\le \dist_{\C(\Delta)}(w(u), w(z))+\dist_{\C(\Delta)}(w(z), w(u'))\le 2,$$
as required.\qedhere
\end{itemize}
\end{proof}

\begin{lemma}[$\Delta$ of triangle-type]\label{lem:qi_embedding_triangle_case}
    Let $\Delta=\{(v,x), (w)\}$ be of triangle-type. Then there is a coarsely Lipschitz retraction
    $$\varrho_w\colon Y_\Delta=\left(X-\left(\{w\}\cup p^{-1}(\link_{\ov X}(w))\right)\right)^{+\W}\to \C(\Delta)\cong L_w,$$
    whose constants are independent of $v$. In particular, $\C(\Delta)$ is quasi-isometrically embedded in $Y_\Delta$.
\end{lemma}

\begin{proof}
For every $y\in\link(\Delta)=(L_w)^{(0)}$, set $\varrho_w(y)=y$. Moreover, for every $a\in\Squid(u)$, where $u\in\ov{X}^{(0)}$ is such that $\dist_{\ov{X}}(w,u)\ge 2$, set $\varrho_w(a)=\Proj_w(\gate_w(P_u))$, which is uniformly bounded in $L_w$ as pointed out in Remark~\ref{rem:bounded_orbits_Lv}. By construction, the map $\varrho_w$ restricts to the identity on $\C(\Delta)$; so we have to prove that it is Lipschitz, by showing that if two vertices of $Y_\Delta$ are adjacent then their images under $\varrho_w$ are uniformly close. There are several cases to consider.
\begin{itemize}
    \item Suppose that $y\in (L_w)^{(0)}$ is adjacent to $a\in \Squid(u)$, for some $u$ as above. In particular, they are not adjacent in $X$, as $\dist_{\ov{X}}(w,u)\ge 2$, so they must be $\W$-adjacent. Regardless of the type of $\W$-edge, this implies that $N(y)$ is $(T+1)$-close to some point in $N_R(P_u)$. In other words, there exists $g\in P_w$ such that $\dist_{L_w}(\Proj_w(g), y)\le R$ and $\dist_G(g, P_u)\le 2R+T+1$. Then
    $$\dist_{L_v}(\varrho_w(y),\varrho_w(a))=\dist_{L_v}(y,\Proj_w(\gate_w(P_u)))\le $$
    $$\le R+\dist_{L_v}(\Proj_w(g),\Proj_w(\gate_w(g)))+\dist_{L_v}(\Proj_w(\gate_w(g)),\Proj_w(\gate_w(P_u))).$$
    Notice that $g$ is within uniform distance from $\gate_w(g)$ because it belongs to $P_w$, so $\dist_{L_v}(\Proj_w(g),\Proj_w(\gate_w(g)))$ uniformly bounded. Regarding the other term, it is enough to notice that $\dist_G(g, P_u)\le 2R+T+1$ and the composition $\Proj_w\circ \gate_w$ is coarsely Lipschitz.

\item Now suppose that $a\in \Squid(u)$ is adjacent in $Y_\Delta$ to $a'\in \Squid(u')$, for some $u,u'$ as above. Either by how $\W$-edges are defined or by Corollary~\ref{cor:close_regions} (depending on whether $a$ and $a'$ are connected by a $\W$-edge or an edge of $X$), this means that $\dist_G(P_u, P_{u'})\le 2R+T+1$, and again $\Proj_w(\gate_w(P_u))$ and $\Proj_w(\gate_w(P_{u'}))$ are uniformly close.\qedhere
\end{itemize}
\end{proof}

\subsubsection{Geometric action}
We need to prove that the action of $G$ on $(X,\W)$ endows $G$ with a combinatorial HHG structure, that is, it satisfies the “moreover” part of Theorem~\ref{thm:hhs_links}. By Remarks~\ref{rem:G_action_on_X} and~\ref{rem:G_action_on_W} we already know that $G$ acts on $X$  with finitely many $G$-orbits of links of simplices, and the action extends to a simplicial action on $\W$. Hence, we are left to prove that the action is geometric, which follows from the next lemma:

\begin{lemma}\label{lem:f_qi}
    The realisation map $f\colon\W\to G$ from Definition~\ref{defn:realisation} is a $G$-equivariant quasi-isometry.
\end{lemma}

\begin{proof}
    We already noticed that $f$ is $G$-equivariant, and in particular it is surjective as $G$ acts transitively on itself. Moreover, Lemma~\ref{lem:W-edge_close} can be rephrased by saying that $f$ is $\widetilde{K}$-Lipschitz. Conversely, if $\Delta, \Sigma \in \W$ are such that $\dist(f(\Delta),f(\Sigma))\le 1$, then $\Delta$ and $\Sigma$ have close realisations, and therefore are joined by a $\W$-edge.
\end{proof}

\subsubsection{Checking the short HHG axioms}\label{subsec:check_short_for_squid}
We are left to prove that the combinatorial HHG structure for $G$ is short:
\begin{lemma}
    $G$ admits a short HHG structure $(G, \ov X, \W)$, where:
    \begin{itemize}
        \item $\ov X$ is the support graph, from Definition~\ref{defn:squid_material}.\eqref{squid_material:graph};
        \item Definition~\ref{defn:squid_material}.\eqref{squid_material:extensions} describes vertex stabilisers as cyclic-by-hyperbolic extensions;
        \item For every $v\in\ov{X}^{(0)}$, $\C\ell_v$ is the graph $L_v$ from Definition~\ref{defn:quasilines_from_quasimorph_for_squid}.
    \end{itemize}
\end{lemma}

\begin{proof} We go through all axioms of a short HHG, as in Subsection~\ref{defn:short_HHG}. Axiom~\eqref{short_axiom:graph} is clear, as $X$ is a blowup of $\ov X$. Moreover, Axiom~\eqref{short_axiom:extension} is a direct consequence of Definition~\ref{defn:squid_material}.\eqref{squid_material:extensions}.

    Regarding Axiom~\eqref{short_axiom:cobounded}, Lemma~\ref{lem:hyp_link_triangle_type} tells us that every $\C\ell_v$ is uniformly and $\Stab{G}{v}$-equivariantly quasi-isometric to $L_v$. Then, the corresponding cyclic direction $Z_v$ acts geometrically on $L_v$ (by construction) and with orbits of diameter $1$ on $L_w$ for every $w\in\link_{\ov X}(v)$ (as pointed out in Remark~\ref{rem:bounded_orbits_Lv}). 
\end{proof}
The proof of Theorem~\ref{thm:squidification} is now complete.

\section{Blowup materials from short structures}\label{sec:short_to_squid}
Here we prove that short HHG and the class of groups admitting blowup materials actually coincide.
\begin{prop}
\label{prop:short_is_squid}
    A short HHG $(G, \ov X, \W)$ admits blowup materials, with support graph $\ov X$ and whose extensions are those from Axiom~\ref{short_axiom:extension}.
\end{prop}

\begin{proof}
We check that $G$ satisfies all points of Definition~\ref{defn:squid_material}. We shall introduce all the relevant data (such as quasimorphisms, gates, and so on) along the proof, and they will be $G$-equivariant by construction.

\par\medskip
\textbf{\eqref{squid_material:graph}} By Axiom~\eqref{short_axiom:graph}, $\ov X$ is triangle- and square-free, and none of its connected components is a point. Now, $G$ acts on $\ov X$, as the latter is a $G$-invariant subgraph of $X$. Furthermore, the action of $G$ on $X$ has finitely many $G$-orbits of links of simplices. This means, in particular, that there are finitely many orbits of edges of $\ov{X}$, as whenever $v$ and $w$ are $\ov{X}$-adjacent then $\{v,w\}=\link(\Delta)$ for any simplex of the form $\Delta=\{(x), (y)\}$, where $x\in (L_v)^{(0)}$ and $y\in (L_w)^{(0)}$.

\par\medskip
\textbf{\eqref{squid_material:extensions}} By Axiom~\eqref{short_axiom:extension}, vertex stabilisers are cyclic-by-hyperbolic extensions, and each cyclic direction acts trivially on the link of the corresponding vertex.

\par\medskip
\textbf{\eqref{squid_material:edge_stab}} Let $e=\{v,w\}$ be an edge of $\ov{X}$. By Lemma~\ref{lem:Hv_hyp_1} we have that $H_v$ is hyperbolic relative to $\{K_{H_v}(\p_v(Z_w))\}_{w\in W}$, for any collection $W$ of $\Stab{G}{v}$-orbit representatives of vertices in $\link_{\ov X}(v)$ with unbounded cyclic direction. In particular, for every $w'\in\link_{\ov X}(v)$, $\p_v(Z_{w'})$ is quasiconvex in $H_v$. This is either because $Z_{w'}$ is finite, or because $Z_{w'}$ is conjugate to $Z_w$ for some $w\in W$ and (conjugates of) peripheral subgroups in a relatively hyperbolic group are quasiconvex by e.g. \cite[Lemma 4.15]{drutu-sapir}. 

Now let $\Delta=\{(v), (w)\}$, seen as a simplex of $X$. Recall that, as explained in Lemma~\ref{lem:simplicial_cont_for_squid}, the product region $P_{[\Delta]}$ associated to the domain $[\Delta]$ is the subspace of maximal simplices of the form $\Delta(x,y)$, where $p(x)=v$ and $p(y)=w$. Furthermore, since $G$ has cobounded product regions by Lemma~\ref{lem:cobounded_stab_for_short}, we have that $P_{[\Delta]}$ is acted on geometrically by $\Stab{G}{e}$, and therefore by its index-two subgroup $\text{P}\Stab{G}{e}$. However, Axiom~\eqref{short_axiom:cobounded} tells us that $\langle Z_v, Z_w\rangle$, which is a subgroup of $\text{P}\Stab{G}{e}$, already acts coboundedly on $P_{[\Delta]}$, as each cyclic direction acts coboundedly on the corresponding $\ell$ and with uniformly bounded orbits on the other. This means that $\text{P}\Stab{G}{e}$ must virtually coincide with $\langle Z_v, Z_w\rangle$. 

\par\medskip
\textbf{\eqref{squid_material:big_papa}} By \cite[Remark 2.10]{hhs_asdim}, the space obtained from $G$ by coning off all proper product regions is quasi-isometric to the main coordinate space $\C S$, and is therefore hyperbolic. Furthermore, $G$ has cobounded product regions by Lemma~\ref{lem:cobounded_stab_for_short}, so its product regions coarsely coincide with the cosets of the $\Stab{G}{v_i}$.

\par\medskip
\textbf{\eqref{squid_material:quasimorphisms}} Whenever $Z_{v_i}$ is infinite, let $C_{v_i}$ be the centraliser of $Z_{v_i}$ in $\Stab{G}{v_i}$, and let $C'_{v_i}$ be the normal subgroup of $\Stab{G}{v_i}$ of all elements acting on the quasiline $\C\ell_{v_i}$ \emph{without inversions}, meaning that they do not swap the two points in the Gromov boundary. Let $E_{v_i}=C_{v_i}\cap C'_{v_i}$, which is again a normal subgroup of $\Stab{G}{v_i}$ and has index at most four. 
Then let $\phi_{v_i}\colon E_{v_i}\to\mathbb{R}$ be the \emph{Busemann quasimorphism} associated to the action, which is defined as follows. Fix a sequence $\{x_n\}_{n\in\mathbb{N}}\subseteq \C\ell_{v_i}$ converging to one of the point at infinity, and for every $g\in \Stab{G}{v_i}$ set
        $$m_{v_i}(g)\coloneq \limsup_{n\to+\infty}\left(\dist_{\ell_v}(g^k x_0, x_n)-\dist_{\ell_v}(x_0, x_n)\right).$$
        Then let $\phi_{v_i}$ be the homogeneous quasimorphism associated to $m_{v_i}$, as in Remark~\ref{rem:plasmon}. One can check that $\phi_{v_i}$ does not depend on the sequence $\{x_n\}_{n\in\mathbb{N}}$, and that an element has non-trivial image if and only if it acts loxodromically on the quasiline (see e.g. \cite[Section 4.1]{Manning_actions_on_hyp} for further details). In particular, $\phi_{v_i}$ is unbounded on $Z_{v_i}\cap E_{v_i}$, while it is trivial on $Z_w\cap E_{v_i}$ for every $w\in\link_{\ov X}({v_i})$. 
        
\par\medskip
\textbf{\eqref{squid_material:gates}}
We will implicitly identify $G$ and $\W$ by fixing once and for all a $G$-equivariant quasi-isometry $G\to \W$. Under this identification, there exists a collection of representatives $V$ of the $G$-orbits of vertices of $\ov{X}$ such that, for every $v=gv_i$ for some $g\in G$ and some $v_i\in V$, the product region $P_{\ell_v}\subseteq \W$ associated to the domain $\ell_v$, in the sense of Definition~\ref{defn:factor}, coarsely corresponds to the coset $g\Stab{G}{v_i}$, which is what we defined as $P_v$ in Definition~\ref{defn:squid_material}, and therefore with $g E_{v_i}$ since $E_{v_i}$ has finite index in $\Stab{G}{v_i}$. In particular, there exists a $G$-equivariant family of coarsely Lipschitz gate maps $\gate_v\colon G\to 2^{E_v}$. Moreover, Remark~\ref{rem:factors_are_hqc} gives an explicit description of the coordinates of $\gate_v(g)$ (up to a uniformly bounded error): for every $g\in G$ and every $W\in\frakS$, the projection of $\gate_v(g)$ to $\C W$ is the same as the projection of $g$ if $W=\ell_v$ or $W\nest\U_v$, and is set to $\rho^{\ell_v}_W$ otherwise.

We shall now check that $\gate_{v_i}$ satisfies all properties from Definition~\ref{defn:squid_material}.\eqref{squid_material:gates} by explicitly describing $\gate_{v_i}(P_u)$ for every $u\in\ov{X}^{(0)}$, depending on the distance $\dist_{\ov X}(u,{v_i})$. For short, we shall drop the index and simply denote $v_i$ by $v$. We also fix a collection of representatives of the $\Stab{G}{v}$-orbits of vertices in $\link_{\ov X}(v)$, and as in Notation~\ref{notation:pre_def_squid} we shall denote such a representative by $hv'$ where $h\in G$ and $v'\in V$.

\begin{itemize}
    \item If $\dist_{\ov X}(u,{v})=1$, up to the action of $G$ suppose that $u=hv'$. Then $\gate_{v}(P_u)$ is within finite Hausdorff distance from the intersection $\Stab{G}{v}\cap \Stab{G}{u}$ (this follows by combining \cite[Lemma 4.10]{quasiflats} and \cite[Lemma 4.5]{Hruska_Wise}), and $$\Stab{G}{u}=h\Stab{G}{u}h^{-1}\subseteq N_r(h\Stab{G}{u})=N_r(P_u)$$ where $r$ is chosen as in Notation~\ref{notation:pre_def_squid}. In other words, $\gate_{v}(P_u)$ is contained in a uniform neighbourhood of $P_u$.
    
    \item If $\dist_{\ov X}(u,{v})=2$ then there exists a unique $w\in\link_{\ov X}({v})\cap\link_{\ov X}(u)$, as $\ov X$ is square-free. Again, up to the action of $G$ we can assume that $w=hv'$, for some $h\in G$ and $v'\in V$ as in Notation~\ref{notation:pre_def_squid}. Moreover, $\gate_v(P_u)$ coarsely coincides with some parallel copy $F$ of the factor $F_{\ell_w}$, as $\ell_w$ is the only domain with unbounded coordinate space which is neither transverse to, nor contains, one between $\ell_{v}$ and $\ell_u$. In turn, $Z_{hv'}$ acts coboundedly on $F$, so $\gate_v(P_u)$ is in a neighbourhood of some coset of $Z_{hv'}$. 
    \item If $\dist_{\ov X}(u,{v})\ge 3$ then the coordinates of any $x\in \gate_{v}(P_u)$ are all prescribed (up to uniformly finite distance). Indeed, every domain in $P_{v}$ with unbounded coordinate space must be transverse to $\ell_u$, and taking the gate sets all other coordinates to the projection of $\ell_{v}$. Then the uniqueness axiom~\eqref{item:dfs_uniqueness} tells us that there exists a constant $B_0$, depending only on the HHS structure, such that $\gate_{v}(P_u)$ has diameter at most $B_0$. In particular, if we choose the constant $B$ to be greater than $B_0$, we can find some $w=ghv'\in \link_{ov X}(v_i)$ such that $\gate_{v_i}(P_u)\subseteq N_B(gZ_{hv'})$.
\end{itemize}
The proof of Proposition~\ref{prop:short_is_squid} is now complete.
\end{proof}

\section{Maquillage on blowup materials}\label{sec:maquillage}
Let $(G, \ov{X},\W)$ be a short HHG. We claim that, by tweaking the blowup materials from the proof of Proposition~\ref{prop:short_is_squid} and then invoking Theorem~\ref{thm:squidification}, one can introduce new cyclic directions, corresponding to elements acting loxodromically on the main curve graph $\C S$ or on some augmented link $\C \U_v$. This shall be extremely relevant in the companion paper, as then, when we are given a quotient of a short HHG, we can (almost always) assume that the kernel is normally generated by cyclic directions, of which we understand the action on the support graph.

\subsection{Adding ``globally'' loxodromic directions}
We first recall that, by \cite[Corollary 14.4]{HHS_I}, a normalised HHG $G$ acts acylindrically on the main coordinate space $\C S$. This means that, if an element $g\in G$ acts loxodromically, then it is also WPD. Then \cite[Lemma 6.5]{DGO} implies that such an element $g$ is contained in a maximal virtually cyclic subgroup $K(g)\le G$.

Furthermore, recall that two elements $g,h\in G$ are \emph{commensurable} if some non-zero powers of them are conjugate in $G$. The following Proposition shows how, given a short HHG, one can introduce new cyclic directions with “trivial” stabiliser, corresponding to non-commensurable loxodromic elements for the action on the main coordinate space.

\begin{prop}\label{prop:add_loxo}
    Let $(G, \ov{X},\W)$ be a short HHG, and let $g_1,\ldots, g_r\in G$ be non-commensurable loxodromic elements for the action on the main coordinate space $\C S$. For every $i=1,\ldots,r$ let $K(g_i)$ be the maximal virtually cyclic subgroup of $G$ containing $g_i$. There exists a short HHG structure $(G, \ov X',\W')$ where:
    \begin{itemize}
        \item The new support graph $\ov{X}'$ is obtained from $\ov{X}$ by adding a new connected component for each coset $\{hK(g_i)\}_{h\in G,\,i=1,\ldots, r}$, consisting of an edge $\{u^h_i, w^h_i\}$.
        \item For every $i=1,\ldots, r$ there exists $N_i\in\mathbb{N}_{>0}$ such that the cyclic direction associated to $u^h_i$ is $h\langle g_i^{N_i}\rangle h^{-1}$, while the cyclic direction associated to $w^h_i$ is trivial.
        \item The data of all other vertices are unchanged. 
    \end{itemize}
    Moreover, if $(G, \ov{X},\W)$ is colourable, then so is $(G, \ov X',\W')$.
\end{prop}

\begin{proof}
Firstly, as $G$ has cobounded product regions by Lemma~\ref{lem:cobounded_stab_for_short}, we can argue as in  Remark~\ref{rem:top_guy_HHG} to get that $\C S$ is ($G$-equivariantly quasi-isometric to) $\Cay{G}{T}$ for some possibly infinite generating set $T$, which we can choose in such a way that $T\cap K(g_1)$ generates $K(g_1)$. Moreover, by \cite[Theorem 6.8]{DGO} we have that $K(g_1)$ is hyperbolically embedded in $(G,T)$.  Combining the two facts, we see that $K(g_1)$ is also \emph{hierarchically hyperbolically embedded}, in the sense of \cite[Definition 6.2]{hhs_asdim} which was designed as the natural generalisation to HHG of the analogous property from \cite{DGO}. Now, a careful inspection of \cite[Proposition 6.14]{hhs_asdim} tells us that there exists a HHG structure $(G,\frakS')$ where:
\begin{itemize}
        \item The index set $\frakS'$ contains $\frakS$, together with one element for every coset $\{hK(g_1)\}_{h\in G}$;
        \item $\nest$ and $\orth$ are unchanged on $\frakS$, while every $hK(g_1)$ is nested in $S$ and transverse to every other domain;
        \item The main coordinate space is quasi-isometric to the cone-off space $\h{\C S}$, obtained from $\C S$ by adding a cone over each coset of $K(g_1)$;
        \item For every $W\in \frakS-\{S\}$, $\C W$ is unchanged, while $\C hK(g_1)=h\Cay{K(g_1)}{\mathcal I}$ for some fixed, finite generating set $\mathcal I$ for $K(g_1)$.
\end{itemize}
Now let $\ov X_1$ be the graph obtained from $\ov{X}$ by adding a new connected component for each coset $\{hK(g_1)\}_{h\in G}$, consisting of an edge $\{u^h_1, w^h_1\}$. Furthermore, we can extend the $G$-action to $\ov{X}_1$ by setting $gu^h_1=u^{gh}_1$ and $gw^h_1=w^{gh}_1$ for every $g,h\in G$. We now construct blowup materials with support graph $\ov X'$, so that Theorem~\ref{thm:squidification} will then grant the existence of a short HHG structure $(G, \ov X',\W')$ with the required properties. The proof will mimic that of Proposition~\ref{prop:short_is_squid}, and indeed we will extensively use that $(G, \ov X, \W)$ already admits blowup materials; therefore we only need to define the blowup materials associated to the new vertices, and check that they interact well with the original blowup materials.

\par\medskip
\textbf{\eqref{squid_material:graph}:} $\ov X_1$ is again triangle- and square-free, since it is obtained from $\ov X$ by adding some new connected components which are edges. As we added two $G$-orbits to the cocompact action on $\ov X$, $G$ still acts cocompactly on $\ov X_1$. Furthermore, any $G$-colouring for $\ov X$ can be extended to $\ov X_1$ by adding two new colours, one for every new $G$-orbit of vertices.
\par\medskip

\textbf{\eqref{squid_material:extensions}:} For every $v\in \ov{X}$, the extensions coming from $(G, \ov X, \W)$ already satisfy all requirements. Moving to $u^h_1$ and $w_h^1$, recall from e.g. \cite[Corollary 6.6]{DGO} that there exists $N_1\in\mathbb{N}_{>0}$ such that $\langle g_1^{N_1}\rangle $ is normal in $K(g_1)$. Then we define the cyclic direction for $u^h_1$ by setting $Z_{u^h_1}=h\langle g_1^{N_1}\rangle h^{-1}$. Notice that the quotient group $\Stab{G}{u^h_1}/Z_{u^h_1}=(hK(g_1)h^{-1})/(h\langle g_1^{N_1}\rangle h^{-1})$ is finite, hence hyperbolic, and that $Z_{u^h_1}$ fixes $w^h_i$ as it belongs to $hK(g_1)h^{-1}$. Regarding $w^h_i$, we can set $Z_{w^h_1}=\{0\}$ to be trivial, and notice that $\Stab{G}{w^h_1}=hK(g_1)h^{-1}$ is already hyperbolic.
\par\medskip

\textbf{\eqref{squid_material:edge_stab}:} The stabiliser of an edge of $\ov X$ is the same as before, so we just need to check that every edge of the form $\{u^h_i, w^h_i\}$ satisfies the properties. The only fact to notice is that $\Stab{G}{u^h_i}=\Stab{G}{w^h_i}=hK(g_1)h^{-1}$ contains $\langle Z_{u^h_1}, Z_{w^h_1}\rangle=h\langle g_1^{N_1}\rangle h^{-1}$ as a subgroup of finite index, and in particular $Z_{u^h_i}$ is quasiconvex in $hK(g_1)h^{-1}$.
\par\medskip

\textbf{\eqref{squid_material:big_papa}:} We already know that the main coordinate space $\C S$ for the original structure $(G,\frakS)$ is quasi-isometric to the cone-off Cayley graph of $G$, with respect to the cosets of a collection of representatives of  stabilisers for the action on $\ov X$. If we furthermore cone-off the cosets of $K(g_1)$ we get the top-level coordinate space for $(G,\frakS')$, which is hyperbolic.
\par\medskip

\textbf{\eqref{squid_material:quasimorphisms}:} Fix a collection $V$ of representatives of the $G$-orbits of vertices of $\ov X$, and let $u_1, w_1$ be the new vertices associated to the coset $K(g_1)$. For all $v\in V$, both $E_v$ and the associated quasimorphism can be defined as in Proposition~\ref{prop:short_is_squid}. Furthermore, let $E_{u_1}$ be the centraliser of $\langle g_1^{N_1}\rangle$ in $K(g_1)$. By e.g. \cite[Lemma 3.2]{Macpherson_virtcyclic}, there exists an epimorphism $\phi_{u_1}\colon E_{u_1}\to \Z$, which in particular is a homogeneous quasimorphism which is unbounded on $\langle g_1^{N_1}\rangle$.
\par\medskip

\textbf{\eqref{squid_material:gates}:}  Notice that, for every $v\in V$, the product region associated to $\ell_v$ in the new structure $(G,\frakS')$ coarsely coincides with the product region for the old structure, as the domains which are either nested into or orthogonal to $\ell_v$, as well as their coordinate spaces, are unchanged. Thus, as in Proposition~\ref{prop:short_is_squid}, the product region coarsely coincides with $E_v$, and we can define a gate map $\gate_v\colon G\to 2^{E_v}$. Furthermore, the product region associated to the domain $K(g_1)$, as well as its stabiliser, both coincide with $K(g_1)$ itself. Thus $K(g_1)$, and in turn $E_{u_1}$, are hierarchically quasiconvex, and we can set $\gate_{v_1}=\gate_{w_1}$ as the gate on $E_{u_1}$. Then, to see that all gates satisfy the requirements of Definition~\ref{defn:squid_material}.\eqref{squid_material:gates}, we can argue as in the proof of Proposition~\ref{prop:short_is_squid}, which only requires that, for every vertex $s\in\ov{X}_1$, there exists a $\nest$-minimal domain (here, either some $\ell_v$ or some $u^h_1$) whose associated product region coarsely coincides with $P_s$.

\par\medskip
The procedure described above gives a short HHG structure for $G$ in which some power of $g_1$ is a cyclic direction. We can repeat the process with $g_2$, then $g_3$ and so on, if we ensure that, for $i=2,\ldots, r$, $g_i$ still acts loxodromically on $\h{\C S}$. Notice that the action on $\C S$ is acylindrical along $K(g_1)$, as $g_1$ is WPD. Then e.g. \cite[Corollary 6.15]{abbott_manning} gives that either $g_i$ acts loxodromically on $\h{\C S}$, or some power of $g_i$ stabilises a coset of $K(g_1)$, and the latter cannot happen as then $g_1$ and $g_i$ would be commensurable. 
\end{proof}

\subsection{Adding ``locally'' loxodromic directions}
The following Proposition roughly says that, whenever $g\in \Stab{G}{v}$ acts loxodromically on $\C\U_v$ and has no hidden symmetries, in the sense of Definition~\ref{defn:hiddensymm_general}, we can tweak the short HHG structure in order to add $\langle g^n\rangle$ as a new cyclic direction, for some $n\in \mathbb N-\{0\}$:

\begin{prop}\label{prop:add_partial_loxo}
Let $(G, \ov{X},\W)$ be a short HHG. Let $v\in\ov{X}^{(0)}$ be a vertex with infinite cyclic direction, and let $g_1,\ldots,g_r\in \Stab{G}{v}$ be such that $\p_v(g_1), \ldots, \p_v(g_r)\in H_v$ are non-commensurable loxodromic elements for the action on $\C\U_{v}$. Suppose further that each $g_i$ has no hidden symmetries in $\Stab{G}{v}$.

For every $i=1\ldots, r$ there exists a non-trivial power $g_i'$ of $g_i$ and a short HHG structure $(G, \ov X',\W')$ such that:
\begin{itemize}
        \item The new support graph $\ov X'$ is obtained from $\ov{X}$ by adding a vertex $u_i^t$ for every coset $\{tS_{i}\}_{t\in G}$, where $S_i$ is the normaliser of $g_i'$ in $G$, and declaring that $u_i^t$ is only adjacent to $t v$. Furthermore, $\ov X'$ is again $G$-colourable.
        \item The cyclic direction associated to $u_i^t$ is $t\langle g_i'\rangle t^{-1}$;
        \item The data of all other vertices are unchanged.
    \end{itemize}
    Moreover, if $(G, \ov X,\W)$ is colourable, then so is $(G, \ov X',\W')$.
\end{prop}

\begin{proof}
Along the proof, we will often use the following observations. Firstly, if we fix a collection $W$ of $\Stab{G}{v}$-orbit representatives of vertices in $\link_{\ov X}(v)$ with unbounded cyclic direction, Lemma~\ref{lem:malnormality} states that $\{\p_v(Z_{w})\}_{w\in W}$ is an independent collection of cyclic subgroups. Furthermore, $H_v$ acts geometrically on $F_{U_v}$, by Claim~\ref{claim:quasiaction_geometric_H_V}, and each $Z_w$ acts geometrically on the corresponding quasiline; hence the cone-off graph $\h{H_v}$ of $H_v$ with respect to $\{K_{H_v}(\p_v(Z_{w}))\}_{w\in W}$ is quasi-isometric to $\C \U_v$.
\par\medskip

Now, for every $i=1,\ldots, r$ let $K_i$ be the maximal virtually cyclic subgroup of $H_v$ containing $\p_v(g_i)$. Since $g_i$ has no hidden symmetries, there exists a non-trivial power $g_i'$ of $g_i$ whose normaliser $S_i$ in $G$ contains $\p_v^{-1}(K_i)$, and in particular it contains $Z_v$. Notice that $S_i\le \Stab{G}{v}$. Indeed, if $t\in S_i$ then $g_i'$ fixes $t(v)$; moreover, the only vertex of $\ov X$ which is fixed by $g_i'$ is $v$, because for any $u\in\ov{X}-\{v\}$ the projection $\rho^{\ell_u}_{\U_v}\in\C\U_v$ is well-defined and cannot be fixed by the loxodromic action of $g_i'$. In turn, this implies that $S_i$ actually coincides with $\p_v^{-1}(K_i)$, since if $t$ (anti)commutes with $g_i'$ then $\p_v(t)$ (anti)commutes with some power of $\p_v(g_i)$, and therefore lies in $K_i$ by how the latter is described in \cite[Corollary 6.6]{DGO}.

Let $\ov X'$ be the graph obtained from $\ov X$ by adding a vertex $u_i^t$ for every coset $\{tS_{i}\}_{t\in G}$, which we declare to be only adjacent to $t v$. Extend the $G$-action to $\ov X'$ by setting $g u_i^t=u_i^{gt}$. We now check the existence of blowup materials with support graph $\ov X'$, and then again the conclusion will follow from Theorem~\ref{thm:squidification}.

\par\medskip
\textbf{\eqref{squid_material:graph}:} $\ov X'$ is again triangle- and square-free and none of its connected components are points, as we simply added some leaves to the $G$-orbit of $v$. Moreover, $G$ acts cocompactly on $\ov X'$, since we added a finite number of $G$-orbits. Finally, as no two vertices in the new orbits are adjacent, any $G$-colouring for $\ov X$ can be extended to $\ov X'$ by adding a single new colour.

\par\medskip
\textbf{\eqref{squid_material:extensions}:} 
For each $w\in\ov{X}^{(0)}$, we already know that $\Stab{G}{w}$ is a cyclic-by-hyperbolic extension satisfying all the properties of the blowup materials. The only thing to notice is that, whenever $t\in \Stab{G}{v}$, $Z_v$ fixes $u_i^t$ for every $i=1,\ldots, r$. Indeed,
$$Z_v t S_i= t Z_v S_i=t S_i,$$
where we used that $Z_v$ is normal in $\Stab{G}{v}$ and contained in $S_i$.

Moving to the new domains, for every $i=1,\ldots, k$ let $Z_i=\langle g_i'\rangle$, and then for every $t\in G$ let $Z_{u^t_i}=tZ_it^{-1}$. Let $u_i=u_i^1$ be the vertex associated to the coset $S_i$. Notice that $Z_i$ fixes $\{v\}=\link_{\ov X'}(u_i)$, as $Z_i\le\Stab{G}{v}$. Furthermore, by construction $\Stab{G}{u_i}=S_i$, and since $S_i=\p_v^{-1}(K_i)$ virtually coincides with $\langle Z_v, Z_i\rangle$, we have that the quotient $S_i/Z_{i}$ is virtually cyclic, hence hyperbolic.

\par\medskip
\textbf{\eqref{squid_material:edge_stab}:} Whenever $e$ is an edge of $\ov X$, $\text{P}\Stab{G}{e}$ already satisfies all requirements. Then up to the $G$-action let $e=\{v, u_i\}$. Now, $\Stab{G}{v}\cap\Stab{G}{u_i}=S_i$, and we already noticed that the latter is a finite-index overgroup of $\langle Z_v, Z_{i}\rangle$. Furthermore, we have that $Z_v\cap Z_{i}=\{0\}$ (the former acts trivially on $H_v$, the latter loxodromically), so $Z_v$ injects in the virtually cyclic quotient $S_i/Z_i$ and is therefore quasiconvex there. Conversely, $\p_v(Z_{i})$ is quasiconvex in $H_v$, as $\p_v(g_i)$ acts loxodromically on $\C \U_v=\h{H_v}$ and therefore on $H_v$.

\par\medskip
\textbf{\eqref{squid_material:big_papa}:} 
Fix a collection $V=\{v_1=v,\ldots, v_k, u_1, \ldots, u_r\}$ of orbit representatives of vertices in $\ov X'$. From the proof of Proposition~\ref{prop:short_is_squid} we know that $G$ is weakly hyperbolic relative to $\{\Stab{G}{v_1},\ldots, \Stab{G}{v_k}\}$; moreover,  we do not change the quasi-isometry type of the cone-off graph if we enlarge the collection to include $\{S_1,\ldots, S_r\}$, as $S_i\le \Stab{G}{v_1}$. 

\par\medskip
\textbf{\eqref{squid_material:quasimorphisms}:} 
For every $v_j$, $j=1,\ldots, k$, one can define $E_{v_j}$ as in the proof of Proposition~\ref{prop:short_is_squid}. If $j\ge 2$, the quasimorphism $\phi_{v_j}$ from the same proof satisfies all requirements already. We need a little more effort for $j=1$, as we have to produce a quasimorphism $E_{v}\to \mathbb{R}$ that vanishes on the new adjacent cyclic directions, as well as on the original ones. To do so, one first notices that the collection $$\mathcal H=\{\p_v(Z_1), \ldots, \p_v(Z_r)\}\cup\{\p_v(Z_{w})\}_{w\in W}$$ is independent in $H_v$. Indeed, no $\p_v(g_i')$ can be commensurable to any element in $Z_{w}$, since the former acts loxodromically on $\h{H_v}$ while the latter acts elliptically; moreover, by assumption no two loxodromic elements are commensurable, and we already pointed out that the collection $\{\p_v(Z_{w})\}_{w\in W}$ is independent. Furthermore, every $Z_i$ has no hidden symmetries in $\Stab{G}{v}$ by hypothesis, while every $Z_w$ for $w\in W$ has no hidden symmetries by Lemma~\ref{lem:nohidden}. Then Lemma~\ref{cor:finding_quasimorphisms} from the Appendix produces a quasimorphism $\Tilde{\phi}_v$ from the centraliser of $Z_v$ to $\R$ which is the identity on $Z_v$ and trivial on all adjacent cyclic directions, and we can then restrict $\Tilde{\phi}_v$ to $E_v$ to get the required quasimorphism.

Moving to the new domains, let $E_i=\langle 2Z_i, 2Z_v\rangle$, which is a normal subgroup of $S_i$ of index at most two and centralises $Z_i$. Then we can choose the quasimorphism $\phi_{u_i}\colon E_i\to \mathbb{R}$ as the projection onto $2Z_i$, which is clearly unbounded on $2Z_i$ and trivial on $2Z_v$. 

\par\medskip
\textbf{\eqref{squid_material:gates}:}
Before proving the existence of gates, we point out that, for every $i=1,\ldots, r$, $E_i\le \Stab{G}{v}$, and the latter coarsely coincides with the product region for $\ell_v$. This means that, in the original HHG structure, the projection of $E_i$ to any domain which is neither $\ell_v$ nor orthogonal to $\ell_v$ is uniformly bounded. Furthermore, we can describe the remaining projections of $E_i$ as follows:
\begin{itemize}
    \item $\pi_{\ell_v}(E_i)$ is coarsely dense, as it contains $Z_v$;
    \item $\pi_{\U_v}(E_i)$ coarsely coincides with $\pi_{\U_v}(\langle g_i'\rangle)$, as $Z_v$ acts trivially on $\C\U_v$. In particular, $\pi_{\U_v}(E_i)$ coarsely coincides with some quasi-axis $\gamma_i$ for $\p_v(g_i')$.
    \item For the same reason, for every $w\in\link_{\ov X}(v)$, $\pi_{\ell_w}(E_i)$ coarsely coincides with $\pi_{\ell_w}(\langle g_i'\rangle)$. We claim that this projection is uniformly bounded. Indeed, the bounded geodesic image axiom~\eqref{item:dfs:bounded_geodesic_image}, combined with the fact that $\gamma_i$ is a quasigeodesic in $\C\U_v$, imply that the “tails” of $\gamma_i$ have uniformly bounded projections to $\C\ell_w$. In other words, there exist $m,n\in\mathbb{Z}$, whose difference is bounded in norm by some constant $K$ depending on the quasigeodesic constants of $\gamma_i$, such that the projection of $\pi_{\ell_w}(\langle g_i'\rangle)$ coarsely coincides with the projection of $\{(g_i')^{l}\}_{l=m}^n$. Then up to the action of $\langle g_i'\rangle $ we can assume that $m=0$, and as coordinate projections in a HHG are uniformly coarsely Lipschitz we have that
    $$\sup_{n\in \Z\cap[-K,K]}\sup_{W\in\frakS}\dist_W(1, (g_i')^{n})<+\infty. $$
\end{itemize}

From this explicit description we get that $E_i$ is hierarchically quasiconvex, in the sense of Definition~\ref{defn:hqc}. Indeed, its coordinate projections are all quasiconvex; furthermore, whenever $x\in G$ projects close to $\pi_W(E_i)$ for every $W\in\frakS$, we can choose some $l\in\mathbb{Z}$ such that $\pi_{\U_v}(x)$ is within a bounded distance from $\pi_{\U_v}((g_i')^{l})$, and then choose $z\in Z_v$ such that $\pi_{\ell_v}(z(g_i')^{l})$ is uniformly close to $\pi_{\ell_v}(x)$ (notice that multiplying by an element of $Z_v$ does not change the projection to $\U_v$). Thus, there exists a coarsely Lipschitz gate map $\gate_{u_i}\coloneq G\to 2^{E_i}$.

Now, for every $v_j\in V$ define $\gate_{v_j}$ as the gate on $E_{v_j}$, which coarsely coincides with the product region $P_{\ell_{v_j}}$ in the original HHG structure. For every $u_1,\ldots, u_r$ define $\gate_{u_i}$ as above. We now check all properties of the gate $\gate_w$, depending on the type of $w$:
\par\medskip
\textbf{If $w=v$:} We know from before that $\gate_{v}$ is coarsely the identity on $P_{v}$, and therefore on each $P_{u^t_i}=tS_i$ for every $t\in\Stab{G}{v}$. In particular $\gate_{v}(P_{u_i^t}) $ is coarsely contained in $P_{u_i^t}$ for every $t\in \Stab{G}{v}$. Furthermore, if $u\in\link_{\ov X}(v)$ then we already know that $\gate_v(P_u)$ is coarsely contained in $P_u$, as proven in Proposition~\ref{prop:short_is_squid}. 

Moving to the second requirement, for every $u\in\ov{X}^{(0)}$ which is at distance at least $2$ from $v$, by the proof of Proposition~\ref{prop:short_is_squid} we already know that $\gate_v(P_{u})$ is coarsely contained in some coset of some cyclic direction adjacent to $v$. Thus, we only need to consider what happens when $\dist_{\ov X'}(v, u^t_i)\ge 2$, so that $\gate_{v}(P_{u_i^t})\subseteq \gate_{v}(P_{tv})$. If $\dist_{\ov X}(v,tv)\ge 2$ then we are done, as pointed out above. If instead $v$ and $tv$ are $\ov X$-adjacent, then $\gate_{v}(P_{tv})$ has unbounded projections only on $\C\ell_v$ and $\C\ell_{tv}$, and $P_{u_i^t}$ projects to a uniformly bounded subset inside $\C\ell_v$. Then $\gate_{v}(P_{u_i^t})$ coarsely coincides with the coset $tZ_{v}$.

\par\medskip
\textbf{If $w=v_j$, $j\ge 2$:} Recall that $\gate_{v_j}$ is defined as the gate on $E_{v_j}$ in the original HHG structure, hence all properties of $\gate_{v_j}(P_u)$ hold whenever $u\in\ov{X}^{(0)}$. Now fix $t\in G$ and $i\in \{1,\ldots, r\}$. We have that $\gate_{v_j}(P_{u^t_i})\subseteq \gate_{v_j}(P_{tv})$, so the only coordinate spaces to which $\gate_{v_j}(P_{u^t_i})$ might have unbounded projection are $\C\ell_{v_j}$ and $\C\ell_{tv}$. If $\dist_{\ov X}(tv, v_j)=1$ then $\pi_{\ell_{v_j}}(\gate_{v_j}(P_{u^t_i}))=\pi_{\ell_{v_j}}(P_{u^t_i})$ is uniformly bounded, by the above description of the projections of $E_i$; thus $\gate_{v_j}(P_{u^t_i})$ coarsely coincides with $tZ_{v}$, and all properties of the gate follow. If instead $\dist_{\ov X}(tv, v_i)\ge 2$ then we already know that $\gate_{v_j}(P_{tv})$ coarsely coincides with some coset of a cyclic direction adjacent to $v_j$, and again all properties of the gate follow.

\par\medskip
\textbf{If $w=u_i$:} Since $\gate_{u_i}$ takes image in $E_i\subseteq P_v$, for every $u\in\ov{X}'$ we have that $\gate_{u_i}(P_u)\subseteq P_v$. Now we explicitly describe the gate, depending on who is $u$.
\begin{itemize}
    \item Suppose first that $u=u^t_j$ for some $t\in\Stab{G}{v}$ and some $j\in\{1,\ldots, r\}$, so that $\dist_{\ov X'}(u_i, u_j^t)=2$. We know that $\pi_{\U_v}(E_i)$ (resp. $\pi_{\U_v}(E_{u^t_j})$) is coarsely a quasi-axis $\gamma_i$ (resp. $t\gamma_j$) for the action of $\p_v(g_i')$ (resp. $\p_v(tg_j't^{-1})$) on $\C \U_v$. We claim that $t\gamma_j$ has uniformly bounded projection to $\gamma_i$, so that the $\U_v$-coordinate of $\gate_{u_i}(E_{u^t_j})$ is uniformly bounded. This will imply that $\gate_{u_i}(E_{u^t_j})$ uniformly coarsely coincides with some coset of $Z_v$. 
    \\
    If $i=j$  we must have that $t\not\in S_i$, as $u_i$ does not coincide with $u_i^t$. In other words, $\p_v(t)\not \in K_i$, which means that $\p_v(g_i')$ and $\p_v(tg_i't^{-1})$ have no non-trivial common power. Then \cite[Lemma 6.7]{DGO} grants the existence of a constant $M\ge 0$, depending on the hyperbolicity constant of $\C\U_v$, such that the projection of $t\gamma_j$ onto $\gamma_i$ has diameter at most $M$. If instead $i\neq j$ then one can proceed as in the proof of \cite[Theorem 6.8]{DGO} (which, in turn, is ultimately an application of \cite[Lemma 6.7]{DGO}), to get the same conclusion.
    \item If $u\in\link_{\ov X}(v)$ then $\pi_{\U_v}(P_u)$ coarsely coincides with $\rho^{\ell_u}_{\U_v}$. Then again $\pi_{\U_v}(\gate_{u_i}(P_u))$, which coarsely coincides with the projection of $\pi_{\U_v}(P_u)$ onto $\pi_{\U_v}(E_i)$, is uniformly bounded, and we conclude as above.  
    \item Moving to those $u\in(\ov X')^{(0)}$ at distance at least $3$ from $u_i$, suppose at first that $u\in(\ov{X}-\operatorname{Star}(v))^{(0)}$. Then both $\pi_{\U_v}(P_u)$ and $\pi_{\ell_v}(P_u)$ are uniformly bounded. In turn, this means that $\gate_{u_i}(P_u)$ is uniformly bounded, as it has bounded projection to every coordinate space.
    \item Finally, suppose that $u=u^t_j$, where $t\in G-\Stab{G}{v}$ and $j\in\{1, \ldots, r\}$. Then $\gate_{u_i}(P_{u^t_j})\subseteq \gate_{u_i}(P_{tv})$, and the latter coarsely coincides with some coset of $Z_v$ by either the second or the third bullet above (depending on whether $tv$ belongs to $\link_{\ov X}(v)$ or not). 
\end{itemize}
The proof of Proposition~\ref{prop:add_partial_loxo} is now complete.
\end{proof}

In the setting of Proposition~\ref{prop:add_partial_loxo}, if one of the $g_i$ has hidden symmetries, we can still add a central direction generated by $g_i^nz_v^k$, for some $n\in\mathbb{N}-\{0\}$ and some $k\in \Z$, in view of the following Lemma:
\begin{lemma}[Straightening cyclic directions]\label{lem:straightening_central_direction}
Let $0\to \langle z\rangle\to G\xrightarrow[]{\p} H\to 1$ be a $\Z$-extension of a hyperbolic group, and let $g\in G$ map to an infinite order element in $H$. Then there exists $g'=g^nz^k$, for some $n\in\mathbb{N}-\{0\}$ and some $k\in \Z$, without hidden symmetries.
\end{lemma}

\begin{proof} 
    We restrict our consideration to the extension
    $$\begin{tikzcd}
        0\ar{r}&\langle z\rangle\ar{r}&S\ar{r}{\p}&K\ar{r}&0
    \end{tikzcd}$$
    where $K=K_{H}(g)$ is virtually cyclic and $S=\p^{-1}(K)$, and we want to find $g'$ as above which is normal in $S$. The proof boils down to basic linear algebra. Firstly, as $\langle\p(g)\rangle$ has finite index in $K$, there exists $N\in\mathbb{N}_{>0}$ such that $\langle \p(g^{2N})\rangle$ is normal in $K$. Thus, the centraliser $K^+$ of $\p(g^{2N})$ inside $K$ has index at most two. Let $C_1=\p^{-1}\left(K^+\right)$, and let $C_2$ be the centraliser of $z$ in $S$. Finally, let $E=C_1\cap C_2$, which is therefore a normal subgroup of index at most four. 

We first claim that every $t\in E$ commutes with $g^{2N}$. Indeed, since $t\in C_1$ we have that $tg^{2N}t^{-1}=g^{2N} z^K$ for some $K\in\Z$; furthermore, using that $t\in C_2$ must commute with $z$, we get that $t^rg^{2N}t^{-r}=g^{2N}z^{Kr}$ for every $r\in \Z$. However, $\p^{-1}(\langle \p(g^{2N})\rangle)=\langle z, g^{2N}\rangle$ has finite index in $S$, thus there exists some $r$ for which $t^r$ must commute with $g^{2N}$ (here we are using that $z$ commutes with any even power of $g$, as $C_2$ has index at most two in $S$). This yields that $K=0$.

Now, if $E=S$ we are done, as then $g^{2N}$ is central in $S$ and we can set $g'=g^{2N}$. If not, there are some cases to consider.

\begin{itemize}
    \item Suppose first that $E=C_1\lneq S$ but $C_2=S$, that is, $z$ is central in $S$, and pick $t\in S-E$. Since $\p(t)$ anticommutes with $\p(g^{2N})$, we must have that $tg^{2N}t^{-1}=g^{-2N}z^{M}$ for some $M\in\mathbb{Z}$. Then let $g'= g^{4N}z^{-M}$, which anticommutes with $t$. In turn, as $E$ has index two, $g'$ anticommutes with every $t'\in S-E$, so $\langle g'\rangle$ is normal in $S$.
    \item A similar argument works if $C_1= S$ but $E=C_2\lneq S$. Pick any $t\in S-E$, and suppose that $tg^{2N}t^{-1}=g^{2N}z^{M}$ for some $M\in\mathbb{Z}$. Then $g'= g^{4N}z^{M}$ commutes with $t$, and this time $g'$ is central in $S$.
    \item If $E=C_1=C_2\lneq S$, pick any $t\in S-E$, and suppose that $tg^{2N}t^{-1}=g^{-2N}z^{M}$ for some $M\in\mathbb{Z}$. Then $t^2g^{2N}t^{-2}=g^{2N}z^{-2M}$. But $t^2\in E$ must commute with $g^{2N}$, so we must have that $M=0$. This means that $t$ anticommutes with $g'=g^{2M}$, and therefore $\langle g'\rangle$ is normal in $S$.
    \item Finally, suppose that both $C_1$ and $C_2$ are proper, distinct subgroups of $S$. Let $t_1\in C_1-C_2$ and $t_2\in C_2-C_1$. As above, there exist $M_1,M_2\in\Z$ such that $t_1g^{2N}t_1^{-1}=g^{2N}z^{M_1}$ and $t_2g^{2N}t_2^{-1}=g^{-2N}z^{M_2}$. Then
    $$t_1t_2g^{2N}t_2^{-1}t_1^{-1}=t_1g^{-2N}z^{M_2}t_1^{-1}=g^{-2N}z^{-M1-M_2}.$$
    As $t_1t_2\in S-(C_1\cap C_2)$, we can argue as in the previous bullet to get that $M_1=-M_2$. Then set $g'=g^{4N}z^{M_1}$, and one can check that it commutes with $t_1$ and anticommutes with $t_2$ and $t_1t_2$. \qedhere
\end{itemize}
\end{proof}

\begin{rem}\label{rem:Z=K_central}
    With the notation from Lemma~\ref{lem:straightening_central_direction}, if the extension is central and $K=\Z$, then $S\cong \Z^2$, so every $g\in G$ already has no hidden symmetries.
\end{rem}

\section{New examples}\label{sec:new_examples}
In this Section, we use the machinery of blowup materials to provide new examples of short HHG.
\subsection{RAAGs on triangle- and square-free graphs}\label{subsec:new_raags}
\begin{prop}\label{prop:raag_short}
    Let $\Lambda$ be a finite, connected, triangle- and square-free simplicial graph with at least three points. The right-angled Artin group $G_\Lambda$ with defining graph $\Lambda$ is a colourable short HHG. 
\end{prop}

\begin{proof} Let $\ov X$ be the \emph{extension graph} from \cite{kim_koberda}, whose vertices are conjugates of standard generators, and where $a^g$ and $b^h$ are adjacent if and only if they commute. $G_\Lambda$ acts on $\ov X$ by conjugation, and we check that this action admits blowup materials.
\par\medskip
\textbf{\eqref{squid_material:graph}} $\ov X$ is triangle- and square-free, as so is $\Lambda$, and no connected component of $\ov X$ is a point. We also stress that, by e.g. \cite[Theorem 15]{kim_koberda_embeddability}, $[a^g,b^h]=1$ if and only if there exists $x\in G_\Lambda$ such that $a^x=a^g$ and $b^xb^h$, so the $G$-action is cofinite. Moreover, by e.g. \cite[Lemma 26.(8)]{kim_koberda_embeddability}, $\ov X$ is coloured by declaring $a^g$ and $b^h$ to have the same colour if and only if $a=b$.

\par\medskip
\textbf{\eqref{squid_material:extensions}} The stabiliser of $a$ is $C(a)=\langle a\rangle \times F$, where $F$ is the free group generated by $\link_{\Lambda}(a)$. In particular, $F$ is non-elementarily hyperbolic if $|\link_{\Lambda}(a)|$. Moreover, the conjugation by $a$ fixes $b$ for every $b\in\link_{\Lambda}(a)$.

\par\medskip
\textbf{\eqref{squid_material:edge_stab}} If $b\in\link_{\Lambda}(a)$ then $C(a)\cap C(b)=\langle a,b\rangle$, as $\Lambda$ is triangle-free, and $\langle b\rangle$ is quasiconvex in the free group $C(a)/\langle a\rangle$.

\par\medskip
\textbf{\eqref{squid_material:big_papa}} The cone-off graph of $G_\Lambda$ with respect to the family $\{C(a)\}_{a\in \Lambda^{(0)}}$ is clearly quasi-isometric to the extension graph, and the latter is hyperbolic by \cite{kim_koberda_embeddability}.

\par\medskip
\textbf{\eqref{squid_material:quasimorphisms}} Since $C(a)=\langle a\rangle \times F$, the projection $\phi\colon C(a)\to\langle a\rangle$ is a (quasi)morphism which is trivial on the link of $a$ and unbounded on $\langle a\rangle$.

\par\medskip
\textbf{\eqref{squid_material:gates}} Metrise $G$ by identifying with the universal cover of the Salvetti complex. By \cite[Proposition 8.3]{HHS_I}, for every $a\in \Lambda$ there exists a coarsely Lipschitz, coarse retraction $\gate_a\colon G_\Lambda\to C(a)$. Moreover, for every $g=a_1\ldots a_k$, where each $a_i$ is a standard generator, the Claim in the same proof describes $\gate_a(g C(b))$ as a $C(a)$-translate of the parabolic subgroup $G_{\Lambda_0}$, where 
$$\Lambda_0=\operatorname{Star}_{\Lambda} (a)\operatorname{Star}_{\Lambda}(b)\bigcap_{i=1}^k\link(a_i).$$ 
We now check that $\gate_a(g C(b))$ satisfies all requirements from Definition~\ref{defn:squid_material}.\eqref{squid_material:gates}.

\begin{itemize}
    \item If $\dist_{\ov X}(a, b^g)= 1$, then up to the $G_\Lambda$ action we can assume that $g=1$. Then $\gate_a(C(b))=C(a)\cap C(b)=\langle a,b\rangle \le C(b)$. 
    \item If $\dist_{\ov X}(a, b^g)\ge 2$ and $a\neq b$ then $\Lambda_0$ can only contain a single vertex $c$, as $\Lambda$ is triangle- and square-free. This means that $\gate_a(g C(b))$ is a translate of $\langle c\rangle$.
    \item Finally, if $\dist_{\ov X}(a, b^g)\ge 2$ and $a=b$, then $g\not\in C(a)$, so if we write $g=a_1\ldots a_k$ here must be some $a_i\not\in \operatorname{Star}_{\Lambda}(a)$. In particular 
    $\Lambda_0\subseteq \operatorname{Star}_{\Lambda}(a)\cap \link_\Lambda(a_i)$ can only contain a single vertex, and we conclude as above. 
\end{itemize}
One can now invoke Theorem~\ref{thm:squidification} to produce the required short HHG structure.
\end{proof}

\begin{rem}
    The above example is not ``new'', as it is possible to get the same short HHG structure ``from scratch'' by means of \cite[Theorem 3.15]{converse}. However, such an argument requires some subtleties which are not in the scope of this paper, so we preferred to derive this example from Theorem~\ref{thm:squidification}, as yet another showcase of its power.
\end{rem}

\subsection{Certain relative hyperbolic groups}\label{sec:hyp_rel_zbyhyp}
Let $G$ be a group which is hyperbolic relative to $\Z$-central extensions of hyperbolic groups. In the spirit of Section~\ref{sec:short_to_squid}, we now prove that $G$ is a short HHG, by extracting blowup materials from the relative hyperbolic structure. This will be a relevant tool in the companion paper, but its proof is presented here as it is similar to many arguments throughout this article.

It should be noted that, in view of \cite[Corollary 4.3]{HRSS_3manifold} and \cite[Theorem 9.1]{HHS_II}, such a $G$ already possesses a HHG structure, but our result upgrades it to a \emph{combinatorial} one, of which we can easily describe the underlying graph and the links of its simplices.  

\begin{prop}\label{prop:hyp_rel_Z-central_is_short}
    Let $G$ be hyperbolic relative to $\Z$-central extensions of hyperbolic groups. Then $G$ is a colourable short HHG, and in particular a combinatorial HHG.
\end{prop}

\begin{proof}
Let $\mathcal{P}$ be a collection of representatives of the peripheral subgroups, under the $G$-action by conjugation. By assumption, every $P\in \mathcal{P}$ is a central extension
$$\begin{tikzcd}
    0\ar{r}&Z_P\ar{r}&P\ar{r}{\p_P}&H\ar{r}&0,
\end{tikzcd}$$
where $Z_P$ is isomorphic to $\Z$ and $H$ is hyperbolic. If $H$ is infinite, let $k\in H$ be an element of infinite order (which exists as every torsion subgroup of a hyperbolic group is finite, see e.g. \cite[Corollary 8.36]{GroupesHyperboliques}), and let $K$ be the maximal virtually cyclic subgroup of $H$ containing $k$. If instead $H$ is finite we set $K=H$. In both cases, let $R_P=\p_P^{-1}(K)\le P$.

Now build a simplicial graph $\ov X$ as follows. The vertex set of $\ov X$ is
$$\bigsqcup_{P\in \mathcal P} \left(G/P\sqcup G/R_P\right).$$
For every $P\in \mathcal P$ and every $h\in G$, we declare $hP$ and $hR_P$ to be adjacent. We now check that the $G$-action on $\ov X$ admits blowup materials.

\par\medskip
\textbf{\eqref{squid_material:graph}:} By construction, the connected components of $\ov X$ are $G$-translates of the following star, where $P\in \mathcal P$:
$$\{P\}\star\{hR_p\}_{h\in P}.$$
In particular $\ov X$ is a forest, on which $G$ acts cocompactly. We can colour $\ov X$ by declaring every $hP$ to be black and every $hR_P$ to be white.

\par\medskip
\textbf{\eqref{squid_material:extensions}:} Notice that $\Stab{G}{P}=P$, which is a $\Z$-central extension of a hyperbolic group by assumption, and multiplication by $Z_P$ on the left fixes $hR_P$ whenever $h\in P$, because $Z_p$ is central in $P$ and contained in $R_P$. Moreover, $\Stab{G}{R_P}=R_P$, which is either virtually cyclic or virtually $\Z^2$. In the former case, we choose the cyclic direction $Z_{R_P}$ for ${R_P}$ to be trivial. In the latter case, by Lemma~\ref{lem:straightening_central_direction} there exists $r=\in \p_P^{-1}(\langle k\rangle)-Z_P$ such that $Z_{R_P}\coloneq \langle r\rangle$ is normal in ${R_P}$. Notice that the quotient ${R_P}/Z_{R_P}$ is virtually cyclic. 

\par\medskip
\textbf{\eqref{squid_material:edge_stab}:} $\Stab{G}{P}\cap \Stab{G}{{R_P}}={R_P}$, which virtually coincides with $\langle Z_P, Z_{R_P}\rangle$. Furthermore, $\p_P(Z_{R_P})$ is either trivial or (a subgroup of) $\langle k\rangle$, and in the latter case it is quasiconvex in $H$ as $k$ has infinite order. Similarly, if $\p_{R_P}\colon {R_P}\to {R_P}/Z_{R_P}$ is the quotient projection, then $\p_{R_P}(Z_P)$ has finite index in ${R_P}/Z_{R_P}$, and is therefore quasiconvex.

\par\medskip
\textbf{\eqref{squid_material:big_papa}:} Since $G$ is hyperbolic relative to $\mathcal P$, it is also weakly hyperbolic relative to $\mathcal P$, and therefore also to $\bigcup_{P\in \mathcal P}\{P,{R_P}\}$ as ${R_P}\le P$.

\par\medskip
\textbf{\eqref{squid_material:quasimorphisms}:} Since we chose $Z_{R_P}$ to have no hidden symmetries, Corollary~\ref{cor:finding_quasimorphisms} from the Appendix produces a homogeneous quasimorphism $\phi_P\colon P\to \mathbb{R}$ which is unbounded on $Z_P$ and trivial on every conjugate of $Z_{R_P}$.

Similarly, whenever $Z_{R_P}$ is non-trivial, let $E_{R_P}$ be a normal subgroup of ${R_P}$ centralising  $Z_{R_P}$ and isomorphic to $\Z^2$. 
Then it is easy to find a homomorphism $\phi_{R_P}\colon E_{R_P}\to \mathbb{R}$ which is unbounded on $Z_{R_P}\cap E_{R_P}$ and trivial on $Z_P\cap E_{R_P}$.

\par\medskip
\textbf{\eqref{squid_material:gates}:} $G$ is hyperbolic relative to $\mathcal P$, so by e.g. \cite[Theorem 2.14]{sisto_treegraded} for every $P\in \mathcal P$ there exists a coarsely Lipschitz, coarse retraction $\gate_{P}\colon G\to P$ which is uniformly bounded on $hP'$ whenever $h\in G$, $P'\in \mathcal P$, and $P\neq hP'$. 

Regarding the gate on $E_{R_P}$, we first recall that, by inspection of the proof of \cite[Corollary 4.3]{HRSS_3manifold}, each $P$ is quasi-isometric to a product $L\times H$, where $L$ is the quasiline coming from the quasimorphism $\phi_P$, as in Lemma~\ref{lem:abo}, while the projection on the second factor is the quotient projection $\p_P\colon P\to H$. Now, $E_{R_P}$ contains $Z_P\cap E_{R_P}$, on which $\phi_P$ is unbounded; hence the projection of $E_{R_P}$ to $L$ is coarsely dense. Moreover, $\p_P(E_{R_P})$, which has finite index in $\p_P({R_P})$, is quasiconvex in $H$. Thus there exists a coarsely Lipschitz, coarse retraction $\rho_{R_P}\colon P\to E_{R_P}$, which is just the product of the coarse retractions in the factors, and let $\gate_{R_P}\coloneq \rho_{R_P}\circ \gate_P$. 

We now check all properties of gates.
\begin{itemize}
    \item Starting with $\gate_P$, we know that $\gate_P$ coarsely coincides with the identity on $P$, and therefore on every $h{R_P}$ whenever $h\in P$. Furthermore, if $hP'\neq P$ for some $h\in G$ and $P'\in \mathcal P$, then $\gate_P(hP')$ is uniformly bounded, and \emph{a fortiori} so is $\gate_P(h{R_{P'}})\subseteq \gate_P(hP')$.
    \item Moving to $\gate_{R_P}$, we first notice that $\gate_{R_P}(P)\subseteq {R_P}\subseteq P$. If $hP'\neq P$ we already know that $\gate_P(hP')$ is uniformly bounded, and therefore so is $\gate_{R_P}(hP')=\rho_{R_P}\circ \gate_P(hP')$. Thus we only need to show that, if $h\in P-{R_P}$, then $\gate_{R_P}(h{R_P})$ coarsely coincides with $hZ_P$. To this purpose, first notice that, as ${R_P}=\p_P^{-1}(K)$, we have that $\p_P(h)\not \in K$. Then, as in the proof of Proposition~\ref{prop:add_partial_loxo}, \cite[Lemma 6.7]{DGO} implies that the projection of $\p_P(h) K$ onto $K$ is uniformly bounded. In other words, the second coordinate of $\gate_{R_P}(h{R_P})$ is uniformly bounded, so $\gate_{R_P}(h{R_P})$ must coarsely coincide with $hZ_P$.\qedhere
\end{itemize}
\end{proof}

\section{Application to coarse median structures}\label{sec:mcg_coarsemedian}
Let $(\cuco Z,\dist)$ be a metric space. Recall that a \emph{coarse median} $\mu\colon \cuco Z^3\to \cuco Z$ for $\cuco Z$ is a ternary operation satisfying the following:
\begin{itemize}
  \item \emph{Localisation}: $\mu(a,a,b)=a$ for all $a,b\in \cuco Z$;
  \item \emph{Symmetry}: $\mu(a_1,a_2,a_3)=\mu( a_{\sigma(1)},a_{\sigma(2)},a_{\sigma(3)})$ for all $a_1,a_2,a_3\in \cuco Z$ and permutation $\sigma$ of $\{1,2,3\}$;
 \item \emph{Affine control}: There exists a constant $R>0$ such that, for all $a,a',b,c\in \cuco Z$, we have
      $$d(\mu(a,b,c), \mu(a',b,c)) \leqslant Rd(a,a')+R;$$
  \item \emph{Coarse 4-point condition}: There exists a constant $\kappa>0$ such that for any $a,b,c,d\in \cuco Z$, we have
$$
\dist\left(\mu\left(\mu(a,b,c),b,d\right), \mu\left(a,b, \mu(c,b,d)\right)\right)\le \kappa.
$$
 \end{itemize}

We call two coarse median $\mu_1,\mu_2\colon \cuco Z^3\to \cuco Z$  \emph{equivalent} if $$\sup_{x,y,z\in \cuco Z} \dist_{\cuco Z}(\mu_1(x,y,z),\mu_2(x,y,z))<+\infty.$$
An equivalence class of coarse medians is called a \emph{coarse median structure}. 

By e.g. \cite[Theorem 7.3]{HHS_II} (which in turn builds on an observation of Bowditch \cite{bowditch_mcg}), a hierarchically hyperbolic space $(\cuco Z,\frakS)$ admits a coarse median, obtained as follows: for every $x,y,z\in \cuco Z$ and every $U\in\frakS$, the $U$-coordinate of $\mu(x,y,z)$ is the coarse centre of the triangle with vertices $\pi_U(x), \pi_U(y),\pi_U(z)$ in the hyperbolic space $\C U$. If moreover $(G,\frakS)$ is a HHG, $\mu$ is coarsely $G$-equivariant, by how $G$ acts on the coordinate spaces. The final goal of this paper is to show the following, which is Theorem~\ref{thmintro_coarsemedian} from the Introduction.

\begin{thm}\label{thm:coarse_median}
    Let $(G, \ov X, \W)$ be a short HHG, and suppose that for some vertex $v\in \ov X$ we have that $Z_v$ is infinite and $H_v$ is non-elementary. Then there are uncountably many short HHG structures on $G$ which give rise to pairwise non-equivalent, coarsely $G$-equivariant coarse medians.
\end{thm}

\begin{proof}
    By Proposition \ref{prop:short_is_squid} we have that $(G, \ov X, \W)$ admits blowup materials. To modify the short HHG structure we will modify the quasimorphism $\phi_v\colon E_v\to \mathbb{R}$ associated to $v$ as in the statement (we can assume that $v$ is the chosen representative in its orbit). Let $H'_v=\mathfrak p_v(E_v)$; this is a hyperbolic group. We can consider all cyclic subgroups of the form $\mathfrak p_v(Z_w\cap E_v)$ for $w\in \link_{\ov X}(v)$, of which there are finitely many conjugacy classes. 

    \begin{claim}
        Let $H$ be a non-elementary hyperbolic group, and let $h_1,\dots,h_n\in H$. Then there exists an unbounded homogeneous quasimorphism $\psi\colon H\to \mathbb R$ such that $\psi(h_i)=0$ for all $i$.
    \end{claim}

    \begin{claimproof}
        Without loss of generality, we can assume that each $h_i$ has infinite order, and that the collection $\langle h_1\rangle,\dots,\langle h_n\rangle$ is independent (indeed, if $\langle h_i\rangle\cup \langle gh_jg^{-1}\rangle\neq \{0\}$ for some $g\in H$ and some $i<j$, then every homogeneous quasimorphism which is trivial on $h_i$ must also be trivial on $h_j$). Now, in a non-elementarily hyperbolic group there are infinitely many commensurability classes of infinite order elements (see e.g. \cite[Lemmas 3.4 and 3.8]{olshanskii_residualizing_hom}, which prove a stronger result); thus we can find $h_0\in H$ such that the collection $\langle h_0\rangle,\dots,\langle h_n\rangle$ is again independent. Then, using independence, one can find a homogeneous quasimorphism $\psi\colon H\to \mathbb{R}$ which is non-trivial on $h_0$ and vanishes on $h_i$ for every $i\ge 1$ (for example one can combine \cite[Theorem 4.2]{Hull_Osin} with \cite[Corollary 6.6 and Theorem 6.8]{DGO}).
    \end{claimproof}

The Claim gives a homogeneous quasimorphism $\psi:H'_v\to \mathbb R$ with $\psi|_{\mathfrak p_v(Z_w\cap E_v)}=0$ for all $w\in \link_{\ov X}(v)$, and we set $\phi^\lambda_v=\phi_v+\lambda \psi\circ \mathfrak p_v$, which is again a homogeneous quasimorphism on $E_v$. If we now take the blowup materials for $G$ and replace $\phi_v$ with $\phi^\lambda_v$, without altering the other data (i.e. the support graph, the extensions, and the gate maps), we get new blowup materials, as Definition~\ref{defn:squid_material}.\eqref{squid_material:quasimorphisms} still holds by construction; thus we can invoke Theorem~\ref{thm:squidification} to produce a family of combinatorial HHG structures $(X,\W_\lambda)$ for $G$. We denote the quasiline associated to $\phi^\lambda_v$ by $L^{\lambda}_v$, and by $\mu_\lambda$ the coarse median arising from the HHG structure.
\par\medskip

We are left to show that, as we let $\lambda$ vary, the coarse medians $\mu_\lambda$ are all non-equivalent.  Fix $\lambda_1\neq \lambda_2$. Since $\psi$ is unbounded, there exists $g\in E_v$ with $\psi(\mathfrak p_v(g))\neq 0$. Let $z\in Z_v\cap E_v$ be any non-trivial element, and consider triples of the form $1,z^k, g^l$, for some integers $k$ and $l$. Let $m_i=\mu_{\lambda_i}(1,z^k, g^l)$, for $i=1,2$. 
 
We want to show that there exists $j\in \{1,2\}$ such that $\dist_{L^{\lambda_j}_v}(m_1, m_2)$ can be made arbitrarily large, by choosing appropriate values of $k$ and $l$. This will imply that $m_1$ and $m_2$ can be made arbitrarily far in $G$, because the projection from $\W_{\lambda_j}$ to  ${L^{\lambda_j}_v}$ is coarsely Lipschitz and $G$ acts geometrically on $\W_{\lambda_j}$.

To determine the projection of $m_i$ to the quasiline $L^{\lambda_j}_v$, we have to determine which of the numbers $\phi^{\lambda_j}_v(1),\phi^{\lambda_j}_v(z^k), \phi^{\lambda_j}_v(g^l)$ lies between the other two, because then the projection of $m_i$ will coarsely coincide with that of the corresponding $x_i\in\{1,z^k, g^l\}$. 
Since $\phi^{\lambda_1}_v(g)\neq \phi^{\lambda_2}_v(g)$, we can choose values of $k$ and $l$ such that $x_1\neq x_2$; thus, without loss of generality, $x_1$ is either $1$ or $z^k$, and in particular $\phi^{\lambda_1}_v(x_1)=\phi^{\lambda_2}_v(x_1)=\phi_v(x_1)$. Furthermore, the difference
$$|\phi^{\lambda_1}_v(x_1)-\phi^{\lambda_2}_v(x_2)|=|\phi^{\lambda_2}_v(x_1)-\phi^{\lambda_2}_v(x_2)|$$
can be made arbitrarily large. This proves that $\dist_{L^{\lambda_2}_v}(m_1, m_2)$ can be made arbitrarily large, as required. 
\end{proof}

By inspection of the short HHG structures in our example Sections~\ref{sec:example} and~\ref{sec:new_examples}, Theorem~\ref{thm:coarse_median} implies the following:
\begin{cor}\label{cor:coarsemedian_for_all}
    Let $G$ be either:
    \begin{itemize}
        \item the mapping class group of a sphere with five punctures;
        \item a RAAG on a connected, triangle- and square-free graph with at least three vertices;
        \item an Artin group of large and hyperbolic type, whose defining graph is not discrete;
        \item the fundamental group of an admissible graph of groups;
        \item an extension of a Veech group, in the sense of Subsection~\ref{subsec:example_veech};
        \item hyperbolic relative to $\Z$-central extensions of hyperbolic groups, of which at least one is non-elementary.
    \end{itemize}
    Then $G$ admits a continuum of coarsely $G$-equivariant coarse median structures.
\end{cor}

\appendix
\section{Quasicocycles from cyclic extension}\label{sec:quasicocycles}
In this appendix we prove that, given a $\Z$-central extension $0\to \Z\to G\to H\to 1$ of a hyperbolic group and a suitable collection of elements, one can produce a quasimorphism $G\to \R$ which is unbounded on $\Z$ but trivial on the collection. This is relevant for short HHG as, if one wants to introduce a new cyclic direction $\langle g\rangle$ which fixes some vertex $v$ of the support graph $ \ov X$, then one must find a quasimorphism on (a finite-index subgroup of) $\Stab{G}{v}$ which vanishes on $g$, in order to produce new blowup materials. In the process, we also identify a gap in a proof from \cite{ELTAG_HHS}, and explain how to circumvent it.

\begin{defn}
    Let $G$ be a group acting on $\Z$ by automorphisms, and let $\sigma\colon G\to \Aut{\Z}$ be the action. A \emph{1-quasicocycle} is a map $f\colon G\to \Z$ such that there exists a constant $D$, called the \emph{defect} of $f$, such that, for every $g,h\in G$,
    $$|f(g)+\sigma(g)(f(h))-f(gh)|\le D.$$
\end{defn}

We also recall the following definition from the main body of the paper:
\begin{defn}[No hidden symmetries]
    Let $0\to \Z\to G\xrightarrow[]{\pi} H\to 1$ be a group extension, with $H$ hyperbolic, and let $C\le G$ be a cyclic subgroup. We say that $C$ has \emph{no hidden symmetries} if $\pi(C)$ is infinite and $C$ contains a finite-index subgroup which is normal in $\pi^{-1}(K_H(\pi(C)))$. We say that an element $g\in G$ has no hidden symmetries if $\langle g\rangle$ has no hidden symmetries.
\end{defn}

\begin{lemma}\label{lem:finding_quasicocycles}
    Let $0\to \Z\to G\xrightarrow[]{\pi} H\to 1$ be a group extension, with $H$ hyperbolic. Let $\{C_i\}_{i=1,\ldots,n}$ be a finite collection of infinite cyclic subgroups of $G$ without hidden symmetries, whose projections $\{\pi(C_i)\}_{i=1,\ldots,n}$ form an independent collection. Then there exists a 1-quasicocycle $\phi\colon G\to \Z$, with respect to the action by conjugation, which is the identity on $\Z$ and is bounded on $C_i $ for every $i=1,\ldots,n$.
\end{lemma}

\begin{proof} 
We first recall a few facts about extensions. For every $g\in G$ and $z\in \Z$ let $z^g=gzg^{-1}$; this $G$-action descends to $H$, and with a little abuse of notation we set $z^\pi(g)=gzg^{-1}$. Next, the extension is represented by a \emph{2-cocycle} $\omega\colon H\times H\to \Z$, that is, a map which satisfies the cocycle relation
$$\omega(h_0,h_1h_2) = \omega(h_0, h_1)^{h_2} + \omega(h_0h_1, h_2) - \omega(h_1, h_2).$$
Up to isomorphism of short exact sequences, we can write $G$ as the group with underlying set $\Z\times H$, and with operation given by
$$(k_1,h_1)(k_2,h_2)=(k_1+k_2^{h_1}+\omega(h_1,h_2), h_1h_2).$$
Furthermore, since $H$ is hyperbolic, by \cite[Theorem 3.1]{neumannreeves} the cocycle can be assumed to be \emph{bounded}, meaning that the image of $\omega$ is finite. Hence, the map $\phi_0\colon G\to \Z$ given by projecting onto the first factor is a 1-quasicocycle which is the identity on $\Z=\Z\times\{1\}$.

Our new goal is to modify $\phi_0$ to get another 1-quasicocycle $\phi$ which is again the identity on $\Z$ and is bounded on $C_i$ for every $i=1,\ldots,n$. We proceed by induction on $n$, the base case $n=0$ being trivial. Hence suppose there exists a 1-quasicocycle $\phi'$ which is the identity on $\Z$ and is bounded on every $C_i$ for $i\le n-1$. Let $C=C_n$, and let $K$ be the maximal virtually cyclic subgroup of $H$ containing $\pi(C)$. As a finite-index subgroup $K'$ of $\pi(C)$ is normal in $K$, there exists a natural action $\sigma\colon K\to \Aut{\Z}=\{\pm \Id\}$ by automorphism, mapping $k\in K$ to $\Id$ if and only if it commutes with $K'$. There are three cases to consider, according to whether $\sigma$ coincides with the action induced by $H$. 
\par\medskip
    \textbf{Case 1:} Suppose first that there exists $k\in K$ such that $\sigma(k)=\Id$, but $z^k=-z$ for every $z\in \Z$, and we want to prove that $\phi'$ is already bounded on $C$. Let $h\in \pi^{-1}(k)$. Since $C$ has no hidden symmetries, there exists $g\in C$ such that $\langle g\rangle$ is normalised by $h$, and indeed $g$ must commute with $h$ since $k$ commutes with $K'$. Given any $n\in \Z$, we use the notation $O(1)$ to denote a quantity which is bounded independently on $n$. Then we have that
    $$\phi'(g^n)=\phi'(hg^nh^{-1})=\phi'(h) -\phi'(g^n)-\phi(h^{-1})+O(1)=-\phi'(g^n)+O(1).$$
    This shows that $\phi'$ is bounded on $\langle g\rangle$, and therefore on its finite-index supergroup $C$.
    \par\medskip
    \textbf{Case 2:} Similarly, suppose that, for some $k\in K$, $\sigma(k)=-\Id$, but $k$ centralises $\Z$. As above, we can find $g\in C$ which anticommutes with some $h\in \pi^{-1}(k)$; up to passing to the square power of $g$, we can assume that $g$ centralises $\Z$. Then again 
    $$\phi'(g^n)=\phi'(hg^{-n}h^{-1})=\phi'(h) +\phi'(g^{-n})+\phi(h^{-1})+O(1)=\phi'(g^{-n})+O(1).$$
    Furthermore, since $g$ commutes with $\Z$,
    $$\phi'(1)=\phi'(g^{n}g^{-n})=\phi'(g^{n})+\phi'(g^{-n})+O(1),$$
    and the two properties combine to give that $\phi'(\langle g\rangle)$ is bounded.
    \par\medskip
    \textbf{Case 3:} We can therefore assume that $\sigma$ coincides with the action induced by $H$. Let $\psi_0\colon K\to \R$ be the 1-quasicocycle associated to the extension $K'\to K\to K/K'$, which is the identity on $K'$. By \cite[Theorem 3.1]{Hull_Osin}, we can extend $\psi_0$ to a quasicocycle $\psi\colon H\to \R$ which is at finite distance from the identity on $K'$ while it is bounded on $\pi(C_i)$ for all $1\le i\le n-1$. Up to rescaling $\psi$, assume that $\psi(\pi(g))=1$, where $g\in C$ commutes with $\Z$ and is normalised by $\pi^{-1}(K)$.  Then let $\theta=\psi\circ \pi\colon G\to \R$, and set 
    $$\phi=\phi'-\phi'(g)\theta.$$
    By construction, $\phi$ is still the identity on $\Z$, and it is bounded on $C_i$ for every $i\le n-1$. Moreover, since $\theta(g)=1$, $\phi$ is now bounded on $C$ as well, as required.
\end{proof}

\begin{cor}\label{cor:finding_quasimorphisms}
    Let $0\to \Z\to G\xrightarrow[]{\pi} H\to 1$ be a group extension, with $H$ hyperbolic. Let $\{C_i\}_{i=1,\ldots, n}$ be a finite collection of infinite cyclic subgroups of $G$ without hidden symmetries, whose projections $\{\pi(C_i)\}_{i=1,\ldots, n}$ form an independent collection. Let $E$ be the centraliser of $\Z$ in $G$. Then there exists a homogeneous quasimorphism $\psi\colon E\to \Z$ which is the identity on $\Z$ and is trivial on $gC_i g^{-1}\cap E$ for every ${i=1,\ldots, n}$ and $g\in G$.
\end{cor}

\begin{proof}
    Let $\phi\colon G\to \Z$ be the 1-quasicocycle from Lemma~\ref{lem:finding_quasicocycles}, which is the identity on $\Z$ and is bounded on every $C_i$. Furthermore, for every ${i=1,\ldots, n}$ and $g\in G$, $\phi$ is also bounded on $gC_i g^{-1}$, because
    $$\sup_{c\in C_i}|\phi(gcg^{-1})|\le 2D+|\phi(g)|+|\phi(g^{-1})|+ \sup_{c\in C_i}|\phi(c)|.$$
    Now let $\phi|_E$ be the restriction of $\phi$ to $E$, which is a quasimorphism because $E$ acts trivially on $\Z$ by conjugation. Then the homogeneous quasimorphism associated to $\phi_E$ is the identity on $\Z$ and trivial on every $gC_i g^{-1}\cap E$, as required.
\end{proof}

\begin{rem}[Having no hidden symmetries is necessary]\label{rem:error_in_ELTAG}
    In \cite[Lemma 4.4]{ELTAG_HHS}, Hagen, Martin, and Sisto claimed that, given a central extension of a hyperbolic group $0\to \Z\to G\to H\to 1$ and a family of cyclic subgroups $C_1,\ldots,C_n$ whose projections to $H$ are independent, there exists a homogeneous quasimorphism which is unbounded on $\Z$ and trivial on every $C_i$. However, this is not true even if there is a single cyclic subgroup, as the following counterexample shows. Let $G=\langle a,t\mid tat^{-1}=a^{-1}\rangle$ be the Klein bottle group, which is a central extension $0\to \langle t^2\rangle \to G\to G/\langle t^2\rangle\to 1$. Since $G$ is an index-two overgroup of $\Z^2$, it is amenable, so every homogeneous quasimorphism is actually a homomorphism (see e.g. \cite[Corollary 3.8]{frigerio}). Now let $g=at^2$, and notice that $G=\langle g,t\mid tgt^{-1}=g^{-1}t^4\rangle$. From this presentation we see that any homomorphism $G\to \Z$ which is trivial on $C=\langle g\rangle$ must also be trivial on $\langle t^2\rangle$. 

    The problem in the above example is that $g$ has a hidden symmetry. Indeed, if one further assumes that all $C_i$ have no hidden symmetries, then the proof of \cite[Lemma 4.4]{ELTAG_HHS} can be carried on with minor corrections, which is exactly what we did in the slightly more general Lemma~\ref{lem:finding_quasicocycles}. 
    
    We stress that \cite[Lemma 4.4]{ELTAG_HHS} is only ever used in the proof of \cite[Lemma 4.5]{ELTAG_HHS}, where it can be replaced by our Corollary~\ref{cor:finding_quasimorphisms}. Indeed, every central extension appearing in that paper is of either of the following two forms:
    \begin{itemize}
        \item $\Z\times F$, where $F$ is a free group. Since every virtually cyclic subgroup of a free group is cyclic, every cyclic subgroup of $\Z\times F$ has no hidden symmetries, as argued in Remark~\ref{rem:Z=K_central}.
        \item a Dihedral Artin group $D$. In this case, the cyclic subgroups are conjugates of the standard generators $a$ and $b$, and we claim that these elements have no hidden symmetries. Indeed, if $h\in D$, $n\in \mathbb{N}-\{0\}$, and $z$ in the centre are such that $ha^nh^{-1}=a^{\pm n}z$, then, arguing as in the proof of  \cite[Lemma 2.8]{ELTAG_HHS}, we get that $z=0$, because otherwise the syllabic length of the powers of $a^{\pm n}z$ would diverge while the syllabic length of the powers of $ha^nh^{-1}$ is constant. 
    \end{itemize}
    Hence all results from \cite{ELTAG_HHS} still hold.
\end{rem}

\bibliography{biblio}
\bibliographystyle{alpha}

\end{document}